\documentclass[11pt,english]{amsart}

\usepackage[T1]{fontenc}
\usepackage[english]{babel}

\usepackage{hyperref,amssymb,url,xspace,mathtools,mathrsfs,stmaryrd}

\usepackage{geometry}
\usepackage{tikz-cd}
\usetikzlibrary{decorations.pathmorphing}
\usepackage{tikz}
\usetikzlibrary{positioning}

\usepackage{amsthm}
\usepackage{thmtools}
\usepackage{cleveref}

\newtheoremstyle{ThmSpace}{10pt}{13pt}{\itshape}{}{\bfseries}{:}{0.5em}{}
\newtheoremstyle{DefSpace}{10pt}{13pt}{}{}{\bfseries}{:}{0.5em}{}
\newtheoremstyle{DefBreak}{10pt}{13pt}{}{}{\bfseries}{:}{\newline}{}

\theoremstyle{ThmSpace}
\newtheorem{thm}{Theorem}[section]
\newtheorem{prop}[thm]{Proposition}
\newtheorem{cor}[thm]{Corollary}
\newtheorem{fact}[thm]{Fact}

\newtheorem{lem}[thm]{Lemma}

\theoremstyle{DefSpace}
\newtheorem{dfn}[thm]{Definition}
\newtheorem{ex}[thm]{Example}
\newtheorem{rk}[thm]{Remark}

\theoremstyle{DefBreak}
\newtheorem{dfnBreak}[thm]{Definition}

\newcommand{\mbb}{\mathbb}
\newcommand{\mbf}{\mathbf}
\newcommand{\mc}{\mathcal}
\newcommand{\mf}{\mathfrak}
\newcommand{\mrm}{\mathrm}
\newcommand{\mscr}{\mathscr}

\renewcommand{\u}{\underline}

\newcommand{\A}{\mbb{A}}

\newcommand{\C}{\mbb{C}}

\newcommand{\F}{\mbb{F}}

\renewcommand{\L}{\mbb{L}}

\renewcommand{\P}{\mbb{P}}
\newcommand{\Q}{\mbb{Q}}
\newcommand{\R}{\mbb{R}}

\newcommand{\Z}{\mbb{Z}}

\newcommand{\BB}{\mbf{B}}

\newcommand{\Ccal}{\mc{C}}

\newcommand{\Ecal}{\mc{E}}
\newcommand{\Fcal}{\mc{F}}
\newcommand{\Gcal}{\mc{G}}
\newcommand{\Hcal}{\mc{H}}

\newcommand{\Ocal}{\mc{O}}
\newcommand{\Pcal}{\mc{P}}

\newcommand{\Scal}{\mc{S}}

\newcommand{\Ucal}{\mc{U}}

\newcommand{\Xcal}{\mc{X}}
\newcommand{\Ycal}{\mc{Y}}

\newcommand{\Ascr}{\mscr{A}}
\newcommand{\Cscr}{\mscr{C}}
\newcommand{\Lscr}{\mscr{L}}
\newcommand{\Mscr}{\mscr{M}}
\newcommand{\Sscr}{\mscr{S}}

\newcommand{\Xscr}{\mscr{X}}
\newcommand{\Zscr}{\mscr{Z}}

\newcommand{\ad}{\mrm{ad}}
\newcommand{\Aut}{\mrm{Aut}}
\newcommand{\can}{\mrm{can}}
\newcommand{\FS}{\mbf{FS}}

\newcommand{\GL}{\mrm{GL}}
\newcommand{\GSp}{\mrm{GSp}}
\newcommand{\GSpin}{\mrm{GSpin}}
\newcommand{\Hom}{\mrm{Hom}}
\newcommand{\id}{\mrm{id}}
\newcommand{\im}{\mrm{im}}
\newcommand{\Ind}{\mrm{Ind}}
\newcommand{\Isom}{\mrm{Isom}}
\newcommand{\Lev}{\mrm{Lev}}
\newcommand{\Loc}{\mbf{Loc}}
\newcommand{\Sh}{\mrm{Sh}}
\newcommand{\sh}{\mrm{sh}}

\newcommand{\SO}{\mrm{SO}}

\newcommand{\Spec}{\mrm{Spec}}
\newcommand{\Triv}{\mrm{Triv}}
\newcommand{\uIsom}{\u{\Isom}}

\renewcommand{\l}{\ell}
\newcommand{\Ql}{\Q_{\ell}}
\newcommand{\Qp}{\Q_p}
\newcommand{\Zl}{\Z_{\ell}}
\newcommand{\Zp}{\Z_p}

\title{Pointwise to global good reduction problems}
\author{François \textsc{Gatine}}
\date{\today}
\email{francois.gatine@imj-prg.fr}

\begin{document}

\begin{abstract}
    Given a smooth proper family of varieties over some $p$-adic field, if (sufficiently many of) the closed fibers have good reduction, does the family as a whole have good reduction? This question is asked and positively answered in recent work of Cadoret and Tamagawa for curves in general, and abelian varieties if the base field has absolute ramification at most $p-1$. In this paper we define a general framework for studying this question, yielding new positive answers for families of K3 surfaces and cubic fourfoulds, again with ramification at most $p-1$. Our strategy is to study several extension properties of moduli spaces and how they transfer under various morphisms. Of particular importance are integral canonical models of Shimura varieties, and whether or not their extension properties are preserved under ramified scalar extension. Additionally, building on earlier counterexamples by Raynaud-Ogus-Gabber and Vasiu-Zink, we show that the ramification assumption is optimal for the result of Cadoret-Tamagawa on families of abelian varieties. This reflects the fact that the extension property of Siegel Shimura varieties no longer holds over a field of absolute ramification at least $p$.
\end{abstract}

\maketitle

\tableofcontents

\section{Introduction}

\subsection{Motivation and main results}

Let $\Ocal_F$ denote a complete DVR of mixed characteristic $(0,p)$, with fraction field $F$. 
Consider $\Xcal$ a smooth, separated scheme of finite type over $\Ocal_F$, let $X/F$ denote its generic fiber. 
Fix $Y \to X$ a smooth and proper family of varieties. 
We are concerned with the following problem: assuming (sufficiently many of) the closed fibers $Y_x$ have good reduction, does $Y$ extend to a smooth and proper scheme over $\Xcal$? 
Does pointwise good reduction imply global good reduction?

This question is asked in the recent work of Cadoret and Tamagawa \cite[§6.2]{cadoret2025pointwise}, where it is answered positively in the cases of pointed smooth proper curves (resp. of stable proper curves) of genus $g \ge 2$, and of abelian schemes if $F$ has absolute ramification index $e \le p-1$. 
The two main inputs of their argument are \cite[Theorem 1]{cadoret2025pointwise} stating that a local system on $X$ extends to $\Xcal$ if and only if it is pointwise unramified for a relevant collection of closed points of $X$, as well as Serre-Tate-type criteria. 
It is also discussed in \cite[§6.2, Example (2)]{cadoret2025pointwise} how \emph{integral canonical models} of Shimura varieties could be leveraged for a strategy when $\Ocal_F$ is the ring of integers of a $p$-adic field. 
Indeed, by definition these models satisfy a property akin to the Néron extension property of Néron models, thus providing a Serre-Tate criterion over general bases.

The main goal of this paper is to formalize the discussion in \textit{loc. cit.}. 
We reduce the pointwise-to-global good reduction problem to the existence of a separated, Deligne-Mumford stack $\Sscr/\Ocal_F$ equipped with a local system $\L$ such that the pair $(\Sscr, \L)$ satisfies some "L-extension property", and show that moduli spaces of curves, as well as integral canonical models of Shimura varieties, can fill this role. 
Consequently, we (re)establish:
\begin{thm}\label{thm:main thm}
    Assume $p>2$. Let $F / \Qp$ be a finite field extension of ramification index $e$. Assume that $Y \to X$ is either:
    \begin{itemize}
        \item[(a)] a smooth, proper family of genus $g$, $n$-pointed, geometrically connected curves with $2-2g < n$,
        \item[(b)] ($e \le p-1$) a polarized abelian scheme,
        \item[(c)] ($e \le p-1$) a primitively polarized K3 surface,
        \item[(d)] ($e \le p-1$) a smooth, proper family of cubic fourfolds.
    \end{itemize}
    If $Y \to X$ has pointwise good reduction, then it has global good reduction.
\end{thm}

Cases (a) and (b) were already known, and our proof is the same as in \cite[§6.2, Examples (1), (3)]{cadoret2025pointwise}, only reformulated to fit our framework. Building upon \cite[Theorem 28]{vasiu2010purity} we also show the optimality of their assumption $e \le p-1$, and relate it to the non-existence of a smooth integral canonical model for Siegel modular varieties if $e \ge p$. We have not investigated the possibility of a similar obstruction for orthogonal-type Shimura varieties involved in cases (c) and (d).

\subsection{Outline and strategy of proof}

\Cref{sec:can loc sys on schemes and stacks} introduces the formalism of local systems and torsors on stacks used in the rest of the paper. Of particular importance to us are towers of stacks with finite étale transition maps inducing compatible systems of torsors called \emph{fit systems}: associated with this setup is a \emph{canonical local system}. When the stacks carry specific moduli interpretations, we define a \emph{universal local system} and argue that canonical and universal local systems agree, thus generalizing constructions and statements from \cite{ullmo2013generalised} and \cite{cadoret2020integral}.

We formulate in \Cref{sec:pointwise and global good reduction problems} the pointwise-to-global good reduction problem. We then show that even in the usually well-behaved case of abelian schemes the answer is negative: if $F/\Q_p$ has ramification index $e \ge p$, \Cref{cor:counterexample PGR not GGR} provides a principally polarized counterexample with arbitrary level structure. To obtain it, we recall the work of Vasiu-Zink in \cite{vasiu2010purity} and descend their construction to finite type bases over $\Ocal_F$.

We spend some time in \Cref{sec:pointwise criterion} recalling the main inputs from \cite{cadoret2025pointwise} to be used later. Our assumptions and conventions differ slightly from the reference, so for convenience we state their results, as well as certain proofs, with the terminology of \Cref{sec:can loc sys on schemes and stacks}.

\Cref{sec:extension properties for good reduction} is the core of the proof toward \Cref{thm:main thm}. For the convenience of the reader, we motivate the technical results with case (c) in mind. With $X/F$ and $\Xcal/\Ocal_F$ as above, let $Y \to X$ be a family of K3 surfaces (with extra structure) having pointwise good reduction, we wish to extend it to a family $\Ycal \to \Xcal$ of K3 surfaces. Let $\Mscr$ denote the moduli stack of K3 surfaces (with extra structure) over $\Ocal_F$, the family defines a \emph{PGR-morphism} $X \to \Mscr$ (as in \emph{Pointwise Good Reduction}, see \Cref{dfn:PGR-morphism}) which we wish to extend to $\Xcal$. If it is true that any PGR-morphism to $\Mscr$ admits an extension over $\Ocal_F$, we say that $\Mscr$ has \emph{PGR-extension} (\Cref{dfn:PGR-extension property}).

To argue that $\Mscr$ has PGR-extension, we need to introduce two other extension properties of moduli stacks: \emph{L-extension} and \emph{smooth-extension} (\Cref{dfn:L-extension property} and \Cref{dfn:smooth extension}). The majority of \Cref{sec:extension properties for good reduction} is spent establishing the implications below:
\begin{enumerate}
    \item[(i)] If $\Sscr_{\infty}/\Ocal_F$ is a stack with smooth-extension which can be expressed as the limit of a fit system
    \[\Sscr_{\infty} = \varprojlim_K \Sscr_K,\]
    then each $\Sscr_K$ has L-extension (\Cref{prop:smooth extension implies L extension}). This is where the notion of canonical local system from \Cref{sec:can loc sys on schemes and stacks} is used. The implication had already been observed in \cite[§1.3]{bakker2024integral} in the context of Shimura varieties.
    \item[(ii)] If $\Sscr/\Ocal_F$ has L-extension, then it also has PGR-extension (\Cref{prop:L-extension implies PGR extension}).
    \item[(iii)] If $\Sscr/\Ocal_F$ has PGR-extension and $\Ucal \to \Sscr$ is étale, then $\Ucal$ has PGR-extension (\Cref{thm:PGR extension pulls back for stacks}). This fails for L-extension, as shown by the open inclusion of the generic fiber in $\Sscr$.
    \item[(iv)] If $\Ucal \to \Mscr$ is a finite torsor over $\Ocal_F$ and if $\Ucal$ has PGR-extension, then so does $\Mscr$ (\Cref{prop:PGR descends along finite torsors}).
\end{enumerate}
The inputs allowing (ii) and (iv) are the results recalled in \Cref{sec:pointwise criterion}. Summarizing, the extension properties imply each-other from top-right to bottom-left in the following diagram, using (i) to (iv) in order:

\[\begin{tikzpicture}[node distance=1cm and 1cm]

    \node (U) {$\Ucal$};
    \node[right=of U] (S) {$\Sscr = \Sscr_K$};
    \node[above=of S] (A) {$\Sscr_{\infty}$};
    \node[below=of U] (M) {$\Mscr$};

    \draw[->] (A) -- node[right] {projection} (S);
    \draw[->] (U) -- node[above] {étale} (S);
    \draw[->] (U) -- node[left] {finite torsor} (M);

    \node[draw, right=0.5cm of A] {smooth-extension};
    \node[draw, right=0.5cm of S] {L-extension};
    \node[draw, left=0.5cm of U] {PGR-extension};
    \node[draw, left=0.5cm of M] {PGR-extension};

\end{tikzpicture}\]

In \Cref{sec:extension for Shimura} we investigate our primary source of stacks with smooth-extension: integral canonical models of Shimura varieties. These are schemes $\Sscr_{\infty}$ usually defined over valuation rings of number fields. However, none of the extension properties above are preserved under base-change in general: for $F/\Qp$ an arbitrary finite extension, $\Sscr_{\infty,\Ocal_F}$ may not have smooth extension. Using an imput from Vasiu-Zink in \cite[Corollary 5]{vasiu2010purity}, the majority of \Cref{sec:extension for Shimura} is spent recalling the construction of certain integral canonical models from the work of Kisin and Madapusi Pera in \cite{kisin2010integral} and \cite{pera2016integral} respectively, to check that smooth extension is preserved whenever $F / \Qp$ has ramification index $e \le p-1$. 

In \Cref{sec:applications} we argue that the relevant moduli spaces $\Mscr$ in \Cref{thm:main thm} have PGR-extension using the models from \Cref{sec:extension for Shimura} and implications from \Cref{sec:extension properties for good reduction}. For case (a), it is known from \cite{stix2005monodromy} that $\Mscr$ has L-extension, hence PGR-extension. Case (b) follows directly from extension properties of moduli spaces of abelian schemes derived in \Cref{sec:extension for Shimura}; we then relate the counterexample from \Cref{sec:pointwise and global good reduction problems} to the possible non-existence of integral canonical models for Siegel Shimura varieties if $F$ is too ramified (\Cref{rk:optimality of e}). Cases (c) and (d) are similar, but require (iii) and (iv) additionally: the period morphism defines an étale map from a moduli space $\Ucal=\Mscr_K$ with $K$-level structure, which is a finite torsor over $\Mscr$, to an integral canonical model $\Sscr_K$ of a Shimura stack of level $K$. \Cref{thm:main thm} summarizes the contents of \Cref{thm: PGR GGR curves}, \Cref{thm:PGR GGR ab schemes}, \Cref{thm:PGR GGR for K3} and \Cref{thm:PGR GGR cubic fourfolds}.

\subsection{Conventions and notations}
Some conventions for the rest of the paper are:
\begin{itemize}
    \item Unless specified otherwise, $p>0$ is a fixed prime number. It will eventually be odd.
    \item A variety over a field is a reduced separated scheme of finite type over the field.
    \item A \emph{smooth model} of a variety $X/F$ is a smooth algebraic space $\Xcal / \Ocal_F$ with generic fiber $X/F$.
    \item The ring $\Ocal_F$ is a complete DVR with fraction field $F$ of characteristic 0. The integral closure of $\Ocal_F$ in some fixed algebraic closure of $F$ is denoted $\overline{\Ocal_F}$. We sometimes ask that $\Ocal_F$ is the ring of integers of a finite extension $F/\Qp$.
\end{itemize}

\subsection{Acknowledgements}
We thank first and foremost Anna Cadoret, for their careful reviewing and numerous suggestions which greatly improved this paper, as well as Livia Grammatica for key inputs in the counterexamples of \Cref{sec:pointwise and global good reduction problems} and multiple helpful discussions on various topics of this article. We thank Drimik Roy and Guillaume Pignon--Ywanne for their help with the arguments of \Cref{subsec:universal fit system} and \Cref{subsec:quotient of sheaves vs quotient sheaves} respectively. We also thank Tomás Fernández, Enrico Lampetti and Benoît Stroh for helpful discussions.

\subsection{AI disclosure}
No AI were used in the making of this paper. Any typos, inaccuracies or mistakes are the author's.

\section{Canonical local systems on schemes and stacks}\label{sec:can loc sys on schemes and stacks}

This section mostly serves as a toolbox of definitions: we advise the reader to skip it at first read, and come back to it whenever necessary. 

We use the language of fibered categories and (algebraic) stacks over a site, and in particular the pro-étale site. The meaning of these terms varies accross references; for the convenience of the reader we spend the first paragraph expliciting our conventions.  

The next two paragraphs define what is meant in the rest of the paper by local systems and torsors on schemes and stacks. In particular, to certain projective systems of schemes/stacks called \emph{fit systems} are attached \emph{canonical} local systems which play a role in later applications; this construction is inspired from \cite[§2]{ullmo2013generalised} and \cite[§3.1]{cadoret2020integral}. We show in \Cref{prop:universal fit system} of the fourth paragraph that there is a universal fit system of classifying stacks inducing all other fit systems. 

When the stacks of a fit system additionally carry a certain moduli interpretation, one can also define a \emph{universal} local system: this is formalized in the fifth paragraph, where it is shown that it is canonically isomorphic to the canonical local system. This identification recovers \cite[paragraph following Remark 2.8]{ullmo2013generalised} (see also \cite[Proposition 5.2]{cadoret2020integral}) by applying it to points on Siegel modular varieties, as well as \cite[end of the proof of Theorem 6.6]{cadoret2020integral} for points on moduli spaces of K3 surfaces. This result is not essential for the rest of the paper.

The sixth and last paragraph is devoted to the proof of \Cref{lem:quotient sheaf is sheaf of quotient} used several times in this section, and could be of independent interest.

We fix $S$ a base scheme and set $\Cscr \coloneq (\mrm{Sch}/S)_{\text{proét}}$ the big pro-étale site over $S$; in the rest of the article we will often implicitely be the case that $S = \Spec \, \Ocal_F$, $S = \Spec \, \Z_{(p)}$ or generic and closed points thereof. Unless stated otherwise, all fibered categories, stacks and morphisms between them are over $\Cscr$. For all $S$-scheme $T$, we write $\Cscr_T = (\mrm{Sch}/T)_{\text{proét}}$ the restriction of $\Cscr$ to $T$. We fix $R$ a topological ring and $V$ a free $R$-module of finite rank, we let $\GL(V)$ denote the topological group of $R$-linear automorphisms of $V$. We also fix $Q$ a topological group.

\subsection{(Algebraic) stacks on the big pro-étale site}\label{subsec:alg stacks on the big pro-étale site}

In the rest of this paper, we will need the notion of (algebraic and Deligne-Mumford) stacks over the big pro-étale site over $S$. In this paragraph we make our conventions precise.

\subsubsection{"The" pro-étale site over $S$}
The reader willing to ignore set-theoretic issues is encouraged to skip this paragraph and proceed with the intuitive understanding that the big pro-étale site over $S$ is the category of all schemes on $S$ with coverings refining the étale site, in that projective limits of finite étale covers are again covers.

We define a site as in \cite[\href{https://stacks.math.columbia.edu/tag/00VG}{Tag 00VG}]{stacks-project}. In particular, we require that objects, morphisms between objects and the collection of all coverings within the site are all sets. This is to ensure the existence of sheafification of presheaves (resp. stackification of prestacks) over a site. The big category $\mrm{Sch}/S$ does not satisfy these conditions for any topology: the collection of objects needs to be reduced to a set. When shrinking, one has to ensure to keep all $S$-schemes which are essential in context\footnote{For instance, we will be studying applications to certain Shimura varieties over $S = \Spec \, \Q$ and $S = \Spec \, \Z_{(p)}$, so whichever set we replace $\mrm{Sch}/S$ with, it should contain those varieties.}, as well as making sure the resulting set is still stable under useful operations. The question thus becomes: is it possible to find a \emph{set} $\Scal$ within the collection of all $S$-schemes such that
\begin{itemize}
    \item[(i)] the set $\Scal$ contains a prescribed set $\Scal_0$ of $S$-schemes, and
    \item[(ii)] for any $T \in \Scal$ and any scheme $U$ obtained from $T$ by some "reasonable scheme-theoretic construction", there exists $U' \in \Scal$ such that $U' \cong U$ as $S$-schemes.
\end{itemize}
The answer is positive, and is made precise in \cite[\href{https://stacks.math.columbia.edu/tag/000H}{Tag 000H}]{stacks-project}. For convenience, we still denote this shrunk down category as $\mrm{Sch}/S$. To equip it with a topology $\tau$, namely fppf or pro-étale, the collection of all coverings needs also to be reduced with similar care; this is done in \cite[\href{https://stacks.math.columbia.edu/tag/000W}{Tag 000W}]{stacks-project}. It is shown in \cite[\href{https://stacks.math.columbia.edu/tag/00VY}{Tag 00VY}]{stacks-project} that the resulting category of sheaves on $\mrm{Sch}/S$ is independent of the set of $\tau$-coverings chosen in this manner. 

The upshot is that there are many possible "big fppf/pro-étale" sites over $S$, but for all intents and purposes in this article any choice is virtually equivalent.

\begin{dfn}
    For the remainder of these notes, we fix $(\mrm{Sch}/S)_{\text{fppf}}$ and $\Cscr \coloneq (\mrm{Sch}/S)_{\text{proét}}$ big fppf and pro-étale sites over $S$, as in \cite[\href{https://stacks.math.columbia.edu/tag/021R}{Tag 021R}]{stacks-project} and \cite[\href{https://stacks.math.columbia.edu/tag/098G}{Tag 098G}]{stacks-project} respectively. We can and will arrange so that the underlying categories are the same, and so that our choice of fppf coverings is a subset of our choice of pro-étale coverings. We call them \emph{the} big sites over $S$.
\end{dfn}

\subsubsection{Algebraic stacks on a site}
We use the language of fibered categories and stacks over a site as in \cite{vistoliNotesGrothendieckTopologies2007}, namely, a stack over a site is a category fibered in groupoids for which every descent datum is effective. In particular, we can make sense of stacks over the pro-étale site $\Cscr$ over $S$.

Algebraic stacks, on the other hand, are only defined over $(\mrm{Sch}/S)_{\text{fppf}}$ (or even $(\mrm{Sch}/S)_{\text{ét}}$) in most references. Our conventions are inspired from \cite[4]{laumonChampsAlgebriques2000}.

\begin{dfn}\label{dfn:algebraic and DM stacks}
    A stack over $(\mrm{Sch}/S)_{\text{fppf}}$ (resp. $\Cscr = (\mrm{Sch}/S)_{\text{proét}}$) is \emph{algebraic} if its diagonal is representable, quasi-compact and separated, and if there exists a surjective smooth morphism (called an \emph{atlas}) from an algebraic space. It is \emph{Deligne-Mumford} if the atlas can be chosen étale. A Deligne-Mumford stack is \emph{separated} if its diagonal is proper (equiv. finite).
\end{dfn}

We will work with standard fppf algebraic or Deligne-Mumford stacks (classifying stacks, moduli stacks of curves, abelian varieties,$\dots$), and view them as pro-étale stacks. When we do so, it is allowed by the following fact, without further mention.

\begin{fact}[\cite{stacks-project},\href{https://stacks.math.columbia.edu/tag/0GRH}{Tag 0GRH}]
    An algebraic (resp. Deligne-Mumford) stack over $(\mrm{Sch}/S)_{\text{fppf}}$ with ind-quasi-affine diagonal has all fpqc descent data effective. In particular, it is automatically an algebraic (resp. Deligne-Mumford) stack over $\Cscr$.
\end{fact}

In this paper, stacks are not required to be algebraic until \Cref{prop:morphisms isomorphic over dense subset}.

\subsection{Local systems and torsors on schemes}\label{subsec:local systems and torsors on schemes}
We fix an $S$-scheme $T$.

\begin{dfn}\label{dfn:sheaf from top space}
    Let $A$ be a topological space. Let $\Fcal_{A}$ denote the presheaf on $\Cscr_T$ mapping any $U \in \Cscr_T$ to the set $\Fcal_A(U) = \mrm{Cont}(U, A)$ of continuous functions $U \to A$. This is a sheaf by \cite[Lemma 4.2.12]{bhatt2013pro}.
\end{dfn}

\begin{rk}\label{rk:F_A when tot disc and qcqs}
    If $A$ is discrete then $\Fcal_{A}$ is the usual constant sheaf $\u A$. More generally, if $A$ is totally disconnected and $U \to T$ is qcqs, then $\Fcal_A(U) = \mrm{Cont}(\pi_0(U), A)$.
\end{rk}

The sheaf $\Fcal_R$ is a sheaf of rings, and $\Fcal_V$ is an $\Fcal_R$-module.
\begin{dfn}\label{dfn:R-local systems}
    An $\Fcal_R$-module locally isomorphic to the $\Fcal_R$-module $\Fcal_{V}$ is called an \emph{$R$-local system with fiber $V$}. The groupoid of $R$-local systems with fiber $V$ over $\Cscr_T$ is denoted $\Loc_V(T)$. The functor
    \[\Loc_V: T \in \Cscr \longmapsto \Loc_V(T)\]
    defines a stack over $\Cscr$.
\end{dfn}

The sheaf $\Fcal_Q$ is a sheaf of groups.
\begin{dfn}\label{dfn:Q-torsor}
    Let $\Gcal$ be a sheaf of groups on $\Cscr$. A \emph{$\Gcal$-sheaf} $\Pcal$ on $\Cscr_T$ is a sheaf of sets together with a morphism $\Gcal_T \times \Pcal \to \Pcal$ inducing a group action for every object of $\Cscr_T$. A \emph{$\Gcal$-torsor} on $\Cscr_T$ is a $\Gcal$-sheaf locally isomorphic to the $\Gcal$-sheaf $\Gcal$. The category of $\Gcal$-sheaves (resp. the groupoid of $\Gcal$-torsors) over $\Cscr_T$ is denoted $\Gcal \mbf{Sh}(T)$ (resp. $\BB \Gcal(T)$). As in \Cref{dfn:R-local systems} this defines a fibered category (resp. a stack) $\Gcal \mbf{Sh}$ (resp. $\BB \Gcal$) on $\Cscr$. When $\Gcal = \Fcal_Q$, we write $Q \mbf{Sh}$ (resp. $\BB Q$) and call its objects $Q$-sheaves (resp. $Q$-torsors). 
\end{dfn}

\begin{fact}[\cite{giraudCohomologieNonAbelienne1971}, III, Corollaire 2.2.6]\label{fact:loc sys equiv GL(V) torsor}
    If $Q = \GL(V)$ then there is an equivalence
    \[\Triv:\Loc_V \rightarrow \BB \GL(V)\]
    mapping a local system $\L$ over $T$ to its sheaf of trivializations $\uIsom(\L,\Fcal_V)$ over $T$.
\end{fact}

\begin{dfn}\label{dfn:Induction of torsors}
    Let $\phi:\Hcal \to \Gcal$ be a morphism of sheaves of groups over $\Cscr$. The forgetful functor from $\Gcal \mbf{Sh}$ to $\Hcal \mbf{Sh}$ admits a left-adjoint
    \[\Ind_{\Hcal}^{\Gcal}: \Hcal\mbf{Sh} \rightarrow \Gcal\mbf{Sh}.\]
    mapping any $\Hcal$-sheaf $\Pcal$ over $T$ to the $\Gcal$-sheaf $\Ind_{\Hcal}^{\Gcal}(\Pcal)$ over $T$ defined as the quotient of $\Gcal \times \Pcal$ under the $\Hcal$-action:
    \begin{align*}
        \Hcal \times (\Gcal \times \Pcal) & \to \Gcal \times \Pcal \\
        (h,g,p) & \mapsto (g \phi(h)^{-1},h \cdot p).
    \end{align*}
    This functor restricts to torsors and defines $\Ind_{\Hcal}^{\Gcal}: \BB \Hcal \rightarrow \BB \Gcal$. If $\Hcal = \Fcal_Q$, $\Gcal = \Fcal_{Q'}$ and $\phi$ comes from a continuous group homomorphism $Q \to Q'$, we denote it $\Ind_Q^{Q'}:\BB Q \to \BB Q'$.
\end{dfn}

\begin{rk}\label{rk:induction from prestack morphism}
    There is an alternative description for $\Ind_{\Hcal}^{\Gcal}$. Recall that $\BB \Gcal$ is the stack associated to the prestack $\BB \Gcal^{\mrm{pre}}$ such that for all $T$, $\BB \Gcal^{\mrm{pre}}(T)$ has a single object $\{*\}$ and $\Aut(\{*\}) = \Gcal(T)$. The morphism $\Hcal \to \Gcal$ induces a morphism of prestacks
    \[\BB \Hcal^{\mrm{pre}} \longrightarrow \BB \Gcal^{\mrm{pre}},\]
    the associated functor between stacks $\BB \Hcal \to \BB \Gcal$ is the induction $\Ind_{\Hcal}^{\Gcal}$. To check that both definitions agree, verify the result for the trivial torsor, then argue locally for an arbitrary torsor.
\end{rk}

Using \Cref{fact:loc sys equiv GL(V) torsor} and \Cref{dfn:Induction of torsors}, any continuous representation of $Q$ on $V$ defines a functor
\[\BB Q \rightarrow \Loc_V\]
sending $\Pcal$ to the $R$-local system with fiber $V$ associated with $\Ind_Q^{\GL(V)} \Pcal$.

We will be interested in torsors over $\Cscr$ of geometric origin.
\begin{dfn}\label{dfn:morphism is finite Q-torsor}
    Let $T' \to T$ be an algebraic space over $T$. Fix a homomorphism $Q \to \Aut_T(T')$. We say that $T' \to T$ is a \emph{$Q$-torsor} over $T$ if the functor of points $h_{T'}$ defined by
    \[(U \to T) \in \Cscr \mapsto h_{T'}(U) = \Hom_T(U,T'),\]
    is a $Q$-torsor on $\Cscr_T$ as in \Cref{dfn:Q-torsor} for the induced action of $\Fcal_Q$. If $Q$ is finite then $T' \to T$ is finite étale and $T'$ is a scheme.
\end{dfn}

If $K_0$ is a topological group, let $\Ccal(K_0)$ denote the set of all open normal subgroups of $K_0$. 
\begin{dfn}\label{dfn:fit system of schemes}
    Let $K_0$ be a profinite group. Let $(T_K)_{K \in \Ccal(K_0)}$ be a projective system of schemes. We say that it is \emph{fit for $K_0$}, or a \emph{$K_0$-fit system}, if it satisfies the following assumptions:
    \begin{enumerate}
        \item[(i)] for each $K \in \Ccal(K_0)$, the morphism $T_K \to T_{K_0}$ is a finite $K_0 / K$-torsor over $T_{K_0}$ as in \Cref{dfn:morphism is finite Q-torsor}, and
        \item[(ii)] for every $K' \subseteq K$ in $\Ccal(K_0)$, the morphism $T_{K'} \to T_{K}$ is $K_0$-equivariant for the $K_0$-action induced from (i). 
    \end{enumerate}
    Then for every $K \in \Ccal(K_0)$ the map $T_K \to T_{K_0}$ is a finite étale cover, thus it makes sense to define the projective limit
    \[T_{\infty} \coloneq \varprojlim_{K \in \Ccal(K_0)} T_K\]
    which is a pro-étale cover of $T_{K_0}$, and a $K_0$-torsor over $T_{K_0}$ as in \Cref{dfn:morphism is finite Q-torsor} with $h_{T_{\infty}} = \varprojlim_K h_{T_K}$ in $K_0\mbf{Sh}(T_{K_0})$.
\end{dfn}

The groupoid of $K_0$-fit systems over a fixed scheme $T_{K_0}$ is denoted $K_0\FS(T_{K_0})$, and $K_0\FS$ is a stack over $\Cscr$. We will see in \Cref{prop:universal fit system} that it is equivalent to $\BB K_0$ via $(T_K)_K \mapsto T_{\infty}$.

\begin{dfn}\label{dfn:can loc sys on scheme}
    Let $K_0$ be a profinite group with a fixed continuous embedding $K_0 \hookrightarrow \GL(V)$.
    Let $(T_K)_{K \in \Ccal(K_0)}$ be schemes fit for $K_0$, let $h_{T_{\infty}}$ be the $K_0$-torsor from \Cref{dfn:fit system of schemes}. The corresponding \emph{canonical local system} $\L^{\can} \in \Loc_V(T_{K_0})$ is the $R$-local system associated with $\Ind_{K_0}^{\GL(V)} (h_{T_{\infty}})$.
\end{dfn}

\begin{rk}
Let us give a more down-to-earth description of the canonical local system in case $T_{K_0}$ is a connected scheme with a $K_0$-fit system $(T_K)_K$. Fix $(\overline{t}_K)_K$ a compatible system of geometric points of $(T_K)_K$. For each $K \in \Ccal(K_0)$, $T_{K} \to T_{K_0}$ is a $K_0/K$-torsor on $T_{K_0}$ in the usual sense for the étale topology on $S$. Having fixed geometric points, this torsor is uniquely described by a continuous morphism
\[\pi_1(T_{K_0},\overline{t}_{K_0}) \to K_0 / K.\]
The compatibility conditions allow to take the projective limit over $K$ so as to find a continuous morphism
\[\pi_1(T_{K_0},\overline{t}_{K_0}) \to K_0.\]
The composition with the embeddings $K_0  \hookrightarrow \GL(V)$ yields the monodromy action on the fibers of $\L^{\can}$. If $T_{K_0}$ is locally noetherian but no longer connected, we can argue component-wise.
\end{rk}

\begin{ex}
    Let $(G,\Omega)$ be a Shimura datum with reflex field $E$, set $S = \Spec \, E$. Let $K_0$ be a neat compact open subgroup of $G(\A_f)$ where $\A_f$ denotes the finite adèles over $\Q$. If $K \in \Ccal(K_0)$, let $T_K = \Sh_K(G,\Omega)$ denote the associated Shimura variety of level $K$. Then $(\Sh_K(G,\Omega))_K$ is fit for $K_0$, so $\Sh_{K_0}(G,\Omega)$ carries a canonical local system. The neatness assumption is equivalent to condition (i) of \Cref{dfn:fit system of schemes}. We can remove this assumption by replacing schemes with stacks: this motivates \Cref{subsec:Local systems and torsors on stacks} below.
\end{ex}

\subsection{Local systems and torsors on stacks}\label{subsec:Local systems and torsors on stacks}
Let $\Xscr$ be a stack on $\Cscr$. We extend the notions of the previous paragraph to $\Xscr$.
\begin{dfn}\label{dfn:loc sys and torsor on stack}
    We denote by $Q\mbf{Sh}(\Xscr)$ (resp. $\Loc_V(\Xscr)$, $\BB Q (\Xscr)$, $K_0\FS(\Xscr)$) the category of morphisms of fibered categories $\Xscr \to Q\mbf{Sh}$ (resp. $\Xscr \to \Loc_V$, $\Xscr \to \BB Q$, $\Xscr \to K_0\FS$), the objects of which are called \emph{$Q$-sheaves on $\Xscr$} (resp. \emph{$R$-local systems with fiber $V$ on $\Xscr$}, \emph{$Q$-torsors on $\Xscr$}, \emph{$K_0$-fit systems over $\Xscr$}).
\end{dfn}
Explicitely, an object of $Q\mbf{Sh}(\Xscr)$ is the data, for each $\xi:T \to \Xscr$, of a $Q$-sheaf $F_{\xi} \in Q\mbf{Sh}(T)$, collectively satisfying compatibility conditions as $\xi$ and $T$ vary. Identical descriptions hold for $\Loc_V(\Xscr)$ and $\BB Q (\Xscr)$.

The induction functor and the equivalence of \Cref{fact:loc sys equiv GL(V) torsor} automatically extend to $Q$-sheaves, local systems and torsors on $\Xscr$. Consequently, any continuous representation of $Q$ on $V$ defines a functor
\[\BB Q(\Xscr) \rightarrow \Loc_V(\Xscr)\]
sending a $Q$-torsor $\Pcal$ on $\Xscr$ to the $R$-local system with fiber $V$ on $\Xscr$ associated with $\Ind_Q^{\GL(V)} \Pcal$.

We will be interested in a geometric description of objects of $\BB Q(\Xscr)$ and $K_0\FS(\Xscr)$.
\begin{dfn}\label{dfn:morphism of stacks is Q-torsor}
    Let $f:\Xscr' \to \Xscr$ be a representable morphism of stacks over $\Cscr$. Assume $\Xscr'$ is equipped with a pseudo-$Q$-action\footnote{Formally, a pseudo-$Q$-action is a 2-morphism $\BB Q \to \mbf{Cat}$ with essential image $\Xscr'$. Explicitely, this is the data for every $q \in Q$ of an equivalence $a_q:\Xscr' \xrightarrow{\sim} \Xscr'$ such that there are isomorphisms $a_q \circ a_{q'} \simeq a_{qq'}$ satisfying the associativity condition.} preserving \footnote{Formally, for any $q \in Q$ inducing an equivalence $a_q:\Xscr' \to \Xscr'$, there is an isomorphism $f \circ a_q \simeq f$ satisfying omitted compatibility conditions.} $f$. For any scheme $T \to \Xscr$, the fiber product $T' \coloneq T \times_{\Xscr} \Xscr'$ is an algebraic space which inherits a $Q$-action $Q \to \Aut_T(T')$. We say that $f$ is a \emph{$Q$-torsor} if for all such $T$ the morphism $T' \to T$ is a $Q$-torsor as in \Cref{dfn:morphism is finite Q-torsor}. This defines a $Q$-torsor $h_{\Xscr'}$ on $\Xscr$ as in \Cref{dfn:loc sys and torsor on stack}.
\end{dfn}

\begin{dfn}\label{dfn:fit system of stacks}
    Let $K_0$ be a profinite group. Let $(\Xscr_K)_{K \in \Ccal(K_0)}$ be a projective system of stacks. We say that it is \emph{fit for $K_0$}, or a \emph{$K_0$-fit system}, if it satisfies the following assumptions:
    \begin{enumerate}
        \item[(i)] for each $K \in \Ccal(K_0)$, the morphism $\Xscr_K \to \Xscr_{K_0}$ is a finite $K_0 / K$-torsor as in \Cref{dfn:morphism of stacks is Q-torsor}, and
        \item[(ii)] for every $K' \subseteq K$ in $\Ccal(K_0)$, the morphism $\Xscr_{K'} \to \Xscr_{K}$ is pseudo-$K_0$-equivariant\footnote{This is the categorical version of an equivariant morphism, sending the pseudo-action on $\Xscr_{K'}$ to the pseudo-action on $\Xscr_K$ up to isomorphisms satisfying omitted compatibility conditions.} for the $K_0$-action induced from (i). 
    \end{enumerate}
    Then for every $K \in \Ccal(K_0)$ the map $\Xscr_K \to \Xscr_{K_0}$ is finite étale. One can define the projective limit
    \[\Xscr_{\infty} \coloneq \varprojlim_{K \in \Ccal(K_0)} \Xscr_K\]
    which is pro-étale over $\Xscr_{K_0}$, and is a $K_0$-torsor as in \Cref{dfn:morphism of stacks is Q-torsor} with $h_{\Xscr_{\infty}} = \varprojlim_K h_{\Xscr_K}$ in $K_0\mbf{Sh}(\Xscr_{K_0})$.
\end{dfn}

We show that both notions of $K_0$-fit systems on a fixed stack $\Xscr_{K_0}$ coincide. Provided a projective system $(\Xscr_K)_{K \in \Ccal(K_0)}$ as in \Cref{dfn:fit system of stacks} and an $S$-scheme $T_{K_0}$, we define a functor $\Xscr_{K_0} \to K_0\FS$ by sending a $T$-object $T \to \Xscr$ to the $K_0$-fit system $(T \times_{\Xscr} \Xscr_{K})_{K \in \Ccal(K_0)}$. Conversely, provided a functor $\Xscr_{K_0} \to K_0\FS$, for any $K \in \Ccal(K_0)$ we define $\Xscr_K(T_{K_0})$ as follows: for any $S$-scheme $T_{K_0}$, we consider the groupoid
\[\Xscr_K(T_{K_0}) = \{(\xi,T_K) \; | \; \xi \in \Xscr_{K_0}(T_{K_0}) \text{ mapping to } (T_{K'})_{K' \in \Ccal(K_0)} \in K_0 \FS(T_{K_0}) \text{ via } \Xscr_{K_0} \to K_0\FS\}.\]
One then checks that $\Xscr_K$ is a stack, and that $(\Xscr_K)_{K \in \Ccal(K_0)}$ is a $K_0$-fit system on $\Xscr_{K_0}$ as in \Cref{dfn:fit system of stacks}. The constructions are quasi-inverses of each-other.

\begin{dfn}\label{dfn:can loc sys on stack}
    Let $K_0$ be a profinite group with a fixed continuous embedding $K_0 \hookrightarrow \GL(V)$.
    Let $(\Xscr_K)_{K \in \Ccal(K_0)}$ be stacks fit for $K_0$, let $h_{\Xscr_{\infty}}$ be the $K_0$-torsor defined in \Cref{dfn:fit system of stacks}. The corresponding \emph{canonical local system} $\L^{\can} \in \Loc_V(\Xscr_{K_0})$ is the $R$-local system associated to $\Ind_{K_0}^{\GL(V)} (h_{\Xscr_{\infty}})$.
\end{dfn}

If $\Xscr_{K_0}$ is a scheme, then so are all the $\Xscr_K$, and Definitions \ref{dfn:can loc sys on scheme} and \ref{dfn:can loc sys on stack} coincide.

\begin{ex}
    Let $(G,\Omega)$ be a Shimura datum with reflex field $E$. Let $K_0$ be a neat compact open subgroup of $G(\A_f)$. Let $\Sh_{\infty}(G,\Omega) = \varprojlim_{K \in \Ccal(K_0)} \Sh_K(G,\Omega)$ on which $G(\A_f)$ acts. For $K'$ \emph{any} compact open subgroup of $G(\A_f)$, define $\Sh_{K'}[G,\Omega]$ as the quotient stack $[\Sh_{\infty}(G,\Omega) / K']$ (if $K'$ is neat this agrees with the earlier definition of $\Sh_{K'}(G,\Omega)$). If we fix $K_0'$ \emph{any} compact open subgroup of $G(\A_f)$ then $(\Sh_{K'}[G,\Omega])_{K' \in \Ccal(K_0')}$ is a $K_0'$-fit system of stacks; in particular $\Sh_{K_0'}[G,\Omega]$ carries a canonical local system.
\end{ex}

\subsection{The universal $K_0$-fit system}\label{subsec:universal fit system}
Induction between classifying stacks provide a universal $K_0$-fit system.

\begin{prop}\label{prop:universal fit system}
    Let $K_0$ be a profinite group. 
    \begin{itemize}
        \item[(i)] If $K \in \Ccal(K_0)$, the induction functor defines a finite $K_0/K$-torsor $\BB K \to \BB K_0$ on $\BB K_0$ as in \Cref{dfn:morphism of stacks is Q-torsor}, so that $(\BB K)_{K \in \Ccal(K_0)}$ is fit for $K_0$. The associated $K_0$-torsor recovers the universal torsor on $\BB K_0$ defined by the identity $\BB K_0 \to \BB K_0$.
        \item[(ii)] There is an equivalence of stacks $K_0\FS \longleftrightarrow \BB K_0$ given by the functors:
        \begin{align*}
            (T_K)_K & \longmapsto T_{\infty} \\
            (T_{K_0} \times_{\BB K_0} \BB K)_K & \longmapsfrom [T_{K_0} \longrightarrow \BB K_0].
        \end{align*}
    \end{itemize}
\end{prop}

\begin{rk}
    \Cref{prop:universal fit system} (ii) immediately extends to stacks: any fit system $(\Xscr_K)_K$ on $\Xscr_{K_0}$ is obtained by pulling back the universal system $(\BB K)_K$ along the morphism $\Xscr_{K_0} \to \BB K_0$ defining $\Xscr_{\infty} \to \Xscr_{K_0}$. For this reason we call $(\BB K)_K$ the \emph{universal $K_0$-fit system}.
\end{rk}

The proof of \Cref{prop:universal fit system} (i) will rely on the following lemma.

\begin{lem}\label{lem:fiber product of classifying stacks}
    Let $\Hcal \to \Gcal$ be a morphism of sheaves of groups on $\Cscr$. Then there is a 2-cartesian square
    \[
    \begin{tikzcd}
    \BB \Hcal \arrow[r] \arrow[d] & \BB \Gcal \arrow[d] \\
    \mrm{pt} \arrow[r]               & \BB (\Gcal/\Hcal)
    \end{tikzcd}
    \]
    where all arrows are induction functors. Here $\mrm{pt}$ denotes $\Cscr$ viewed as a stack over itself, or as the classifying stack of the trivial group.
\end{lem}
\begin{proof}
    As stackification commutes with 2-fiber product (see \cite[\href{https://stacks.math.columbia.edu/tag/04Y1}{Tag 04Y1}]{stacks-project}), it suffices to show the result at the level of prestacks. See \Cref{rk:induction from prestack morphism} for the definitions of the prestacks $\BB \Hcal^{\mrm{pre}}$, $\BB \Gcal^{\mrm{pre}}$ and $\BB (\Gcal/\Hcal)^{\mrm{pre}}$, and of the induction functor between them. Let $T$ be an $S$-scheme, a $T$-object of $\mrm{pt} \times_{\BB (\Gcal/\Hcal)^{\mrm{pre}}} \BB \Gcal^{\mrm{pre}}$ is just the datum of an element $q \in (\Gcal/\Hcal)(T)$, and one checks that two objects defined by $q$ and $q'$ are isomorphic if and only if $q' q^{-1}$ lies in the image of $\Gcal(T)$. As $\Gcal \to \Gcal/\Hcal$ is surjective, we find that any two objects of the fiber product are locally isomorphic, and so the same holds true of the stackified fiber product. Since this stack always admits a $T$-point induced by the neutral element of $(\Gcal/\Hcal)(T)$, it is the classifying stack of the automorphism group sheaf of this element. This is easily computed to be $\Hcal$.
\end{proof}

We will apply \Cref{lem:fiber product of classifying stacks} to $\Gcal = \Fcal_{K_0}$ and $\Hcal = \Fcal_{K}$. Before doing so, we need one more lemma.

\begin{lem}\label{lem:universal family for BG}
    Let $\Gcal$ be a sheaf of groups over $\Cscr$, fix $T$ an $S$-scheme, and $\Pcal$ a $\Gcal$-torsor over $T$. Let $T \to \BB \Gcal$ be the morphism defining $\Pcal$. Then there is a 2-cartesian diagram
    \[
    \begin{tikzcd}
    \tilde{\Pcal} \arrow[r] \arrow[d] & T \arrow[d] \\
    \mrm{pt} \arrow[r]               & \BB \Gcal.
    \end{tikzcd}
    \]
    Here $\tilde{\Pcal}$ is the sheaf on $\Cscr$ obtained as the extension from $\Cscr_T$ to $\Cscr$ of $\Pcal$ by the empty set (see \cite[\href{https://stacks.math.columbia.edu/tag/00Y0}{Tag 00Y0, (5)}]{stacks-project}).

    If $\Gcal = \Fcal_Q$ for a profinite group $Q$, then both $\Pcal$ and $\tilde{\Pcal}$ are represented by the same scheme (viewed as a $T$-scheme and as an $S$-scheme respectively). In particular, $\mrm{pt} \to \BB Q$ is a $Q$-torsor which is the universal family over $\BB Q$.
\end{lem}
\begin{proof}
    A direct computation shows that the 2-fiber product is the sheaf of sets described as:
    \[(U \to S) \longmapsto \coprod_{u:U \to T} \Pcal(u:U \to T).\]
    From \cite[\href{https://stacks.math.columbia.edu/tag/03CD}{Tag 03CD}]{stacks-project} this coincides with $\tilde{\Pcal}$. The second statement follows from \cite[\href{https://stacks.math.columbia.edu/tag/03HU}{Tag 03HU}]{stacks-project} along with the fact that (pro)finite torsors are representable by (pro)finite étale covers.
\end{proof}

\begin{rk}\label{rk:universal Q torsor}
    \Cref{lem:universal family for BG} immediately extends to stacks: if $Q$ is profinite, any $Q$-torsor $\Xscr' \to \Xscr$ fits in the 2-cartesian square
    \[
    \begin{tikzcd}
        \Xcal' \arrow[r] \arrow[d] & \Xcal \arrow[d] \\
        \mrm{pt} \arrow[r]               & \BB Q.
    \end{tikzcd}
    \]
\end{rk}

We are now ready for the proof.

\begin{proof}[Proof of \Cref{prop:universal fit system}]
    \Cref{lem:quotient sheaf is sheaf of quotient} below shows that $\BB (\Fcal_{K_0} / \Fcal_K) \simeq \BB \Fcal_{K_0/K} = \BB K_0/K$. From \Cref{lem:fiber product of classifying stacks} we find the 2-cartesian diagram
    \[
    \begin{tikzcd}
        \BB K \arrow[r] \arrow[d] & \BB K_0 \arrow[d] \\
        \mrm{pt} \arrow[r]               & \BB K_0/K.
    \end{tikzcd}
    \]
    \Cref{lem:universal family for BG} ensures that the bottom arrow, and hence also the top arrow, are $K_0/K$-torsors. It is straightforward to check that the associated $K_0$-torsor
    \[\BB K_0 \longrightarrow \varprojlim_K \BB K_0 / K = \BB K_0\]
    is (isomorphic to) the identity functor. This proves (i).

    To prove (ii), notice that for a fit system $(T_K)_K$ with $K_0$-torsor $T_{\infty}$ over $T_{K_0}$ corresponding to $T_{K_0} \to \BB K_0$, \Cref{lem:fiber product of classifying stacks} and \Cref{lem:universal family for BG} induce a 2-commutative diagram of 2-cartesian squares
    \[
    \begin{tikzcd}
    T_{\infty} \arrow[d] \arrow[r] & \mrm{pt} \arrow[d]        &                    \\
    T_K \arrow[d] \arrow[r]        & \BB K \arrow[d] \arrow[r] & \mrm{pt} \arrow[d] \\
    T_{K_0} \arrow[r]              & \BB K_0 \arrow[r]         & \BB K_0/K.         
    \end{tikzcd}
    \]
    It follows that there is an isomorphism $T_K \simeq T_{K_0} \times_{\BB K_0} \BB K$. Conversely, provided an $S$-scheme $T_{K_0}$ and a $K_0$-torsor $T_{K_0} \to \BB K_0$, let $T_{\infty} \coloneq T_{K_0} \times_{\BB K_0} \mrm{pt}$ be the $T_{K_0}$-scheme representing it. For any $K \in \Ccal(K_0)$ we set
    \[T_K \coloneq T_{K_0} \times_{\BB K_0} \BB K\]
    so that the diagram above holds again. This diagram shows that $T_{\infty}$ is (isomorphic to) the $K_0$-torsor associated with $(T_K)_K$. We are done.
\end{proof}

\subsection{Canonical vs universal local systems}\label{subsec:can vs univ}

Let us first motivate our goal and the formalism we introduce to achieve it. Fix some integer $g>0$ and let $\Mscr_{K_0}$ denote the moduli stack over $\Q$ of principally polarized abelian schemes of dimension $g$ with $K_0$-level structure, where $K_0$ is a compact open subgroup of $\GSp_{2g}(\A_f)$. Replacing $K_0$ with finite index open subgroups $K$, we find a $K_0$-fit system $(\Mscr_{K})_K$ of stacks in the sense of \Cref{dfn:fit system of stacks}, so we can consider the corresponding \emph{canonical} $\A_f$-local system $\L^{\can}$ on $\Mscr_{K_0}$. There is a second local system on $\Mscr_{K_0}$: the relative (adelic) Tate-module of the universal abelian scheme, which we call the \emph{universal} local system $\L_{K_0}$. We aim to show that $\L^{\mrm{can}}$ and $\L_{K_0}$ are canonically isomorphic, which is the content of \Cref{prop:can equals univ} below. 

To do so, we formalize the notion of a moduli stack with level structure in \Cref{dfn:Q moduli datum} and \Cref{dfn:Q moduli stack}. We sketch how this is done in the explicit case of $\Mscr_{K_0}$. For any scheme $T/\Q$, an object in $\Mscr_{K_0}(T)$ is a triple $(f:A\to T,\lambda,[\eta]_{K_0})$ where
\begin{itemize}
    \item $(f:A \to T,\lambda)$ is a principally polarized abelian scheme of dimension $g$, whose relative (adelic) Tate-module is denoted $\L(A)$, and
    \item $[\eta]_{K_0}$ is a section of the quotient by $\Fcal_{K_0}$ of some subsheaf of $\uIsom(\L(A), \Fcal_V)$ where $V = \A_f^{2g}$ is acted-on by $K_0$; namely we restrict to those isomorphisms that are compatible with the usual symplectic structure on $V$, and the symplectic structure on $\L(A)$ induced from $\lambda$.
\end{itemize}
To define $\Mscr_{K_0}$, we first consider the stack $\Xscr = \Ascr_{g,1}$ of principally polarized abelian schemes of dimension $g$ and view relative cohomology as a morphism from $\Xscr$ to $\Loc_{V}$. From there, we could define a stack of triples $(f:A\to T,\lambda,[\eta]_{K_0})$ with the slight caveat that $[\eta]_{K_0}$ is a section of the whole sheaf $\uIsom(\L(A), \Fcal_V)/\Fcal_{K_0}$ (no compatibility condition with the symplectic structures). To formalize replacing every $\GL_d(\A_f)$-torsor $\uIsom(\L(A), \Fcal_V)$ with the appropriate $\GSp_{2g}(\A_f)$-subsheaf (which is a $\GSp_{2g}(\A_f)$-torsor), we use the language of natural transformations between fibered categories, as explained below.

\hspace{1em}

Fix an embedding $Q \hookrightarrow \GL(V)$. Recall the equivalence of stacks from \Cref{fact:loc sys equiv GL(V) torsor}
\[\Triv:\Loc_{V} \to \BB \GL(V)\]
assigning to any $T \in \Cscr$ and any $\L \in \Loc_V(T)$ the $\GL(V)$-torsor $\uIsom(\L,\Fcal_V)$ over $T$.
\begin{dfn}\label{dfn:Q moduli datum}
    A \emph{$Q$-moduli datum} is a quadruple $(\Xscr, \L_{\Xscr}, \Lev, \iota)$ consisting of 
    \begin{itemize}
        \item[(i)] a stack $\Xscr$ over $\Cscr$,
        \item[(ii)] a local system $\L_{\Xscr}:\Xscr \to \Loc_V$ on $\Xscr$; we let $\Triv_{\Xscr} = \Triv \circ \L_{\Xscr}$,
        \item[(iii)] a $Q$-torsor $\Lev:\Xscr \to \BB Q$ on $\Xscr$, and
        \item[(iv)] a natural transformation $\iota$ between the morphisms of fibered categories:
            \[
            \begin{tikzcd}[row sep=0.2cm]
                                                               & \BB \GL(V) \arrow[rd, "\mrm{forget}", bend left]                            &          \\
            \Xscr \arrow[rd, "\Lev"' ,bend right] \arrow[ru, "\Triv_{\Xscr}",bend left] &                                                          & Q\mbf{Sh}. \\
                                                               & \BB Q \arrow[ru, "\mrm{forget}"', bend right] \arrow[uu, "\iota"', shorten <=8pt, shorten >=8pt, Rightarrow] &         
            \end{tikzcd}
            \]
    \end{itemize}
\end{dfn}
A $Q$-moduli datum is to be understood as follows. The functor $\Triv_{\Xscr}: \Xscr \to \BB \GL(V)$ assigns to any $T \in \Cscr$ and any $A \in \Xscr(T)$ the $\GL(V)$-torsor $\uIsom(\L_{\Xscr}(A),\Fcal_V)$ over $T$. The functor $\Lev$ assigns to the same data a $Q$-subsheaf of $\Triv_{\Xscr}(A)$ which is a $Q$-torsor for the induced action, considering only isomorphisms with additional constraints. The condition of being a $Q$-subsheaf for each $T \in \Cscr$ and $A \in \Xscr(T)$ is formalized by $\iota$, which is automatically a natural monomorphism.

\begin{lem}\label{lem:L comes from Lev}
    Let $(\Xscr, \L_{\Xscr}, \Lev, \iota)$ be a $Q$-moduli datum. Then $\L_{\Xscr}$ is the $R$-local system with fiber $V$ associated with the $\GL(V)$-torsor $\Ind_{Q}^{\GL(V)} (\Lev)$.
\end{lem}
\begin{proof}
    Condition (iv) is equivalent to a natural isomorphism
        \[
    \begin{tikzcd}
    \Xscr \arrow[rr, "\Triv_{\Xscr}"] \arrow[rd, "\Lev"', bend right] & {}                                               & \BB \GL(V) \\
                                                    & \BB Q \arrow[ru, "\Ind_{Q}^{\GL(V)}."', bend right] \arrow[u, shorten <=10pt, Rightarrow] &           
    \end{tikzcd}
    \]
    The equivalence follows formally from the fact that $\Ind_{Q}^{\GL(V)}:Q\mbf{Sh} \to \GL(V)\mbf{Sh}$ is left adjoint to the forgetful functor $\GL(V)\mbf{Sh} \to Q\mbf{Sh}$.
\end{proof}

It follows that a $Q$-moduli datum is determined up to isomorphism by the $Q$-torsor $\Lev$ on $\Xscr$. 

We are interested in the following moduli stacks:

\begin{dfn}\label{dfn:Q moduli stack}
    Let  $K_0$ be a compact subgroup of $Q$ which is profinite, and $K$ an open subgroup of $K_0$. Let $(\Xscr, \L_{\Xscr}, \Lev, \iota)$ be a $Q$-moduli datum.
    The associated \emph{moduli stack of level $K$} is the stack $\Mscr_K$ over $\Cscr$ assigning to each $T \in \Cscr$ a pair $(A, [\eta]_K)$, where
    \begin{itemize}
        \item[(i)] $A$ is an object in $\Xscr(T)$, and
        \item[(ii)] $[\eta]_K \in H^0(T,\Lev(A)/\Fcal_K)$ is a global section of the quotient of $\Lev(A)$ by the action of $\Fcal_K$.
    \end{itemize}
    If $K$ is normal in $K_0$, the obvious morphism $\Mscr_{K} \to \Mscr_{K_0}$ is a $K_0/K$-torsor as in \Cref{dfn:morphism of stacks is Q-torsor}.
\end{dfn}

Fix $K_0$ and $(\Xscr, \L_{\Xscr}, \Lev, \iota)$ as in \Cref{dfn:Q moduli stack}. The projective system $(\Mscr_K)_{K \in \Ccal(K_0)}$ is fit for $K_0$. By \Cref{dfn:can loc sys on stack} we can consider the canonical local system $\L^{\can}$ on $\Mscr_{K_0}$. We now define the universal local system on $\Mscr_{K_0}$ and show that it is canonically isomorphic to $\L^{\can}$. Observe that the projection onto the first component defines a morphism of stacks $p:\Mscr_{K_0} \to \Xscr$.

\begin{dfn}\label{dfn:univ loc sys}
    The \emph{universal local system} $\L_{K_0} \in \Loc_V(\Mscr_{K_0})$ on $\Mscr_{K_0}$ is the pullback of $\L_{\Xscr}$ by $p$, that is:
    \[\Mscr_{K_0} \xrightarrow{p} \Xscr \xrightarrow{\L_{\Xscr}} \Loc_V.\]
    The \emph{universal torsor} $\Lev_{K_0} \in \BB Q(X_{K_0})$ on $\Mscr_{K_0}$ is the pullback of $\Lev$ by $p$, that is:
    \[\Mscr_{K_0} \xrightarrow{p} \Xscr \xrightarrow{\Lev} \BB Q.\]
\end{dfn}

By \Cref{lem:L comes from Lev}, $\L_{K_0}$ is the $R$-local system with fiber $V$ associated with $\Ind_Q^{\GL(V)} (\Lev_{K_0})$.

\begin{prop}\label{prop:can equals univ}
    There is a canonical isomorphism $\L^{\can} \simeq \L_{K_0}$.
\end{prop}
\begin{proof}
    Recall that $\L^{\can}$ is a local system defined from a $\GL(V)$-torsor, itself induced from the $K_0$-torsor $h_{\Mscr_{\infty}}$ as in \Cref{dfn:fit system of stacks}. We will define a canonical morphism of $K_0$-sheaves between $h_{\Mscr_{\infty}}$ and $\Lev_{K_0}$ ($Q$-torsor viewed as a $K_0$-sheaf). By adjunction, as in the proof of \Cref{lem:L comes from Lev}, this defines a canonical isomorphism between the $Q$-torsors $\Ind_{K_0}^Q (h_{\Mscr_{\infty}})$ and $\Lev_{K_0}$. Applying $\Ind_{Q}^{\GL(V)}$, we find the desired isomorphism. 

    Defining a morphism $h_{\Mscr_{\infty}} \to \Loc_{K_0}$ of $K_0$-sheaves on $\Mscr_{K_0}$ amounts to find, for every $\xi:T \to \Mscr_{K_0}$, a morphism of $K_0$-sheaves $h_{\Mscr_{\infty}}(\xi) \to \Loc_{K_0}(\xi)$ on $T$ which is natural in $T$. Fixing $T \to \Mscr_{K_0}$, let $(A,[\eta]_{K_0}) \in \Mscr_{K_0}(T)$ be the corresponding object. A section of $h_{\Mscr_{\infty}}(A,[\eta]_{K_0})$ is a compatible collection of sections $T \to \Mscr_{K}$ over $\Mscr_{K_0}$ as $K$ varies, i.e a compatible system of pairs $(A, [\eta]_K)$ as $K$ varies, which is determined by the projective system of level structures $([\eta]_K)_{K} \in \varprojlim_K H^0(T,\Lev(A)/\Fcal_K)$. Now there are $K_0$ equivariant identifications
    \[\varprojlim_K H^0(T,\Lev(A)/\Fcal_K) = H^0(T,\varprojlim_K \Lev(A)/\Fcal_K) \simeq H^0(T, \Lev(A)) = H^0(T, \Lev_{K_0}(A,[\eta]_{K_0})).\]
    For the middle isomorphism we used that $K_0$ is profinite. This defines the desired morphism
    \[h_{\Mscr_{\infty}}(A,[\eta]_{K_0}) \longrightarrow \Lev_{K_0}(A,[\eta]_{K_0})\]
    which is indeed natural in $(A,[\eta]_{K_0})$.
\end{proof}

\begin{ex}\label{ex:can loc sys on shimura}
    Let $(G,\Omega)$ be a Shimura datum with reflex field $E$. Let $K_0$ be a compact open subgroup of $G(\A_f)$. If $\Sh_{K_0}[G,\Omega]/E$ carries a modular interpretation (e.g. is Hodge-type), then the local system on $\Sh_{K_0}[G,\Omega]$ coming from the relative cohomology of the universal family coincides with the canonical local system. 
\end{ex}

\subsection{Quotient of sheaves vs quotient sheaves}\label{subsec:quotient of sheaves vs quotient sheaves}

We prove the following result used in previous subsections, which could be of independent interest.

\begin{lem}\label{lem:quotient sheaf is sheaf of quotient}
    Suppose $Q$ is Hausdorff totally disconnected and $K \subseteq Q$ is a closed subgroup. Then there is a natural isomorphism
    \[\Fcal_{Q} / \Fcal_K \simeq \Fcal_{Q / K}.\]
\end{lem}

For the proof, we introduce elementary notions from condensed mathematics.

\begin{dfn}
    We denote $*_{\text{proét}}$ the pro-étale site over the point, as introduced in \cite[Definition 1.2]{scholze2026lecturescondensedmathematics}: it is the category of profinite sets with covers that are finite families of jointly surjective maps.

    If $X$ is a topological space, we let $\u{X}$ denote the sheaf on $*_{\text{proét}}$ mapping a profinite set $Y$ to the set of continuous functions $Y \to X$. If $G$ is a topological group, $\u G$ is a sheaf of groups.
\end{dfn}

The idea is to define a functor $\Cscr = (\mrm{Sch}/S)_{\text{proét}} \to *_{\text{proét}}$ mapping $T$ to its space of connected components $\pi_0(T)$. However the latter may not be compact Hausdorff in general; only on qcqs objects $T$ is it ensured that $\pi_0(T)$ is profinite (see \cite[\href{https://stacks.math.columbia.edu/tag/0906}{Tag 0906}]{stacks-project}). As affine schemes are qcqs, this motivates:

\begin{dfn}
    We denote $\Cscr_{\mrm{aff}}$ the full subcategory of $\Cscr$ whose objects are affine schemes\footnote{We ask that objects are maps $T \to S$ where $T$ is an affine scheme, \emph{not} that the morphism $T \to S$ be affine.} mapping to $S$. This is again a site whose pretopology is inherited from that of $\Cscr$.
\end{dfn}
 
There is now a restriction functor between the categories of sheaves
\[\Sh(\Cscr) \longrightarrow \Sh(\Cscr_{\mrm{aff}})\]
which is an equivalence. Indeed, every proétale cover in $\Cscr$ can be refined into a proétale cover by affine schemes belonging to $\Cscr_{\mrm{aff}}$, and so the equivalence follows from \cite[III, Théorème 4.1]{sga4tome1}. In particular, it suffices to establish the isomorphism of \Cref{lem:quotient sheaf is sheaf of quotient} between sheaves restricted to $\Cscr_{\mrm{aff}}$.

We may now consider the functor
\begin{align*}
    \pi: \Cscr_{\mrm{aff}} & \longrightarrow *_{\mrm{proet}} \\
            U & \longmapsto \pi_0(U).
\end{align*}

The following result is implicit in the paragraph preceding \cite[Lemma 4.2.13]{bhatt2013pro}.

\begin{lem}\label{lem:pi_0 cocontinuous}
    The functor $\pi$ is cocontinuous as in \cite[\href{https://stacks.math.columbia.edu/tag/00XJ}{Tag 00XJ}]{stacks-project}.
\end{lem}
\begin{proof}
    Consider $T \to S$ with $T$ affine, and $\{X_i \to \pi_0(T)\}$ a covering in $*_{\mrm{proet}}$. Then by \cite[Lemma 2.2.8]{bhatt2013pro} the topological space $U_i = X_i \times_{\pi_0(T)} T$ is endowed with the structure of an affine scheme pro-étale over $T$, and applying $\pi_0$ to $U_i \to T$ recovers $X_i \to \pi_0(T)$. The joint surjectivity of $\{U_i \to T\}$ now follows from the joint surjectivity of $\{X_i \to \pi_0(T)\}$ and the surjectivity of $T \to \pi_0(T)$.
\end{proof}

It follows from \cite[\href{https://stacks.math.columbia.edu/tag/00XO}{Tag 00XO}]{stacks-project} (see also \cite[IV, 4.7]{sga4tome1}) that $\pi$ induces a geometric morphism of topoi
\[(\pi_*,\pi^*):\Sh(\Cscr_{\mrm{aff}}) \longleftrightarrow \Sh(*_{\mrm{proet}}).\] 
In particular, $\pi^*:\Sh(*_{\mrm{proet}}) \to \Sh(\Cscr_{\mrm{aff}})$ sends $F \in \Sh(*_{\mrm{proet}})$ to the sheafification of the presheaf
\[[T \to S] \longmapsto F(\pi_0(T)).\]
By \cite[Lemma 4.2.13 (1)]{bhatt2013pro} this presheaf is already a sheaf. Thus, for any totally disconnected topological space $X$ we find
\begin{equation}\label{eq:pi*}
    \pi^* \u X = X \circ \pi_0 = \Fcal_X
\end{equation}
where the last equality follows from \Cref{rk:F_A when tot disc and qcqs} and the fact that all objects of $\Cscr_{\mrm{aff}}$ are qcqs schemes.

\begin{fact}\label{fact:condensed quotient}
    Let $Q,K$ be as in the statement of \Cref{lem:quotient sheaf is sheaf of quotient}. Then in $\Sh(*_{\mrm{proet}})$
    \[\u{Q/K} = \u Q / \u K.\]
\end{fact}
\begin{proof}
    This is explicitely stated in \cite[Remark 2.7]{clausen2026weil} and proved somewhat implicitely in \cite{scholze2026lecturescondensedmathematics}. For the convenience of the reader we sketch the argument.

    By Proposition 2.7 of \textit{loc. cit.} it suffices to check that this holds over the subcategory of compact Hausdorff extremally disconnected sets (see Definition 2.4 of \textit{loc. cit.}). Let $W$ be such a space. In the first paragraph of the Proof of Theorem 2.2 of \textit{loc. cit.} it is argued that there is a natural identification
    \[(\u Q / \u K)(W) = \u{Q}(W) / \u K(W).\]
    It remains to show that the right-hand side is identified with $\u{Q/K}(W)$ via the injective map (functorial in $W$)
    \[\u{Q}(W) / \u K(W) \longrightarrow \u{Q/K}(W).\]
    We show surjectivity, i.e. that any continuous map $f:W \to Q/K$ lifts to a continuous map $W \to Q$. Let $\overline{H} = f(W) \subseteq Q/K$ with the induced topology, and $H$ its pre-image in $Q$. Then $\overline H$ is Hausdorff totally disconnected and quasi-compact, and so is $H$ (Hausdorff totally disconnected is clear, quasi-compact follows from $Q \to Q/K$ being open). As $W$ is extremally disconnected, it follows from \cite[\href{https://stacks.math.columbia.edu/tag/08YN}{Tag 08YN}]{stacks-project} that there exists a dotted arrow
    \[\begin{tikzcd}
                                   & H \arrow[d, two heads] \\
    W \arrow[r] \arrow[ru, dashed] & \overline{H}          
    \end{tikzcd}\]
    making the diagram commute, as desired.
\end{proof}

We may now conclude.

\begin{proof}[Proof of \Cref{lem:quotient sheaf is sheaf of quotient}]
    From \Cref{eq:pi*} we find $\Fcal_{Q/K} = \pi^* \u{Q/K}$, as well as $\Fcal_Q / \Fcal_K = \pi^* \u Q / \pi^* \u K = \pi^* (\u Q / \u K)$ as $\pi^*$ is a left adjoint and hence commutes with all colimits. The result follows from applying $\pi^*$ to \Cref{fact:condensed quotient}.
\end{proof}

\section{Pointwise and global good reduction problems}\label{sec:pointwise and global good reduction problems}

In \Cref{subsec:formulation of the problem} we define what is meant by pointwise and global good reduction of a smooth and proper family of varieties, possibly with extra data. The rest of the section is devoted to showing that even in the usually well-behaved case of principally polarized abelian schemes, it is not true in general that the former implies the latter. In \Cref{subsec:abelian schemes that do not extend} we recall some constructions around abelian schemes over dense open subsets $U$ of a base local scheme $\Spec \, R$ which do not extend to $\Spec \, R$ as a whole; here $R$ is a regular, formally smooth $\Ocal_F$ algebra where $F/\Qp$ is finite with ramification index $e \ge p$. In \Cref{subsec:abelian schemes with PGR not GGR} we descend this construction to a finite type $\Ocal_F$-algebra, and argue that this provides an abelian scheme with pointwise but not global good reduction. As \Cref{thm:main thm} indicates, no such counterexample exists if $e \le p-1$ instead.

\subsection{Formulation of the problem}\label{subsec:formulation of the problem}
We define the notions of pointwise and global good reduction.
\begin{dfn}
    Let $Y/F$ be a smooth proper variety. We say that it has \emph{good reduction} if it is the generic fiber of some smooth proper model $\Ycal/\Ocal_F$.
\end{dfn}

We define the integral locus of a stack over $\Ocal_F$. To this end, we introduce some notation. If $X/F$ is a stack, its set of closed points is denoted by $|X|$. This is the set of $F$-morphisms $\Spec \, E \to X$, with $E/F$ a finite extension, where two morphisms $\Spec \, E \to X$ and $\Spec \, E' \to X$ are identified if there exists a field $E''$, embeddings $E,E' \hookrightarrow E''$ and a 2-commutative diagram
\[
\begin{tikzcd}
    \Spec \, E'' \arrow[r] \arrow[d] & \Spec \, E' \arrow[d] \\
    \Spec \, E \arrow[r]             & X.              
\end{tikzcd}
\]
If $X$ is a scheme, this coincides with the closed points of its underlying topological space in the usual sense.
\begin{dfn}\label{dfn:integral locus}
    Let $\Xcal / \Ocal_F$ be a stack with generic fiber $X/F$. We define the \emph{integral locus} of $\Xcal$ as
    \[|\Xcal|_{\eta}^{\mrm{int}} = \im (\Xcal(\overline{\Ocal_F}) \to |X|).\]
    It is the set of closed points of $X$ which admit a specialization in the special fiber of $\Xcal$. In other words, $x_0:\Spec \, E \to X$ lies in $|\Xcal|_{\eta}^{\mrm{int}}$ if and only if, for some finite extension $E'/E$ there exists $\Spec \, \Ocal_{E'} \to \Xcal$ and a 2-commutative diagram
    \[
    \begin{tikzcd}
        \Spec \, E' \arrow[r] \arrow[d, hook]     & \Spec \, E \arrow[r,"x_0"] & X \arrow[d, hook] \\
        \Spec \, \Ocal_{E'} \arrow[rr, "\exists"] &                            & \Xcal.              
    \end{tikzcd}
    \]
\end{dfn}

If $\Xcal = X$, the integral locus is empty and if $\Xcal / \Ocal_F$ is a proper scheme, then $|\Xcal|_{\eta}^{\mrm{int}} = |X|$.

\begin{dfn}\label{dfn:pointwise good reduction}
    Let $X/F$ be a smooth variety with smooth model $\Xcal / \Ocal_F$. Consider a smooth proper morphism of schemes $Y \to X$. 
    
    \noindent We say that this family has \emph{pointwise good reduction} (relative to $\Xcal$) if for every $x \in |\Xcal|_{\eta}^{\mrm{int}}$, the fiber $Y_x / F(x)$ has good reduction. 
    
    \noindent We say that this family has \emph{global good reduction} (relative to $\Xcal$) if there exists a smooth proper map $\Ycal \to \Xcal$ of algebraic spaces over $\Ocal_F$ with generic fiber $Y \to X$.
\end{dfn}

It is clear that global good reduction implies pointwise good reduction. We are interested in whether the converse holds, and if not, what may some obstructions be.

\begin{rk}
    By default, we require that the model $\Ycal \to \Xcal$ extends as much of the structure of $Y \to X$ as possible. Typically if $Y \to X$ is a polarized abelian scheme or K3 surface with some level structure, we require $\Ycal \to \Xcal$ to be so too.
\end{rk}

\subsection{Abelian schemes that do not extend}\label{subsec:abelian schemes that do not extend}
For abelian schemes over a sufficiently ramified $p$-adic field $F$, pointwise good reduction does not imply global good reduction. We construct a counterexample in the next paragraph. Prior to this, we need some recollections on how abelian schemes over some dense open subscheme can fail to extend to an abelian scheme over the whole base.

We recall the context behind this question. In \cite[Chapter V, Corollary 6.8]{faltingsDegenerationAbelianVarieties1990a} it is stated that given a base regular scheme $S$ with generic points of characteristic zero, and a closed subscheme $Z$ of codimension at least $2$, any abelian scheme over $S \setminus Z$ uniquely extends to an abelian scheme\footnote{In the language of \Cref{dfn:healthy}, the statement claims that any such $S$ is healthy.} over $S$. A counterexample due to Raynaud-Ogus-Gabber (see \cite[§6]{de1997extending}) showed that the claim is wrong as it stands. Vasiu-Zink later generalized this construction in \cite{vasiu2010purity}; the goal of this paragraph is to recall their work so as to prove \Cref{prop:ab scheme over formally smooth base} below. From now on we consider a finite extension $F / \Qp$ of ramification index at least $p$.

\begin{prop}\label{prop:ab scheme over formally smooth base}
    There exists a local étale $\Ocal_F[[T]]$-algebra $R$ with maximal ideal $\mf m$ and an abelian scheme $A_U \to U \coloneq \Spec \, R \setminus \{\mf m\}$, together with a degree-$p^2$ (resp. principal) polarization $\lambda:A_U \to A_U^{\vee}$, such that $A_U$ does not extend to an abelian scheme over $\Spec \, R$.

    Moreover, let $\overline{\eta} \to \Spec \, R$ be a geometric point over the generic point and $N$ a prime-to-$p$ positive integer. We can arrange for $A_{\overline{\eta}}[N]$ to have trivial monodromy.
\end{prop}

\begin{rk}
    In the language of Siegel Shimura varieties, the second part of \Cref{prop:ab scheme over formally smooth base} states that we can arrange for $A_U \to U$ to have any prime-to-$p$ level structure. From now on we use this terminology for conciseness.
\end{rk}

The rest of the paragraph is dedicated to the proof of \Cref{prop:ab scheme over formally smooth base}. The main tool for constructing abelian schemes that do not extend is the following lemma.

\begin{lem}[\cite{vasiu2010purity}, Lemma 27]\label{lem:setup for abelian scheme not extending}
    Let $R$ denote a regular local ring of dimension $\ge 2$ and mixed characteristic $(0,p)$. Let $\mf{m}$ denote its maximal ideal, and set $U = \Spec \, R \setminus \{\mf m\}$.
    
    \noindent Let $D \to H$ be a homomorphism of finite flat group schemes over $\Spec \, R$ which is not a monomorphism, but whose restriction over $U$ is a monomorphism\footnote{\cite[Lemma 27]{vasiu2010purity} is stated with $H \to D$ which is an epimorphism over $U$ but not globally. However the proof only uses the Cartier dual morphism $D^t \to H^t$: here we consider the dual statement to skip taking duals.}. Let $B$ be an abelian scheme over $\Spec \, R$ in which $H$ embeds. Define the quotient abelian scheme $A_U = B_U / D_U$ over $U$.

    \noindent Then $A_U \to U$ is an abelian scheme which does not extend into an abelian scheme over $R'$.
\end{lem}
\begin{proof}

    Suppose that $A_U$ extends to an abelian scheme $A \to \Spec \, R$. The isogeny $B_U \to A_U$ extends to a homomorphism $B \to A$ (see \cite[IX Corollaire 1.4]{raynaudFaisceauxAmplesSchemas1970}) which is automatically an isogeny. Its kernel $K$ is a finite flat group scheme which coincides with $D_U$ over $U$ (and so $D \simeq K$ as finite flat group schemes over $\Spec \, R$ by \Cref{fact:purity for finite flat group schemes} below). We will argue that there is an identification $D \simeq K$ as $B$-schemes, which then implies that $D \to H$ is a monomorphism over the whole base, contradiction.

    Consider the cartesian diagram:
    \[
    \begin{tikzcd}
    D \times_H K \arrow[r] \arrow[d, hook] & K \arrow[d, hook] \\
    D \arrow[r]                            & H                
    \end{tikzcd}
    \]
    It follows that $D \times_H K$ is a finite flat subgroup scheme of $D$, which coincides with $D_U$ over $U$, hence the left vertical arrow is an isomorphism. In particular, $D \to H \to B$ factors through $K \to H \to B$. This defines a morphism $D \to H$ compatible with the embeddings in $B$. This is an isomorphism over $U$, hence also over $\Spec \, R$ by \Cref{fact:purity for finite flat group schemes}.
\end{proof}

For the convenience of the reader, we provide a proof of the following known purity result for finite flat group schemes used in the proof of \Cref{lem:setup for abelian scheme not extending}.

\begin{fact}\label{fact:purity for finite flat group schemes}
    Let $X$ be a noetherian normal scheme, $j:U \hookrightarrow X$ the complement of a closed point.
    Then the pullback functor $j^*$ from vector bundles on $X$ to vector bundles on $U$ is fully faithful. Moreover, its restriction to the subcategories of finite flat group schemes is again fully faithful.
\end{fact}
\begin{proof}
    Let $\Ecal_1,\Ecal_2$ be vector bundles on $X$; the Hom sheaf $\Fcal \coloneq \u{\Hom}(\Ecal_1,\Ecal_2)$ is again a vector bundle. We aim to show that the restriction map
    \[\Fcal(X) \longrightarrow \Fcal(U)\]
    is a bijection. This is the content of \cite[Exercice III.3.5]{hartshorneAlgebraicGeometry1977}. Note that normality could be weakened to $X$ having property $(S_2)$.

    The statement on finite flat group schemes follows from the observation that such a group scheme $G$ over $X$ (or $U$) is completely understood by its corresponding $\Ocal_X$-Hopf algebra $\Ocal_G$, which itself is the datum of a vector bundle over $X$ with additional morphisms between vector bundles induced from $\Ocal_G$ (multiplication/comultiplication, unit/counit, and so on...) subject to a collection of commutative diagrams.
\end{proof}

\begin{rk}\label{rk:base change of ab scheme not extending}
    Consider the setup of \Cref{lem:setup for abelian scheme not extending}. If $R'$ is a faithfully flat local $R$-algebra of relative dimension 0, with maximal ideal $\mf m'$, let $U' \coloneq U \times_R \Spec \, R' = \Spec \, R' \setminus \{\mf m'\}$. Then $A_U \times_U \Spec \, R' \to U'$ is an abelian scheme which also does not extend into an abelian scheme over $\Spec \, R'$. Indeed, The morphism $\Spec \, R' \to \Spec \, R$ is fpqc, so the assumptions on monomorphisms still hold after base-change to $R'$. Let $D', H', A_U', B'$ denote the group schemes base-changed to $R'$, then $A_U'$ coincides with $B'_U / D'_U$, and the conclusion follows from \Cref{lem:setup for abelian scheme not extending}.

    We will mainly use this fact to replace $R$ with an étale extension so as to allow for arbitrary level structure on $A_U$.
\end{rk}

Disregarding polarizations for now, the proof of \Cref{prop:ab scheme over formally smooth base} is the following. According to \cite[Theorem 28 (i)]{vasiu2010purity} and its proof we can take the power series ring $R = \Ocal_F[[T]]$ and find a homomorphism of group schemes $D \to H$ over $R$ satisfying the assumptions of \Cref{lem:setup for abelian scheme not extending}, hence an abelian scheme $A_U \to U$ which does not extend to an abelian scheme over $\Spec \, R$. By \Cref{rk:base change of ab scheme not extending} we can replace $R$ with an appropriate étale extension of local rings to satisfy the monodromy condition on $A_{\overline{\eta}}[N]$.

To find an example with degree-$p^2$ polarization, we take a deeper dive into the proof of \cite[Theorem 28 (i)]{vasiu2010purity} and argue in three steps. First we will see that $D$ can be arranged to be of order $p$. Second, we shall observe that the abelian scheme $B \to U$ in which $H$ embeds can be chosen to be principally polarized. Lastly, we will show that under these conditions, the dual abelian scheme $A_U^{\vee} \to U$ has degree-$p^2$ polarization and does not extend to $\Spec \, R$.

\medskip

\textbf{Step 1.} We argue that $D$ can be chosen to have order $p$. Set $R \coloneq \Ocal_F[[T_1]]$. Let $k$ denote the residue field of $F$, so that $\Ocal_F / W(k)$ is totally ramified of index $\ge p$. In particular, we can write $\Ocal_F \simeq W(k)[T_2]/(E(T_2))$ where $E(T_2)$ is an Eisenstein polynomial. Multiplying by a unit, we may assume its constant coefficient is $p$, so we can write $E(T_2) = h-p$ where $h \in (T_2) \subseteq W(k)[T_2]$. Now, if we write $\mf{S} \coloneq W(k)[[T_1,T_2]]$, we have shown that
\[R \simeq \mf{S} / (h-p)\]  
where $h \in (T_1,T_2) \setminus p \mf{S}$ is such that the reduction $\overline{h}$ of $h$ modulo $p$ is divisible by $T_2^p$. This shows that \cite[Theorem 28 (i)]{vasiu2010purity} applies, the proof of which finds a homomorphism of group schemes $D \to H$ over $R$ satisfying the assumptions of \Cref{lem:setup for abelian scheme not extending}. Let us sketch this construction to argue that $D$ can be chosen of order $p$. The key ingredient is this.

\begin{fact}[\cite{lauFrames}, Theorem 10.7]
    Let $R$ be a complete regular local ring with perfect residue field $k$ of characteristic $p>0$. Thus we can write $R = \mf{S} / E \mf{S}$ where
    \[\mf{S} = W(k)[[T_1,\dots,T_r]]\]
    and $E \in \mf{S}$ is a power series with constant term $p$.

    Then the category of connected finite flat group schemes over $R$ of $p$-power order is equivalent\footnote{For $p \neq 2$, \cite[Theorem 1.2]{lauFrames} states an equivalence between finite flat group schemes and Breuil modules. The adjectives \emph{connected} and \emph{nilpotent} allow the result to extend to characteristic 2.} to the category of nilpotent Breuil modules relative to $\mathfrak{S}$.
\end{fact}

Following \cite[Definition 8.4]{lauFrames}, a \emph{Breuil module relative to $\mf{S} \to R$} is a triple $(M,\varphi,\psi)$ consisting of 
\begin{itemize}
    \item a finitely generated $\mf{S}$-module $M$ anihilated by a power of $p$, and of projective dimension $\le 1$,
    \item maps $\varphi:M \to M$ and $\psi:M \to M$ which are $\sigma^{-1}$ and $\sigma$-linear respectively, where $\sigma:\mf{S} \to \mf{S}$ is some appropriate lift of the Frobenius, satisfying
    \[\psi \circ \varphi = \varphi \circ \psi = E.\]
\end{itemize}
A Breuil module is \emph{nilpotent} if $\varphi$ or $\psi$ is nilpotent modulo the maximal ideal of $\mf{S}$. When $R = \Spec \, k$, Breuil modules relative to $W(k) \to k$ recover the more familiar notion of Dieudonné modules. In our setting we consider the case $r = 2$ and the map
\[\mf{S} \to R \simeq \mf{S} / (h-p)\]
is the quotient map. In the proof of \cite[Theorem 28 (i)]{vasiu2010purity}, the authors define an epimorphism $H^t \to D^t$ by constructing nilpotent Breuil modules and an appropriate morphism between them. In particular, the Breuil module for $D$ has underlying module $\mf{S} / p \mf{S}$. This implies that $D$ has order $p$. To see this, observe that it suffices to prove it for the fiber of $D$ above the closed point $s: \Spec \, k \to \Spec \, R$. The Dieudonné module of $D_s$ is the base change
\[\mf{S} / p \mf{S} \otimes_{\mf{S}} W(k) \simeq k\]
which is a $W(k)$-module of length 1. Thus $D_s$ has order $p$. The same argument shows that $H$ has order $p^3$, although we make no use of it.

\medskip

\textbf{Step 2.} We argue that $H$ embeds in a principally polarized abelian scheme. This follows from the following theorem of Raynaud.

\begin{fact}[\cite{berthelotTheorieDieudonneCristalline1982}, Théorème 3.1.1]
   Let $S$ be a noetherian local scheme, $G$ a finite flat group scheme over $S$. Then there exists a principally polarized abelian scheme $B \to S$ in which $G$ embeds as a closed subgroup scheme.
\end{fact}

In the reference $B \to S$ is only claimed to be projective with no mention of polarization. In the proof, $B$ is constructed as the Jacobian of a relative smooth curve over $S$, hence it is principally polarized.

\medskip

\textbf{Step 3.} Let $A_U \coloneq B_U / D_U$ as in \Cref{lem:setup for abelian scheme not extending}; we argue that the dual $A_U^{\vee} \to U$ is a degree-$p^2$ polarized abelian scheme that does not extend to $\Spec \, R$. Consider the dual isogeny
\[q^{\vee}:A_U^{\vee} \to B_U^{\vee}\]
to the quotient $q:B_U \to A_U$. A priori this is a morphism of algebraic spaces over $\Spec \, R$, but $B_U$ is principally polarized, hence $B_U^{\vee}$ is a scheme, and so is its cover $A_U^{\vee}$. As $q$ and $q^{\vee}$ have same degree, which is the order of $D_U$, we find that the pullback of the principal polarization on $B_U^{\vee}$ through $q^{\vee}$ defines a degree-$p^2$ polarization on $A_U^{\vee}$. If $A_U^{\vee}$ were to extend to $\Spec \, R$, its dual would extend $A_U$, which is impossible according to \Cref{lem:setup for abelian scheme not extending}.

\medskip

It remains to deduce a principally polarized example. Let $A_U$ denote the abelian scheme from above, the Zarhin trick states that $(A_U \times A_U^{\vee})^4\to U$ is principally polarized. The following lemma shows that it does not extend over $\Spec \, R$. This concludes the proof of \Cref{prop:ab scheme over formally smooth base}.

\begin{lem}
    Let $W$ be a regular scheme, and $V \subseteq W$ a dense open subscheme. Let $A_V,B_V$ be abelian schemes over $V$. If $C_V \coloneq A_V \times B_V$ extends to an abelian scheme $C$ over $W$, then so do $A_V$ and $B_V$.
\end{lem}
\begin{proof}
    Let $p_U,q_u$ be the idempotent morphisms from $C_V \to C_V$ with images $A_V$ and $B_V$ respectively. As all these abelian schemes are regular, the idempotents extend to idempotent morphisms on $C$ (see \cite[8.4, Corollary 6]{boschNeronModels1990a}). Then the images (or kernels) of $p$ and $q$ define the sought-after abelian schemes.
\end{proof}

\subsection{Abelian schemes with pointwise but not global good reduction}\label{subsec:abelian schemes with PGR not GGR}
Consider $F$ and $N$ as in the previous paragraph. We first argue that \Cref{prop:ab scheme over formally smooth base} also holds when $R$ is a smooth local $\Ocal_F$-algebra.

\begin{prop}\label{prop:Vasiu finite type}
    There exists a smooth local $\Ocal_F$-algebra $R$ with maximal ideal $\mf m$ and an abelian scheme $A_U \to U \coloneq \Spec \, R \setminus \{\mf m\}$, together with a degree-$p^2$ (resp. principal) polarization $\lambda:A_U \to A_U^{\vee}$, such that $A_U$ does not extend to an abelian scheme over $\Spec \, R$.

    \noindent Moreover, we can arrange for $A_U \to U$ to have any prime-to-$p$ level structure.
\end{prop}
\begin{proof}
    We denote by $\tilde R$ the finite local étale $\Ocal_F[[T]]$-algebra with maximal ideal $\tilde{\mf m}$, by $\tilde{U} \coloneq \Spec \, \tilde{R} \setminus \{\tilde{\mf m}\}$ the open subset and by $f:\tilde{A}_U \to \tilde{U}$ an abelian scheme provided by \Cref{prop:ab scheme over formally smooth base}. 
    We wish to descend all of this data over a smooth $\Ocal_F$-algebra.  

    We first argue that $\Ocal_F \to \tilde{R}$ is a regular ring homomorphism, i.e. that $\Spec \, \tilde{R} \to \Spec \, \Ocal_F$ is flat with geometrically regular fibers. To do so we may assume $\tilde{R} = \Ocal_F[[T]]$.
    As $\tilde{R}$ is torsion-free over the DVR $\Ocal_F$, the morphism is flat.
    It is clear that $k[[T]]$ is geometrically regular, where $k$ is the residue field of $\Ocal_F$. 
    It remains to show that $W = \Spec (\Ocal_F[[T]] \otimes_{\Ocal_F} F) = \Spec (\Ocal_F[[T]][1/p])$ is geometrically regular, i.e. regular since $F$ is of characteristic zero. 
    As $\tilde{R}$ is a UFD, so is its localization $\Ocal_F[[T]][1/p]$. Thus $W$ is a normal noetherian $F$-scheme of dimension 1, hence it is regular, as desired.

    We may now apply Popescu's theorem, see \cite[\href{https://stacks.math.columbia.edu/tag/07GC}{Tag 07GC}]{stacks-project}: there is a cofiltered system $(R_i)$ of smooth $\Ocal_F$-algebras such that $\tilde{R} = \varinjlim R_i$. All maps between the $R_i$ and $\tilde{R}$ factor through local and integral rings, so we may assume that the $R_i$ are local integral with maximal ideal $\mf{m}_i$, and set $U_i \coloneq \Spec \, R_i \setminus \{\mf{m}_i\}$. 
    We let $p_i:\tilde{U} \to U_i$ be the maps induced from the projective system. 
    We wish to descend $f$ to an abelian scheme over some $U_i$, along with the extra data.

    Since the $U_i$ are Noetherian, by \cite[\href{https://stacks.math.columbia.edu/tag/0CNL}{Tag 0CNL}]{stacks-project} for large enough $i$ there is a smooth and proper morphism $f_i:A_i \to    U_i$ descending $f$. Using \cite[\href{https://stacks.math.columbia.edu/tag/0CNR}{Tag 0CNR}]{stacks-project} we can assume that the group scheme structure, polarization and level structure descend to $f_i$, thus $f_i$ defines an abelian scheme. Let $A_U=A_i$ and $U = U_i$ for this appropriate choice of $i$.

    If $A_U \to U$ were to extend to an abelian scheme over $\Spec \, R$, its pullback to $\Spec \, \tilde{R}$ would extend $\tilde{A}_U \to \tilde{U}$, contradicting the definition of $\tilde{A}_U$.
\end{proof}

\begin{cor}\label{cor:counterexample PGR not GGR}
    With notations from \Cref{prop:Vasiu finite type}, set $\Xcal = \Spec \, R$ which is a smooth model over $\Ocal_F$ of its generic fiber $\iota \colon X \hookrightarrow \Xcal$, and set $A_X \coloneq \iota^* A_U$. Then $A_X \to X$ is a degree-$p^2$ (resp. principally) polarized abelian scheme having pointwise good reduction but not global good reduction as an abelian scheme.

    Moreover, we can arrange for $A_X \to X$ to have any prime-to-$p$ level structure.
\end{cor}
\begin{proof}
    We argue that the polarized abelian scheme $A_X \to X$ does not have global good reduction.
    Assume $A_X \to X$ extends to an abelian scheme $A \to \Xcal$, we argue that this also extends $A_U \to U$, contradicting \Cref{prop:Vasiu finite type}. The polarization on $A_X$ extends to a polarization on $A$ due to \Cref{fact:polarization extends} below. Let $j\colon U \to \Xcal$ be the inclusion of open set, we find that the polarized abelian schemes $A_U \to U$ and $j^* A \to U$ have isomorphic generic fibers. Separatedness of the stack of polarized abelian schemes implies that they are isomorphic (see \Cref{prop:morphisms isomorphic over dense subset} below).
    
    It remains to check that $A_X \to X$ has pointwise good reduction. Let $\l \neq p$ be a prime number, the local system $\L = R^1f_* \Ql$ over $U$ corresponds to an $\l$-adic representation of the étale fundamental group $\pi_1(U)$. As $\Xcal \setminus U$ has codimension at least 2 in $\Xcal$, Zariski-Nagata purity implies that $\L$ extends to a local system on $\Xcal$. In particular, for any $x \in |\Xcal|_{\eta}^{\mrm{int}}$ the local system $\L_x$ is unramified, which by \cite{serre1971good} means that $(A_X)_x$ has good reduction.
\end{proof}

\begin{fact}[\cite{milne1992points}, Proposition 2.14]\label{fact:polarization extends}
        Let $A \to \Xcal$ be an abelian scheme on a normal scheme $\Xcal$, let $\lambda:A_X \to A_X^{\vee}$ be a polarization on $A_X$ where $X \subseteq \Xcal$ is a dense open subscheme. Then $\lambda$ extends to a polarization of $A$.
\end{fact}

For this next proposition we deal with Deligne-Mumford stacks over the pro-étale site over a base scheme $S$. The reader is refered to \Cref{subsec:alg stacks on the big pro-étale site}, and more precisely \Cref{dfn:algebraic and DM stacks}, for the meaning of these objects.

\begin{prop}\label{prop:morphisms isomorphic over dense subset}
    Let $\Xscr$ be a separated Deligne-Mumford stack over some base scheme $S$. Let $Y$ be a locally Noetherian, normal $S$-scheme and $U$ a dense open subscheme of $Y$. Let $f,g:Y \to \Xscr$ be two morphisms such that $f_{|U} \simeq g_{|U}$. Then $f \simeq g$.
\end{prop}
\begin{proof}
    Let $I = \u{\Isom}(f,g)$ the algebraic stack obtained by the pullback of the diagonal $\Delta$
     \[
     \begin{tikzcd}
    I \arrow[r, "\theta"] \arrow[d]     & Y \arrow[d, "f \times g"] \\
    \Xscr \arrow[r, "\Delta"] & \Xscr \times_S \Xscr.     
    \end{tikzcd}
    \]
    By assumption, $\Delta$, and hence also $\theta$, is finite, so by \cite[Théorème A.2]{laumonChampsAlgebriques2000} $I$ is a scheme. The assumption $f_{|U} \simeq g_{|U}$ means that $\theta$ has a section over $U$ $s:U \to I_{|U}$ which is an open immersion by \cite[\href{https://stacks.math.columbia.edu/tag/024T}{Tag 024T}]{stacks-project}. Let $Z$ denote the scheme-theoretic closure of $s(U)$ in $I$. We want to show that $\theta_{|Z}:Z \to Y$ is an isomorphism. Working componentwise, we can assume that $Y$, and so $U$ and $Z$, are irreducible. Since $\theta(Z)$ contains $U$, $\theta_{|Z}$ is dominant. By \cite[\href{https://stacks.math.columbia.edu/tag/0356}{Tag 0356}]{stacks-project} $Z$ is reduced. Thus $\theta_{|Z}:Z \to Y$ is a finite birational morphism between integral schemes, with $Y$ normal: by \cite[\href{https://stacks.math.columbia.edu/tag/0AB1}{Tag 0AB1}]{stacks-project} it is an isomorphism.
\end{proof}

\begin{rk}
    The normality assumption on $Y$ in \Cref{prop:morphisms isomorphic over dense subset} cannot be dropped. Over some base field $k$, consider $Y$ to be $\P^1$ with $0$ and $\infty$ glued together, let $U$ be $Y$ minus the singular point, and set $\Xscr = \BB \Z / 2\Z$. Let $Y_1$ be two disjoint copies of $Y$ and $Y_2$ be two copies of $\P^1$ glued at zero and glued at infinity. Then there are obvious projection maps $Y_i \to Y$, $i \in \{1,2\}$, which are double étale covers, and which both restrict to the trivial double cover over $U$. In this way, we find two non-isomorphic maps $Y \to \Xscr$ which are isomorphic over the dense open subset $U$. 
\end{rk}

In \Cref{sec:extension properties for good reduction} we will apply \Cref{prop:morphisms isomorphic over dense subset} repeatedly as follows: for $\Xcal/ \Ocal_F$ a regular scheme, and $\Sscr / \Ocal_F$ a separated Deligne-Mumford stack, two $\Ocal_F$ morphisms $f,g:\Xcal \to \Sscr$ that are isomorphic over $\Spec \, F$ are isomorphic.

\section{A pointwise criterion for extending local systems}\label{sec:pointwise criterion}

Here the residue characteristic of $\Ocal_F$ is $p>0$. In this section, we recall a theorem of Cadoret-Tamagawa in \cite{cadoret2025pointwise} to be used in \Cref{sec:extension properties for good reduction}, giving a pointwise criterion to check whether a local system over a smooth variety $X/F$ extends to a smooth model $\Xcal / \Ocal_F$.

\medskip

Fix $\Xcal/\Ocal_F$ a smooth, separated scheme of finite type, let $X$ denote its generic fiber. Consider a finite Galois\footnote{By this, we mean that the pre-image of any connected component of $X$ defines a finite Galois cover in the usual sense (see \cite[\href{https://stacks.math.columbia.edu/tag/03SF}{Tag 03SF}]{stacks-project}).} cover $\psi:Y \to X$, set $\Ycal$ the relative normalization of $\Xcal$ in $Y \to X \to \Xcal$. We begin with a pointwise criterion to check whether $\Ycal \to \Xcal$ is again étale. For any closed point $x \in |X|$ we say that $\psi$ is \emph{unramified at $x$} if for any (equiv. all) $y \in |Y|$ mapping to $x$, the induced morphism $\Spec \, \Ocal_{F(y)} \to \Spec \, \Ocal_{F(x)}$ is étale (i.e. $F(y)/F(x)$ is unramified). For a subset $\Sigma \subseteq |X|$ we say that $\psi$ is \emph{unramified at $\Sigma$} if it is unramified at every $x \in \Sigma$. 

\begin{thm}[\cite{cadoret2025pointwise}, Theorem 12 (U)]\label{thm:pointwise criterion for étale}
    The following are equivalent:
    \begin{itemize}
        \item[(i)] The map $\Ycal \to \Xcal$ is étale, and
        \item[(ii)] The map $\psi$ is unramified at $|\Xcal|_{\eta}^{\mrm{int}}$ (recall \Cref{dfn:integral locus}).
    \end{itemize}
\end{thm}

The reference has the added assumption that $X$ be geometrically connected. Let us show that it can be removed. We only need to show (ii)$\implies$(i).
First, as $\Xcal / \Ocal_F$ is smooth, we can write
\[\Xcal = \bigsqcup_i \Xcal_i\]
where each $\Xcal_i$ is the closure in $\Xcal$ of a component of $X$. Arguing with each $\Xcal_i$ individually, we may assume that $X$ is connected. Furthermore we may assume that the special fiber of $\Xcal$ is nonempty, otherwise the theorem is vacuously true.

According to \Cref{lem:geom connected after unram base change} below, for some unramified extension $F'/F$ the components of $X_{F'}$ are geometrically connected. As relative normalization commutes with étale base-change (see \cite[\href{https://stacks.math.columbia.edu/tag/0ABP}{Tag 0ABP}]{stacks-project}) and étaleness is étale-local, it suffices to show that $\Ycal_{\Ocal_{F'}} \to \Xcal_{\Ocal_{F'}}$ is étale. The reference ensures that this is true above every component of $\Xcal_{\Ocal_{F'}}$, and we are done.

\begin{lem}\label{lem:geom connected after unram base change}
    Let $\Xcal$ be a smooth scheme of finite type over $\Ocal_F$, with connected generic fiber $X/F$ and nonempty special fiber $\Xcal_s$. Then for some finite unramified field extension $F'/F$, the components of the base-change $X_{F'}$ are geometrically connected.
\end{lem}
\begin{proof}
    Let $F'/F$ be a large enough finite unramified extension so that the special fiber $\Xcal_{\Ocal_{F'},s}$ of the base-change has geometrically connected components. We want to show that the components of $X_{F'}$ are also geometrically connected. Let $\Xcal'$ denote the closure in $\Xcal_{\Ocal_{F'}}$ of a component of $X_{F'}$; it is an irreducible component of $\Xcal_{\Ocal_{F'}}$, hence a connected component as $\Xcal_{\Ocal_{F'}}$ is smooth. The special fiber $\Xcal'_s$ is nonempty: this is because all components of $\Xcal_{\Ocal_{F'}}$ are isomorphic over $\Ocal_F$, and at least one of them has a nonempty special fiber since $\Xcal_s \neq \varnothing$. Moreover $\Xcal'_s$ has geometrically connected components, so the result follows from \Cref{lem:not geom connected from generic to special}.
\end{proof}

\begin{lem}\label{lem:not geom connected from generic to special}
    Let $\Xcal$ be a smooth scheme of finite type over $\Ocal_F$, with connected generic fiber $X/F$ and nonempty special fiber $\Xcal_s$. If the components of the special fiber $\Xcal_s$ are geometrically connected, then $X$ is geometrically connected.
\end{lem}
\begin{proof}
    Suppose $X$ is not geometrically connected, we show that neither are the components of $\Xcal_s$. For some finite extension $F'/F$ we have a non-trivial decomposition into connected (automatically irreducible) components of the base-change
    \[X_{F'} = \bigsqcup_i X^{(i)}.\]
    Let $\Xcal^{(i)}$ denote the closure of $X^{(i)}$ in $\Xcal_{\Ocal_{F'}}$. The $\Xcal^{(i)}$ are pairwise disjoint as $\Xcal_{\Ocal_{F'}}$ is smooth. Thus we find a decomposition of the special fiber
    \[(\Xcal_{s})_{k_{F'}} = \Xcal_{\Ocal_{F'},s} = \bigsqcup_i \Xcal^{(i)}_s\]
    with $k_{F'}$ the residue field of $F'$, and $\Xcal^{(i)}_s$ the special fiber of $\Xcal^{(i)}$ (which may not be connected in general). None of the $\Xcal^{(i)}_s$ are empty: this is because all components $\Xcal^{(i)}$ are isomorphic over $\Ocal_F$, and at least one of them has a nonempty special fiber since $\Xcal_s \neq \varnothing$. Now we claim that none of the $\Xcal^{(i)}_s$ descend to a component of $\Xcal_s$. It is enough to show that the descent datum
    \[\varphi_s: \Xcal_{\Ocal_{F'},s} \times_{\Xcal_s} \Xcal_{\Ocal_{F'},s} \xrightarrow{\sim} \Xcal_{\Ocal_{F'},s} \times_{\Xcal_s} \Xcal_{\Ocal_{F'},s},\]
    when restricted to an open and closed subscheme of the form $\Xcal^{(i)}_{\Ocal_{F'},s} \times_{\Xcal_s} \Xcal^{(i)}_{\Ocal_{F'},s}$, does not map into the same subscheme. Note that $\varphi_s$ is obtained by base-change from the datum
    \[\varphi: \Xcal_{\Ocal_{F'}} \times_{\Xcal} \Xcal_{\Ocal_{F'}} \xrightarrow{\sim} \Xcal_{\Ocal_{F'}} \times_{\Xcal} \Xcal_{\Ocal_{F'}},\]
    so that it suffices to show that the restriction of $\varphi$ to the subscheme $\Xcal^{(i)}_{\Ocal_{F'}} \times_{\Xcal} \Xcal^{(i)}_{\Ocal_{F'}}$ does not map into the same subscheme. If this were the case, this would imply that the datum
    \[\varphi_{\eta}: X_{F'} \times_{X} X_{F'} \xrightarrow{\sim} X_{F'} \times_{X} X_{F'}\]
    maps $X^{(i)} \times_X X^{(i)}$ into itself, so that $X^{(i)}$ descends to a component of $X$. But $X$ is connected, so $X_{F'} = X^{(i)}$, contradicting our assumption on $F'$.
\end{proof}

Let $\Gcal$ be a sheaf of groups, consider $\Xcal, X$ as above. Let $P$ be a $\Gcal$-torsor on $X$ (see \Cref{dfn:Q-torsor}). For any closed point $x \in |X|$, we say that \emph{$P$ is unramified at $x$} if the $\Gcal$-torsor $P_x$ over $\Spec \, F(x)$ extends to a $\Gcal$-torsor over $\Spec \, \Ocal_{F(x)}$. For a subset $\Sigma \subseteq |X|$ we say that $P$ is \emph{unramified at $\Sigma$} if it is unramified at every $x \in \Sigma$. 

\begin{cor}\label{cor:pointwise criterion for torsor}
    Let $Q$ be a profinite group, suppose $\Gcal = \Fcal_Q$. The following are equivalent:
    \begin{enumerate}
        \item[(i)] $P$ extends to a $G$-torsor over $\Xcal$, and
        \item[(ii)] $P$ is unramified at $|\Xcal|_{\eta}^{\mrm{int}}$.
    \end{enumerate}
\end{cor}
\begin{proof}
    We only need to show (ii)$\implies$(i). Let us first assume that $Q$ is finite. The torsor $P$ is represented by a scheme whose components are finite Galois covers of $X$. Assumption $(ii)$ implies that the morphism $P \to X$ is unramified at $|\Xcal|_{\eta}^{\mrm{int}}$, hence by \Cref{thm:pointwise criterion for étale} the relative normalization $\Pcal$ of $\Xcal$ in $P \to X \to \Xcal$ defines an étale map $\Pcal \to \Xcal$. It remains to show that this is a $Q$-torsor. Normality of $\Xcal$ implies that the restriction from finite étale covers of $\Xcal$ to those of $X$ is fully faithful, so that the isomorphism
    \[\Fcal_Q \times_X P \longrightarrow P \times_X P\]
    extends to an isomorphism $\Fcal_Q \times_{\Xcal} \Pcal \to \Pcal \times_{\Xcal} \Pcal$, and we are done.

    For the general case, write $Q = \varprojlim_i Q/K_i$ where each $K_i$ is an open subgroup of $Q$. For each $i$ write $Q_i \coloneq Q/K_i$, and $P_i \coloneq P / K_i$, so that $P_i$ is a $Q_i$-torsor over $X$, unramified at $|\Xcal|_{\eta}^{\mrm{int}}$. From the finite group case, $P_i$ extends to a $Q_i$-torsor $\Pcal_i$ over $\Xcal$, and the quotient morphisms $P_i \to P_j$ extend to quotient morphisms $\Pcal_i \to \Pcal_j$. Now, set
    \[\Pcal \coloneq \varprojlim_i \Pcal_i\]
    which is a $Q$-torsor over $\Xcal$ extending $P$. 
\end{proof}

Let $R$ be a topological ring, $V$ a free $R$-module of finite rank. Let $\L$ be an $R$-local system with fiber $V$ over $X$ (see \Cref{dfn:R-local systems}). For any closed point $x \in |X|$, we say that \emph{$\L$ is unramified at $x$} if the local system $\L_x$ over $\Spec \, F(x)$ extends to a local system over $\Spec \, \Ocal_{F(x)}$. For a subset $\Sigma \subseteq |X|$ we say that $\L$ is \emph{unramified at $\Sigma$} if it is unramified at every $x \in \Sigma$. 

\begin{cor}[\cite{cadoret2025pointwise}, Theorem 1 (U-1)]\label{cor:pointwise criterion for local system}
    Suppose $R$ is profinite. The following are equivalent:
        \begin{enumerate}
        \item[(i)] $\L$ extends to a local system over $\Xcal$, and
        \item[(ii)] $\L$ is unramified at $|\Xcal|_{\eta}^{\mrm{int}}$.
    \end{enumerate}
\end{cor}
\begin{proof}
    By \Cref{fact:loc sys equiv GL(V) torsor} the statement is equivalent to that of \Cref{cor:pointwise criterion for torsor} with $Q = \GL(V)$ profinite.
\end{proof}

\begin{rk}
    In \cite{cadoret2025pointwise}, \emph{local systems} are defined over a \emph{connected} scheme $T$ as morphism $\pi_1(T) \to Q$ where $Q$ is some profinite group. With the terminology of \Cref{sec:can loc sys on schemes and stacks} such morphisms define $Q$-torsors on $T$, which induce a local system only after a choice of embedding $Q \hookrightarrow \GL(V)$.
\end{rk}

\section{Extension properties for good reduction}\label{sec:extension properties for good reduction}

Here the residue characteristic of $\Ocal_F$ is $p>0$. In this section we will deal with schemes and stacks over $\Ocal_F$: we use the subscripts $\eta$ and $s$ to denote the generic and special fibers of these objects respectively.

\begin{dfn}
    Let $\Xcal, \Sscr / \Ocal_F$ be stacks. A \emph{generic morphism} $f:\Xcal \rightsquigarrow \Sscr$ is a representable morphism of stacks between the generic fibers $f_{\eta}:\Xcal_{\eta} \to \Sscr_{\eta}$. 
\end{dfn}

A separated Deligne-Mumford stack $\Sscr/\Ocal_F$ such that certain types of generic morphisms extend to actual morphisms is said to have an \emph{extension property}. In \Cref{subsec:L and PGR extension} we define two such properties, L-extension and PGR-extension, and argue that the former implies the latter. Unlike L-extension, PGR-extension is preserved under pullback by any étale map, as is shown in \Cref{subsec:pullback PGR}; this will be of use in \Cref{sec:applications}. Lastly, we define smooth-extension in \Cref{subsec:smooth extension} and argue that it implies L-extension in the following sense: if $\Sscr_{\infty} = \varprojlim_K \Sscr_K$ is the limit of an appropriate tower of stacks and $\Sscr_{\infty}$ has smooth-extension, then each $\Sscr_K$ is endowed with a canonical local system $\L$ such that $(\Sscr_K,\L)$ has smooth extension. The question of finding stacks $\Sscr_{\infty}$ with smooth-extension will be addressed in \Cref{sec:extension for Shimura} where we will apply some of the lemmas of the second half of \Cref{subsec:smooth extension}.

\subsection{L-extension and PRG-extension}\label{subsec:L and PGR extension}

\begin{dfn}\label{dfn:L-extension property}
    Let $\Sscr / \Ocal_F$ be a separated Deligne-Mumford stack, and $R$ a profinite ring. Let $\mbb L$ be an $R$-local system on $\Sscr$ (see \Cref{dfn:R-local systems}, \Cref{dfn:loc sys and torsor on stack}). We say that the pair $(\Sscr, \mbb L)$ has \emph{L-extension} if any generic morphism $f:\Xcal \rightsquigarrow \Sscr$ from a smooth scheme $\Xcal / \Ocal_F$, satisfying that the pullback $f_{\eta}^{-1} \L$ to $\Xcal_{\eta}$ extends to a local system over $\Xcal$, automatically and uniquely extends to an $\Ocal_F$-morphism $\Xcal \to \Sscr$.
\end{dfn}

The uniqueness in \Cref{dfn:L-extension property} is automatic by \Cref{prop:morphisms isomorphic over dense subset}. This will also be the case for all other extension properties considered.

\begin{lem}\label{lem:atlas has L-extension}
    Let $\Sscr/\Ocal_F$ be a stack, $\L$ a local system on $\Sscr$ such that $(\Sscr, \L)$ has $\L$-extension. Consider $\pi:\Sscr' \to \Sscr$ a surjective étale morphism over $\Sscr / \Ocal_F$. Then $(\Sscr',\pi^* \L)$ has $\L$-extension.
\end{lem}
\begin{proof}
    Let $\Xcal / \Ocal_F$ be a smooth scheme and $f:\Xcal \rightsquigarrow \Sscr'$ a generic morphism. Using $\L$-extension, the composition $\Xcal \rightsquigarrow \Sscr' \to \Sscr$ extends uniquely to $h:\Xcal \to \Sscr$.
    The situation is summarized in the following commutative diagram
    \[
    \begin{tikzcd}
    \Xcal_{\eta} \arrow[r, "f_{\eta}"] \arrow[d, hook]        & \Sscr'_{\eta} \arrow[r] \arrow[d, hook] & \Sscr_{\eta} \arrow[d, hook] \\
    \Xcal \arrow[r, "?", dashed] \arrow[rr, "h"', bend right] & \Sscr' \arrow[r]                 & \Sscr.                
    \end{tikzcd}
    \]
    Let $\Xcal' \coloneq \Xcal \times_{\Sscr} \Sscr'$ which is an étale cover of $\Xcal$. We wish to show that $h': \Xcal' \to \Sscr'$ descends to a morphism $\Xcal \to \Sscr'$, in other words we need to show that both composition arrows 
    \[\Xcal' \times_{\Xcal} \Xcal' \rightrightarrows \Xcal' \to \Sscr'\]
    coincide. Since this holds over the generic fibers, the result follows from \Cref{prop:morphisms isomorphic over dense subset}.
\end{proof}

In our applications however, we will be considering non-surjective étale morphisms $\Ucal \to \Sscr$: it is no longer true that $\Ucal$ inherits L-extension. Fortunately, L-extension implies a weaker property called PGR-extension, which we define below. In the next paragraph we show that PGR-extension pulls back under étale maps.

\begin{dfn}\label{dfn:PGR-morphism}

    Let $\Xcal,\Mscr / \Ocal_F$ be stacks. A \emph{PGR-morphism} (as in Pointwise Good Reduction) is a generic morphism $f:\Xcal \rightsquigarrow \Mscr$ such that $f_{\eta}(|\Xcal|_{\eta}^{\mrm{int}}) \subseteq |\Mscr|_{\eta}^{\mrm{int}}$. In other words, for any finite extension $E/F$ and any $\Ocal_E$-point of $\Xcal$, the image by $f_{\eta}$ of the induced $E$-point $x$ of $\Xcal_{\eta}$ extends to an $\Ocal_{E'}$-point of $\Sscr$ for some finite extension $E'/E$:
    \[
    \begin{tikzcd}
                                                                           & {\Spec \, E} \arrow[r,"x"] \arrow[d] & \Xcal_{\eta} \arrow[d] \arrow[r, "f_{\eta}"] & \Sscr_{\eta} \arrow[d] \\
    {\Spec \, \Ocal_{E'}} \arrow[r, "\exists"] \arrow[rrr, "\exists", dashed, bend right] & {\Spec \, \Ocal_{E}} \arrow[r]   & \Xcal                            & \Sscr.                 
    \end{tikzcd}
    \]
\end{dfn}

The restriction of an $\Ocal_F$-morphism $\Xcal \to \Mscr$ to the generic fibers is a PGR-morphism.

\begin{rk}\label{rk:PGR implies PGR morphism}
    The acronym PGR is motivated by the case, of interest to us, where $\Mscr/\Ocal_F$ is a moduli stack for some families of smooth and proper algebraic spaces over $\Ocal_F$. In this case, a PGR-morphism $f:\Xcal \rightsquigarrow \Mscr$ corresponds to a smooth proper family $Y \to \Xcal_{\eta}$ with pointwise good reduction as in \Cref{dfn:pointwise good reduction}.
\end{rk}

\begin{dfn}\label{dfn:PGR-extension property}
    Let $\Mscr/\Ocal_F$ be a separated Deligne-Mumford stack. We say that it has \emph{PGR-extension} if every PGR-morphism $f:\Xcal \rightsquigarrow \Mscr$ from a smooth scheme $\Xcal / \Ocal_F$ automatically and uniquely extends to an $\Ocal_F$-morphism $\Xcal \to \Mscr$.
\end{dfn}

\begin{rk}\label{rk:PGR extensions means PGR implies GGR}
    Assuming $\Mscr/\Ocal_F$ is a moduli stack for some smooth and proper families of schemes over $\Ocal_F$, having PGR-extension exactly states that such a family with pointwise good reduction has global good reduction. 
\end{rk}

\begin{prop}\label{prop:L-extension implies PGR extension}
    Let $\Sscr/\Ocal_F$ be a separated Deligne-Mumford stack with a local system $\L$. Assume the pair $(\Sscr, \mbb L)$ has L-extension as in \Cref{dfn:L-extension property}. Then $\Sscr$ has PGR-extension.
\end{prop}
\begin{proof}
    Let $\Xcal / \Ocal_F$ be as in \Cref{dfn:PGR-extension property}, consider $f:\Xcal \rightsquigarrow \Sscr$ a PGR-morphism. We want to show that $f$ extends to a morphism $\Xcal \to \Sscr$ over $\Ocal_F$. It suffices to show that $f^{-1} \L$ extends to a local system on $\Xcal$ and apply L-extension.
    
    Let $x \in |\Xcal|_{\eta}^{\mrm{int}}$, then $f(x) \in |\Sscr|_{\eta}^{\mrm{int}}$ is in the image of some $\overline{\Ocal_F}$-point of $\Sscr$. In particular, the Galois representation defined by the pullback of $\L$ to $f(x)$ is unramified, hence so is its pullback to $x$. Thus $f^{-1} \L$ is unramified at every point of $|\Xcal|_{\eta}^{\mrm{int}}$; by \Cref{cor:pointwise criterion for local system} it extends to a local system over $\Xcal$.
\end{proof}

\subsection{Pulling back PGR-extension along étale maps}\label{subsec:pullback PGR}

We aim to show the following statement.
\begin{thm}\label{thm:PGR extension pulls back for stacks}
    Let $\Sscr/\Ocal_F$ be a separated Deligne-Mumford stack with PGR-extension. Let $\Ucal/\Ocal_F$ be an algebraic stack étale over $\Sscr$. Then $\Ucal$ has PGR-extension.
\end{thm}

We first consider the case when $\Sscr$ is a scheme.

\begin{prop}\label{prop:PGR extension pulls back for schemes}
    Let $\Sscr/\Ocal_F$ be a separated scheme with PGR-extension. Let $\Ucal/\Ocal_F$ be a scheme étale over $\Sscr$. Then $\Ucal$ has PGR-extension.
\end{prop}
\begin{proof}
    Let $\Xcal / \Ocal_F$ be as in \Cref{dfn:PGR-morphism}, consider $f:\Xcal \rightsquigarrow \Ucal$ a PGR-morphism. We want to show that $f$ extends to a morphism $\Xcal \to \Ucal$ over $\Ocal_F$. The composition $\Xcal_{\eta} \to \Ucal_{\eta} \to \Sscr_{\eta}$ is again PGR, hence extends to $h:\Xcal \to \Sscr$. 
    The situation is summarized in the following commutative diagram:
    \[
    \begin{tikzcd}
    \Xcal_{\eta} \arrow[r, "f_{\eta}"] \arrow[d, hook]                     & \Ucal_{\eta} \arrow[r, "g_{\eta}"] \arrow[d, hook] & \Sscr_{\eta} \arrow[d, hook] \\
    \Xcal \arrow[r, "?", dashed] \arrow[rr, "h"', bend right] & \Ucal \arrow[r, "g"]     & \Sscr.                
    \end{tikzcd}
    \]
    Let $\Gamma_{f_{\eta}} \hookrightarrow \Xcal_{\eta} \times_F \Ucal_{\eta}$ be the graph of $f_{\eta}$, let $Z$ denote its closure in $\Xcal \times_{\Ocal_F} \Ucal$. It suffices to show that the projection $Z \to \Xcal$ is an isomorphism, so that the extension of $f_{\eta}$ is provided by
    \[\tilde{f}:\Xcal \to Z \hookrightarrow \Xcal \times_{\Ocal_F} \Ucal \to \Ucal.\]
    Let $\Gamma_h \hookrightarrow \Xcal \times_{\Ocal_F} \Sscr$ denote the graph of $h$. Let $\tilde{Z}$ denote the pullback of $\Gamma_h \to \Sscr$ along $g$:
    \[\begin{tikzcd}
    \tilde{Z} \arrow[d] \arrow[r, hook] & \Xcal \times_{\Ocal_F} \Ucal \arrow[d, "\id \times g"] \arrow[r, "p_2"] & \Ucal \arrow[d, "g"] \\
    \Gamma_h \arrow[r, hook]            & \Xcal \times_{\Ocal_F} \Sscr \arrow[r, "p_1"]                                   & \Sscr.                       
    \end{tikzcd}
    \]
    By definition $\tilde{Z}$ is a closed subscheme of $\Xcal \times_{\Ocal_F} \Ucal$ containing $\Gamma_{f_{\eta}}$, hence $Z$ is a closed subscheme of $\tilde{Z}$. By \Cref{lem:closure of graph has surjective projection} below, the projection $Z \to \Xcal$ is surjective (here we use again that $f$ is PGR). The commutative diagram
    \[
    \begin{tikzcd}
    Z \arrow[d] \arrow[r, hook]                              & \Xcal \times_{\Ocal_F} \Ucal \arrow[d, "\id \times \tilde{g}"] \arrow[r, "p_1"] & \Xcal \arrow[d, "\id"] \\
    \Gamma_h \arrow[r, hook] \arrow[rr, "\sim"', bend right] & \Xcal \times_{\Ocal_F} \Sscr \arrow[r, "p_1"]                                   & \Xcal                 
    \end{tikzcd}
    \]
    shows that this projection factors through $Z \to \Gamma_h \xrightarrow{\sim} \Xcal$, thus $\theta:Z \to \Gamma_h$ is surjective. By \Cref{lem:restriction to irred comp is etale} below, $Z \to \Gamma_h$ is étale. By \cite[Proposition 18.2.8]{EGAIV} the number of points in the geometric fibers of $\theta$ is lower semicontinuous, bounded below by 1 by surjectivity, and equals 1 on the dense open subscheme $\Gamma_{f_{\eta}}$ (as $\Gamma_{f_{\eta}} \to \Gamma_{g_{\eta} \circ f_{\eta}}$ is an isomorphism); it is thus constant equal to 1. By \textit{loc.cit.}, $\theta$ is finite, and so by \cite[\href{https://stacks.math.columbia.edu/tag/0AB1}{Tag 0AB1}]{stacks-project} it is an isomorphism.
\end{proof}

\begin{lem}\label{lem:closure of graph has surjective projection}
    Let $\Xcal, \Ucal / \Ocal_F$ be schemes of finite type, with $\Xcal / \Ocal_F$ smooth, and $f:\Xcal \rightsquigarrow \Ucal$ a PGR-morphism with graph $\Gamma \hookrightarrow \Xcal_{\eta} \times_F \Ucal_{\eta}$. Let $Z$ be the closure of $\Gamma$ in $\Xcal \times_{\Ocal_F} \Ucal$. Then the projection $Z \to \Xcal$ is surjective.
\end{lem}
\begin{proof}
    We show that the set-theoretic image of $p:Z \to \Xcal$ contains all closed points. Then by Chevalley's theorem the image is a locally constructible subset of $\Xcal$ containing all closed points, thus it is the whole of $\Xcal$. Since $\Gamma \subseteq Z$, $\Xcal_{\eta}$ is contained in $p(Z)$, so it remains to show that the closed points of the special fiber $\Xcal_s$ of $\Xcal$ are attained. 

    Let $x_s \in \Xcal_s$ be a closed point. We use the smoothness of $\Xcal$: by the infinitesimal lifting criterion, $x_s$ lifts to an $\Ocal_E$-point $x: \Spec \, \Ocal_E \to \Xcal$ for some finite extension $E/F$. Let $x_0: \Spec \, E \to \Xcal_{\eta}$ denote the generic fiber of $x$. Let $y_0 = f_{\eta}(x_0)$, $y_0 \in |\Ucal|_{\eta}^{\mrm{\mrm{int}}}$ by the PGR assumption, hence $y_0$ is the generic fiber of some $\Ocal_{E'}$-point $y:\Spec \, \Ocal_{E'} \to \Ucal$ for $E' /E$ a finite extension. We may assume $E = E'$, so that the pair $(x,y)$ defines an $\Ocal_E$-point of the product $\Xcal \times_{\Ocal_F} \Ucal$. The situation is summarized in the following commutative diagram:
    \[
    \begin{tikzcd}
    \Spec \, E \arrow[r, "{(x_0,y_0)}"] \arrow[d, hook]                     & \Gamma_{f_{\eta}} \arrow[r,hook] \arrow[d, hook] & \Xcal_{\eta} \times_F \Ucal_{\eta} \arrow[d, hook] \\
    \Spec \, \Ocal_E \arrow[r, "?", dashed] \arrow[rr, "{(x,y)}",bend right] & Z \arrow[r, hook]     &\Xcal \times_{\Ocal_F} \Ucal.                
    \end{tikzcd}
    \]
    The factorization via the dashed arrow exists because of \Cref{lem:factoring morphisms through closed substacks} below.
    Now the composition morphism
    \[\Spec \, \Ocal_E \longrightarrow Z \longrightarrow \Xcal \times_{\Ocal_F} \Ucal \longrightarrow \Xcal\]
    gives back the $\Ocal_E$-point $x$, so $x_s$ lies in the image of $Z \to \Xcal$, as desired.
\end{proof}

\begin{lem}\label{lem:factoring morphisms through closed substacks}
    Let $\Xcal / \Ocal_F$ be a flat scheme, $\Sscr / \Ocal_F$ a Deligne-Mumford stack, and $\Zscr \hookrightarrow \Sscr$ a closed substack. Let $f:\Xcal \rightsquigarrow \Zscr$ be a generic morphism, such that the composition $\Xcal \rightsquigarrow \Sscr$ extends to $h:\Xcal \to \Sscr$. Then $h$ factors through $\Zscr$.
\end{lem}
\begin{proof}
    We may argue étale-locally and assume that $\Xcal, \Sscr$ and $\Zscr$ are affine schemes over $\Ocal_F$; denote by $A, B$ and $B/I$ the corresponding $\Ocal_F$-algebras. As $\Xcal / \Ocal_F$ is flat, $A$ is torsion-free. The situation is summarized by the following commutative diagram with solid arrows:
    \[
    \begin{tikzcd}
    A \otimes_{\Ocal_F} F & B/I \otimes_{\Ocal_F} F \arrow[l] & B \otimes_{\Ocal_F} F \arrow[l]                        \\
    A \arrow[u]           & B/I \arrow[u] \arrow[l, dashed, "?"']   & B. \arrow[u] \arrow[l, two heads] \arrow[ll, bend left]
    \end{tikzcd}
    \]
    As $A$ is torsion-free, the left-most vertical arrow is injective. A diagram chase shows that $I$ is in the kernel of $B \to A$, hence there exists a dashed-arrow making the diagram commute, as desired.
\end{proof}

\begin{lem}\label{lem:restriction to irred comp is etale}
    Let $\phi:\tilde{Z} \to \Gamma$ be an étale morphism of schemes with $\Gamma$ regular irreducible. Let $Z \hookrightarrow \tilde{Z}$ be an irreducible closed subscheme such that $\theta:Z \to \Gamma$ is again surjective. Then $\theta$ is étale. 
\end{lem}
\begin{proof}
    The surjectivity of $\theta$ ensures $\dim \tilde{Z} \ge \dim Z \ge \dim \Gamma = \dim \tilde{Z}$, so $Z$ is an irreducible component of $\tilde{Z}$.
    As $\phi$ is étale, $\tilde{Z}$ is also regular, hence its irreducible components are connected components: $Z \hookrightarrow \tilde{Z}$ is an open immersion. Hence $\theta: Z \hookrightarrow \tilde{Z} \to \Gamma$ is étale.
\end{proof}

\begin{proof}[Proof of \Cref{thm:PGR extension pulls back for stacks}]
    Let $\Xcal / \Ocal_F$ be as in \Cref{dfn:PGR-extension property}, and $f:\Xcal \rightsquigarrow \Ucal$ a PGR-morphism. We want to show that $f$ extends to a morphism $\Xcal \to \Ucal$ over $\Ocal_F$.

    Consider $\Sscr'/\Ocal_F$ a scheme which is an étale atlas over $\Sscr / \Ocal_F$. 
    By \Cref{lem:atlas has L-extension} (rather, its PGR analogue), $\Sscr'$ has PGR-extension. Let $\Xcal' / \Ocal_F$ and $\Ucal' / \Ocal_F$ be the algebraic spaces obtained by base change by $\Sscr' \to \Sscr$. Since $\Sscr' \to \Sscr$ is an étale presentation of a Deligne-Mumford stack with separated and quasi-compact diagonal, $\Xcal'$ and $\Ucal'$ are schemes (see \cite[Remark 3.1.5]{alper2024stacks}). By \Cref{prop:PGR extension pulls back for schemes}, $\Ucal'$ has PGR-extension, and since $f':\Xcal' \rightsquigarrow \Ucal'$ is again PGR, thus we find a unique morphism
    \[\tilde{f}':\Xcal' \to \Ucal'\]
    extending $f'$. Now it suffices to show that the composition $\Xcal' \to \Ucal' \to \Ucal$ descends to a morphism $\Xcal \to \Ucal$, in other words we need to show that both composition arrows
    \[\Xcal' \times_{\Xcal} \Xcal' \rightrightarrows \Xcal' \to \Ucal' \to \Ucal\]
    are isomorphic. This follows from \Cref{prop:morphisms isomorphic over dense subset}.
\end{proof}

We end this paragraph by showing that PGR-extension descends along finite torsors.

\begin{prop}\label{prop:PGR descends along finite torsors}
    Let $H$ be a finite group. Let $\Mscr/\Ocal_F$ be a separated Deligne-Mumford stack. Let $P:\Mscr \to \BB H$ denote an $H$-torsor on $\Mscr$, consider $\Ucal \to \Mscr$ the corresponding finite étale cover. 
    If $\Ucal$ has PGR-extension, then so does $\Mscr$.
\end{prop}
\begin{proof}
    Let $\Xcal / \Ocal_F$ be as in \Cref{dfn:PGR-extension property}, and $f:\Xcal \rightsquigarrow \Mscr$ a PGR-morphism. We want to show that $f$ extends to a morphism $\Xcal \to \Mscr$ over $\Ocal_F$.

    Let $\Ucal'_{\eta} \coloneq \Xcal_{\eta} \times_{\Mscr_{\eta}} \Ucal_{\eta}$. We will argue that $\Ucal'_{\eta}/F$ has a smooth model $\Ucal'/\Ocal_F$, and that there is an $H$-torsor $\Ucal' \to \Xcal$ extending $\Ucal'_{\eta} \to \Xcal_{\eta}$. The following commutative diagram summarizes the situation
    \[
    \begin{tikzcd}
    \Ucal'_{\eta} \arrow[rrd] \arrow[d] \arrow[r] & \Xcal_{\eta} \arrow[rrd] \arrow[d, "f_{\eta}"'] &                                    &                         \\
    \Ucal_{\eta} \arrow[rrd] \arrow[r]            & \Mscr_{\eta} \arrow[rrd]            & \Ucal' \arrow[r] \arrow[d, dashed, "?"] & \Xcal \arrow[d, dashed, "?"] \\
                                                &                                     & \Ucal \arrow[r]                    & \Mscr                  
    \end{tikzcd}
    \]
    It then follows from \Cref{lem:semi base change of PGR is PGR} below that the projection $\Ucal' \rightsquigarrow \Ucal$ is a PGR-morphism, hence extends\footnote{While the existence of $\Ucal$ holds even assuming $H$ profinite, here we need $H$ finite because $\Ucal'/\Ocal_F$ is required to be smooth (in particular, finite type) in \Cref{dfn:PGR-extension property}.} to $\Ucal' \to \Ucal$. It remains to show that the composition $\Ucal' \to \Ucal \to \Mscr$ descends to a morphism $\Xcal \to \Mscr$; this is done exactly as in the proof of \Cref{thm:PGR extension pulls back for stacks} using \Cref{prop:morphisms isomorphic over dense subset}.

    We construct $\Ucal'$ as follows. From \Cref{rk:universal Q torsor} we have the 2-cartesian diagram
    \[
    \begin{tikzcd}
        \Ucal'_{\eta} \arrow[r] \arrow[d] & \Xcal_{\eta} \arrow[d] \\
        \mrm{pt} \arrow[r]            & \BB H.              
    \end{tikzcd}
    \]
    If we can show that $\Xcal_{\eta} \to \BB H$ extends to $\Xcal \to \BB H$, we then define $\Ucal'$ as the fiber product
    \[\begin{tikzcd}
        \Ucal' \arrow[r] \arrow[d] & \Xcal \arrow[d] \\
        \mrm{pt} \arrow[r]            & \BB H              
    \end{tikzcd}\]
    (i.e. the total space of the $H$-torsor on $\Xcal$, which is a scheme). Then $\Ucal' \to \Xcal$ has all required properties. To extend $\Xcal_{\eta} \to \BB H$, we use \Cref{cor:pointwise criterion for torsor}. Let $x \in |\Xcal|_{\eta}^{\mrm{int}}$ with residue field $F(x)$, consider the $H$-torsor $P = \Ucal \times_{\Mscr} \Spec \, F(x)$ over $\Spec \, F(x)$. By Definition, $\Spec \, F(x) \to \Mscr$ extends to $\Spec \, \Ocal_{F(x)} \to \Mscr$, so $\Pcal \coloneq \Ucal \times_{\Mscr} \Spec \, \Ocal_{F(x)}$ is an $H$-torsor over $\Ocal_{F(x)}$ extending $P$.
\end{proof}

\begin{lem}\label{lem:semi base change of PGR is PGR}
    Let $\Mscr/\Ocal_F$ be an algebraic stack, $\Ucal \to \Mscr$ a proper affine morphism of stacks, and let $\Xcal / \Ocal_F$ be as in \Cref{dfn:PGR-extension property}, together with a PGR-morphism $f:\Xcal \rightsquigarrow \Mscr$. Let $\Ucal'_{\eta} \coloneq \Ucal_{\eta} \times_{\Mscr_{\eta}} \Xcal_{\eta}$, suppose it admits a model $\Ucal' / \Ocal_F$ such that the generic morphism $\Ucal' \rightsquigarrow \Xcal$ extends to $\Ucal' \to \Xcal$ over $\Ocal_F$. Then the generic morphism $\Ucal' \rightsquigarrow \Ucal$ is PGR.
\end{lem}
\begin{proof}
    Consider the 2-cartesian diagram
    \[
    \begin{tikzcd}
        \Ucal'_{\eta} \arrow[r] \arrow[d] & \Xcal_{\eta} \arrow[d] \\
        \Ucal_{\eta} \arrow[r]            & \Mscr_{\eta}.              
    \end{tikzcd}
    \]
    Since $\Ucal' \rightsquigarrow \Xcal$ extends, it is a PGR morphism, hence so is the composition $\Ucal' \rightsquigarrow \Mscr$. It is enough to show that the pre-image of $|\Mscr|_{\eta}^{\mrm{int}}$ via $\Ucal \to \Mscr$ is $|\Ucal|_{\eta}^{\mrm{int}}$. Let $x:\Spec \, E \to \Ucal$, for some finite $E/F$, mapping to $|\Mscr|_{\eta}^{\mrm{int}}$; up to extending $E$, we may assume that there is a 2-commutative diagram of solid arrows
    \[
    \begin{tikzcd}
        \Spec \, E \arrow[r] \arrow[d] & \Ucal \arrow[d] \\
        \Spec \, \Ocal_E \arrow[r] \arrow[ur,dashed]            & \Mscr.              
    \end{tikzcd}
    \]
    We argue that there exists a dashed arrow making the diagram 2-commutative, so that $x$ lies in $|\Ucal|_{\eta}^{\mrm{int}}$. First, we base-change $\Ucal \to \Mscr$ along the bottom horizontal arrow so as to reduce to the case:
    \[
    \begin{tikzcd}
        \Spec \, E \arrow[r] \arrow[d] & \Ucal_{\Ocal_E} \arrow[d] \\
        \Spec \, \Ocal_E \arrow[r, "\mrm{id}"] \arrow[ur,dashed]            & \Spec \, \Ocal_E              
    \end{tikzcd}
    \]    
    where the vertical arrow on the right is proper and affine. We conclude with the valuative criterion for properness of a morphism of schemes\footnote{Dropping the affineness assumption, the valuative criterion for properness of a morphism of stacks only gives the dashed arrow after an extension of valuation ring, which may not be finite in general, so we cannot conclude. More generally, we could have replaced "affine" with "representable by a scheme".}.
\end{proof}

\subsection{Smooth-extension}\label{subsec:smooth extension}

\begin{dfn}\label{dfn:smooth extension}
    Let $\Sscr/\Ocal_F$ be a separated Deligne-Mumford stack. We say that it has \emph{smooth-extension} if every generic morphism $f:\Xcal \rightsquigarrow \Sscr$ from a regular scheme $\Xcal$ formally smooth over $\Ocal_F$, automatically and uniquely extends to an $\Ocal_F$-morphism $\Xcal \to \Sscr$.
\end{dfn}

Restricting \Cref{dfn:smooth extension} to smooth separated schemes $\Xcal/\Ocal_F$ would recover the definition of Néron models. As we will be dealing with pro-schemes, we need to remove the finite type assumption. 

\medskip

L-extension can be inferred from smooth-extension. The following statement is a reformulation of \cite[§1.3]{bakker2024integral}.

\begin{prop}\label{prop:smooth extension implies L extension}
    Let $R$ be a topological ring, let $V$ be a free $R$-module of finite rank, let $K_0$ be a profinite subgroup of $\GL(V)$. Recall that $\Ccal(K_0)$ denotes the set of all open normal subgroups of $K_0$.

    Let $(\Sscr_K)_{K \in \Ccal(K_0)}$ be a projective system of separated Deligne-Mumford stacks over $\Ocal_F$ which is fit for $K_0$ as in \Cref{dfn:fit system of stacks}. Let 
    \[\Sscr_{\infty} \coloneq \varprojlim_{K \in \Ccal(K_0)} \Sscr_K.\]
    Let $\L^{\can}$ be the canonical local system on $\Sscr_{K_0}$ as in \Cref{dfn:can loc sys on stack}. If $\Sscr_{\infty}$ has smooth-extension, then $(\Sscr_{K_0},\L^{\can})$ has L-extension.
\end{prop}
\begin{proof}
    Let $\Xcal / \Ocal_F$ be a smooth scheme, let $f:\Xcal \rightsquigarrow \Sscr_{K_0}$ such that $f_{\eta}^* \L^{\can}$ extends to a local system $\L$ on $\Xcal$. We want to show that $f$ extends to $\tilde{f}:\Xcal \to \Sscr_{K_0}$.

    By definition, $\L^{\can}$ is induced from a $K_0$-torsor on $\Sscr_{K_0}$, thus $\L$ is induced from a $K_0$-torsor on $\Xcal$. From \Cref{prop:universal fit system}, the corresponding morphism $\Xcal \to \BB K_0$ defines a $K_0$-fit system of schemes $(\Xcal_K)_{K \in \Ccal(K_0)}$, with $\Xcal_{K_0} = \Xcal$. Let
    \[\Xcal_{\infty} \coloneq \varprojlim_K \Xcal_K\]
    which is a regular scheme formally smooth over $\Ocal_F$. The $K_0$-fit system of generic fibers $((\Xcal_K)_{\eta})_K$ over $F$ coincides with the pullback along $f_{\eta}$ of the fit system $((\Sscr_K)_{\eta})_K$. In particular, there is an induced $F$-morphism
    \[(\Xcal_{\infty})_{\eta} \to (\Sscr_{\infty})_{\eta}\]
    By smooth-extension, there is an induced morphism
    \[\Xcal_{\infty} \to \Sscr_{\infty}.\]
    Now it suffices to show that the composition $\Xcal_{\infty} \to \Sscr_{\infty} \to \Sscr_{K_0}$ descends to a morphism $\Xcal_{K_0} \to \Sscr_{K_0}$, in other words we need to show by fpqc descent that both composition arrows
    \[\Xcal_{\infty} \times_{\Xcal_{K_0}} \Xcal_{\infty} \rightrightarrows \Xcal_{\infty} \to \Sscr_{\infty} \to \Sscr_{K_0}\]
    are isomorphic. This follows from \Cref{prop:morphisms isomorphic over dense subset}.
\end{proof}

We state some lemmas to be used in \Cref{sec:extension for Shimura}.

\begin{lem}\label{lem:normalized closed substack has smooth extension}
    Let $\Sscr/\Ocal_F$ be a separated Deligne-Mumford stack with smooth-extension. Let $\Zscr \hookrightarrow \Sscr$ a closed substack. Then $\Zscr$ has smooth-extension.
\end{lem}
\begin{proof}
    Let $\Xcal/\Ocal_F$ be as in \Cref{dfn:smooth extension}, $f:\Xcal \rightsquigarrow \Zscr$ a PGR-morphism. Then the composition $\Xcal \rightsquigarrow \Sscr$ is again PGR, hence extends to $g:\Xcal \to \Sscr$. It follows from \Cref{lem:factoring morphisms through closed substacks} that $g$ factors through $\Zscr$; note that $\Xcal / \Ocal_F$ is flat because formally smooth (see \cite[Théorème 19.7.1]{EGAIV}).
\end{proof}

\begin{rk}\label{rk:closed substack has extension}
    \Cref{lem:normalized closed substack has smooth extension} and its proof still hold with smooth-extension replaced with any extension property involving generic morphisms from flat $\Ocal_F$-schemes $\Xcal$.
\end{rk}

\begin{lem}\label{lem:descent of smooth extension}
    Let $\Ocal'$ be a DVR which is faithfully flat, formally smooth over $\Ocal_F$, with the following property: for any regular scheme $\Xcal$ formally smooth over $\Ocal_F$, the scheme $\Xcal_{\Ocal'}$ is again regular.

    \noindent Let $\Sscr/\Ocal_F$ be a separated Deligne-Mumford stack such that $\Sscr_{\Ocal'}/\Ocal'$ has smooth extension. Then $\Sscr/\Ocal_F$ has smooth extension.
\end{lem}
\begin{proof}
    Consider a generic morphism $f:\Xcal \rightsquigarrow \Sscr$ from a regular scheme $\Xcal$ formally smooth over $\Ocal_F$. By assumption $\Xcal_{\Ocal'}$ is a regular scheme formally smooth over $\Ocal'$. By smooth-extension we find $\Xcal_{\Ocal'} \to \Sscr_{\Ocal'}$. Now it suffices to show that the composition $\Xcal_{\Ocal'} \to \Sscr_{\Ocal'} \to \Sscr$ descends to a morphism $\Xcal \to \Sscr$, in other words we need to show by fpqc descent that both composition arrows
    \[\Xcal_{\Ocal'} \times_{\Xcal} \Xcal_{\Ocal'} \rightrightarrows \Xcal_{\Ocal'} \to \Sscr_{\Ocal'} \to \Sscr\]
    are isomorphic. This follows from \Cref{prop:morphisms isomorphic over dense subset} once we show that $\Xcal_{\Ocal'} \times_{\Xcal} \Xcal_{\Ocal'}$ is regular. Observe that $\Xcal_{\Ocal'}$ is a regular scheme formally smooth over $\Ocal_F$ (here we use that $\Ocal'$ is formally smooth over $\Ocal_F$), so its base-change along $\Spec \, \Ocal' \to \Spec \, \Ocal_F$ is regular by assumption. This base-change recovers $\Xcal_{\Ocal'} \times_{\Xcal} \Xcal_{\Ocal'}$.
\end{proof}

To prove \Cref{lem:proétale lifting of smooth extension} below, we need the following fact:
\begin{fact}[\cite{milne1992points}, Proposition 2.4]\label{fact:Milne regular ring}
    Let $(R_i)_i$ be a inductive system of regular local rings with étale transition maps. Then $R = \varinjlim_i R_i$ is again regular.
\end{fact}

\begin{lem}\label{lem:proétale lifting of smooth extension}
    Let $\Sscr, \Sscr' / \Ocal_F$ be separated Deligne-Mumford stacks, suppose $\Sscr$ has smooth-extension. Consider a pro-étale morphism $\Sscr' \to \Sscr$ representable by schemes. Then $\Sscr'$ has smooth extension.
\end{lem}
\begin{proof}
    If $\Xcal / \Ocal_F$ is as in \Cref{dfn:smooth extension} with a generic morphism $\Xcal \rightsquigarrow \Sscr'$, composing to $\Sscr$ and applying smooth-extension yields $\Xcal \longrightarrow \Sscr.$
    Setting $\Xcal' \coloneq \Xcal \times_{\Sscr} \Sscr'$ which is a scheme, it is again regular by \Cref{fact:Milne regular ring}. Thus by \Cref{prop:morphisms isomorphic over dense subset} and pro-étale descent, the morphism $\Xcal' \to \Sscr'$ descends to $\Xcal \to \Sscr'$. 
\end{proof}

\section{Smooth-extension for integral models of Shimura varieties}\label{sec:extension for Shimura}

Producing schemes with smooth-extension will be our starting point for finding stacks with PGR-extension. One source of such schemes is the theory of \emph{integral canonical models} of Shimura varieties, which we recall in the first paragraph below.

The remaining paragraphs deal with the following problem: having smooth-extension is not preserved under general base change. As integral canonical models of Shimura varieties are usually defined over a valuation ring $\Ocal_{(v)}$ of their reflex field, it is not certain that their base-change to $\Ocal_F$ still has smooth-extension if, say, $F/\Q_p$ is a finite, possibly ramified extension. We argue that there is no issue whenever $F / \Qp$ has absolute ramification index $e \le p-1$. To do so, we sketch the original constructions (by works of Milne, Moonen, Kisin and Madapusi Pera) of integral canonical models $\Sscr$ in the cases of Siegel, Hodge, abelian and orthogonal-type, and investigate the proofs that they have smooth-extension to see that $\Sscr_{\Ocal_F}$ also does if $\Ocal_F$ has absolute ramification index $e \le p-1$.

\medskip

The reader willing to take the results of this section for granted is advised to read only \Cref{subsec:int can model} and \Cref{subsec:moduli of polar ab}, and to go back to the statements of the remaining paragraphs while reading \Cref{sec:applications}. For this section and the next, we let $\widehat{\Z}^p = \prod_{\l \neq p} \Zl$ denote the prime-to-$p$ profinite integers and $\A_f^p$ the prime-to-$p$ adèles.

\subsection{Integral canonical models}\label{subsec:int can model}
Let $(G,\Omega)$ be a Shimura datum with reflex field $E/\Q$, and $K_p \subseteq G(\Qp)$ a compact open subgroup; $K_p$ will often, but not always, be a hyperspecial subgroup. Let $v$ be a place of $E$ lying above $p$, $\Ocal_{(v)}$ the local ring of $\Ocal_E$ at $v$ and $\Ocal_v$ its $v$-adic completion. Consider a neat level $K = K^p K_p$ where $K^p$ is a compact open subgroup of $G(\A_f^p)$, write $\Sh_K(G,\Omega)/E$ for the corresponding variety over $E$. Set
\[\Sh_{K_p}(G,\Omega) = \varprojlim_{K^p} \Sh_{K^p K_p}(G,\Omega)\]
which is an $E$-scheme with continuous $G(\A_f^p)$-action (see \Cref{rk:continuous action} below). If $K=K^p K_p$ is no longer assumed neat, we let $\Sh_K[G,\Omega] = [\Sh_{K_p}(G,\Omega) / K^p]$ denote the quotient stack; for $K$ neat this is a scheme which coincides with $\Sh_K(G,\Omega)$.

Suppose $\Ocal_F$ is a discrete valuation ring, faithfully flat over $\Ocal_{(v)}$, let $\mf m_{F}$ denote its maximal ideal. We recall (from \cite[§2]{milne1992points}, \cite[§3]{moonen1998models}, \cite[§1.0]{vasiu1999integral} and \cite[2.3]{kisin2010integral}) 
the notions of integral canonical models of Shimura varieties.

\begin{dfn}\label{dfn:int model}
    An \emph{integral model} of $\Sh_{K_p}(G,\Omega)$ over $\Ocal_F$ is a faithfully flat $\Ocal_F$-scheme denoted $\Sscr_{K_p}(G,\Omega)$ with an continuous $G(\A_f^p)$-action (see \Cref{rk:continuous action} below), such that
    \begin{itemize}
        \item[(i)] there is a $G(\A_f^p)$-equivariant identification $\Sscr_{K_p}(G,\Omega) \times_{\Ocal_F} F \simeq \Sh_{K_p}(G,\Omega) \times_{\Ocal_E} F$, and
        \item[(ii)] there exists a compact open subgroup $K_0^p \subseteq G(\A_f^p)$ such that for any $K_2^p \subseteq K_1^p \subseteq K_0^p$ open subgroups, the canonical quotient map
            \[\Sscr_{K_p}(G,\Omega) / K_1^p \to \Sscr_{K_p}(G,\Omega) / K_2^p\]
            is a finite étale morphism between separated schemes of finite type over $\Ocal_F$.
    \end{itemize} 
\end{dfn}

\begin{rk}\label{rk:continuous action}
    The notion of (right) "continuous action" of a locally compact, totally disconnected group $\Gamma$ on an $S$-scheme $T$ is to be understood in the sense of Deligne, (see \cite[§2.7]{deligne1979varietes}, also \cite[Paragraph following Definition 2.1]{milne1992points}). This means that we can write
    \[T = \varprojlim_K T_K\]
    where $K$ ranges over all compact open subgroups of $\Gamma$, and the $T_K$ are schemes with ample line bundles (so that quotients by finite group actions exxist as schemes), so that
    \begin{itemize}
        \item[(i)] For any $g \in \Gamma$, there is an isomorphism of projective systems
            \[\rho(g) = (\rho_K(g))_K:(T_{K})_K \xrightarrow{\sim} (T_{g^{-1}Kg})_K\]
            such that $\rho(g_1g_2) = \rho(g_2) \circ \rho(g_1)$,
        \item[(ii)] For any $K$ as above and $k \in K$, $\rho_K(k) = \id_{T_K}$, and
        \item[(iii)] For any $K$ as above and any open normal subgroup $L \subseteq K$, the quotient morphism
            \[T_K / (K/L) \longrightarrow T_L\]
        is an isomorphism.
    \end{itemize}
    With the terminology of \Cref{sec:can loc sys on schemes and stacks}, this is equivalent to a $\Gamma$-sheaf structure
    \[h_T \times \Fcal_{\Gamma} \longrightarrow h_T\]
    on the functor of points of $T$, along with some representability conditions on the quotient stacks $[T / K]$.
\end{rk}

We want our models to be canonical in some sense, by requiring that $\Sscr_{K_p}(G,\Omega)$ satisfies some extension property. Namely, let $\Pcal$ be some property satisfied by $\Ocal_F$-schemes, then one could ask for the following:
\begin{enumerate}
    \item[(i)] the scheme $\Sscr_{K_p}(G,\Omega)$ satisfies $\Pcal$, and
    \item[(ii)] every generic morphism $f:\Xcal \rightsquigarrow \Sscr_{K_p}(G,\Omega)$ from a scheme $\Xcal / \Ocal_F$ satisfying $\Pcal$, automatically and uniquely extends to an $\Ocal_F$-morphism $\Xcal \to \Sscr$.
\end{enumerate}
These properties ensure the uniqueness of the integral model $\Sscr_{K_p}(G,\Omega)$ over $\Ocal_F$ up to unique isomorphism. \Cref{dfn:smooth extension} motivates the case 
\[\Pcal = \text{being a regular scheme formally smooth over } \Ocal.\]
In fact, item (i) above is replaced with a slightly stronger condition.

\begin{dfn}\label{dfn:smooth int model}
    Let $\Sscr_{K_p}(G,\Omega)$ be an integral model over $\Ocal_F$ as in \Cref{dfn:int model}. It is a \emph{smooth integral canonical model} if moreover:
    \begin{enumerate}
        \item[(i)] For small enough $K^p \subseteq K_0^p$, the scheme $\Sscr_{K_p}(G,\Omega) / K^p$ is smooth, and
        \item[(ii)] the scheme $\Sscr_{K_p}(G,\Omega)$ has smooth extension as in \Cref{dfn:smooth extension}.
    \end{enumerate}
\end{dfn}

However, useful integral models for orthogonal-type Shimura varieties constructed in \cite{pera2016integral} are not all shown to be formally smooth. Nevertheless, they satisfy an extension property which we now describe. The following notions originate from the work of Vasiu.

\begin{dfnBreak}[\cite{pera2016integral}, Definition 2.12]\label{dfn:healthy} ~
    \vspace{-1.3em}
    \begin{itemize}
        \item A regular, local, faithfully flat $\Z_{(p)}$-algebra $R$ with maximal ideal $\mf{m}$ is \emph{quasi-healthy} if every abelian scheme over $\Spec(R) \setminus \{\mf m\}$ extends uniquely to an abelian scheme over $\Spec \,R$.
        \item A regular, faithfully flat $\Z_{(p)}$-scheme $\Xcal$ is \emph{healthy} is for every open subscheme $U$ containing the generic fiber $\Xcal_{\Q}$ and all generic points of the special fiber $\Xcal_{\F_p}$, every abelian scheme over $U$ extends uniquely to an abelian scheme over $\Xcal$.
        \item A regular, faithfully flat $\Z_{(p)}$-scheme $\Xcal$ is \emph{locally healthy} if for every $x \in \Xcal_{\F_p}$ of codimension at least $2$, the complete local ring $\widehat{\Ocal_{\Xcal,x}}$ is quasi-healthy.
    \end{itemize}
\end{dfnBreak}

Note that locally healthy implies healthy (see \cite[Remark 2.13]{pera2016integral}). The following extension property will not be referenced outside of \Cref{sec:extension for Shimura}.

\begin{dfn}\label{dfn:healthy extension}
    Let $\Sscr/\Ocal_F$ be a separated Deligne-Mumford stack. We say that it has \emph{healthy-extension} if every generic morphism $f:\Xcal \rightsquigarrow \Sscr$ from a regular, locally healthy scheme $\Xcal / \Ocal_F$, automatically and uniquely extends to an $\Ocal_F$-morphism $\Xcal \to \Sscr$.
\end{dfn}

We define integral canonical models in general by choosing $\Pcal$ following \Cref{dfn:healthy extension}.

\begin{dfn}[\cite{madapusi2015tate}, Definition 4.3]\label{dfn:int can model}
    Let $\Sscr_{K_p}(G,\Omega)$ be an integral model over $\Ocal_F$ as in \Cref{dfn:int model}. It is an \emph{integral canonical model} if moreover:
    \begin{enumerate}
        \item[(i)] The scheme $\Sscr_{K_p}(G,\Omega)$ is regular, locally healthy, and
        \item[(ii)] The scheme $\Sscr_{K_p}(G,\Omega)$ has healthy-extension.
    \end{enumerate}
\end{dfn}

Although integral canonical models may not be formally smooth, the work of Vasiu-Zink still shows that they have smooth extension whenever $\Ocal_F$ is not too ramified.

\begin{prop}\label{prop:smooth implies healthy}
   Suppose $\Ocal_F$ has absolute ramification index $e \le p-1$. Then any regular scheme $\Xcal$ formally smooth over $\Ocal_F$ is locally healthy, hence healthy.

   \noindent In particular, if $\Sscr / \Ocal_F$ is a separated Deligne-Mumford stack with healthy-extension, then it has smooth extension.
\end{prop}
\begin{proof}
    It is stated in \cite[Corollary 5]{vasiu2010purity} that such an $\Xcal$ is healthy. Let us run through the argument to see that $\Xcal$ is in fact locally healthy. Let $x \in \Xcal_{\F_p}$ be of codimension at least 2, set $R = \Ocal_{\Xcal,x}$. Then the first statement is given by \Cref{lem:smooth implies healthy} below. The second statement follows.
\end{proof}

\begin{lem}\label{lem:smooth implies healthy}
    Suppose $\Ocal_F$ has absolute ramification index $e \le p-1$. Let $R$ be a regular local ring of dimension at least 2, formally smooth over $\Ocal_F$, and let $\widehat{R}$ be the completion with respect to its maximal ideal. Then $\widehat{R}$ is quasi-healthy.
\end{lem}
\begin{proof}
    \cite[Theorem 3]{vasiu2010purity} provides a criterion on 
    a regular local ring $A$ of dimension $d \ge 2$, characteristic $(0,p)$ and residue field $k$, implying that $A$ is quasi-healthy. The criterion is as follows: if there exists a faithfully flat, regular, local, complete $A$-algebra $\overline{A}$ of dimension $d$ with algebraically closed residue field $\overline{k}$, together with an epimorphism
    \[\overline{A} \twoheadrightarrow W(\overline{k})[[T_1,T_2]] / (p-h)\]
    where $h \in (T_1,T_2) W(\overline{k})[[T_1,T_2]]$ does not lie in the ideal $(p, T_1^p,T_2^p,T_1^{p-1}T_2^{p-1})$, then $A$ is quasi-healthy. We argue below that the criterion is satisfied for $A=R$ and $\overline{k}$ an algebraic closure of $k$, so $R$ is quasi-healthy. As $R \to \overline{R}$ factors through the completion $\widehat{R}$, the latter is also quasi-healthy.

    First, we write
    \[\Ocal_F = W(k_F)[T_1] / (p-h)\]
    where $k_F \subseteq k$ is the residuel field of $F$ and $p-h \in W(k_F)[T_1]$ is an appropriate Eisenstein polynomial; here $h \in T_1 W(k_F)[T_1]$ is monic of degree at most $p-1$, thus $h \not\in (p,T_1^p)$. 
    Consider the $\Ocal_F$-algebra
    \[B \coloneq W(\overline{k})[T_1]/(p-h)\]
    which is a DVR. By formal smoothness of $\Ocal_F \to R$, the morphism $R \twoheadrightarrow k \to \overline{k}$ extends to a morphism $R \to B$ over $\Ocal_F$, hence to a morphism of $B$-schemes
    \[\Spec \, B \longrightarrow \Spec (R \otimes_{\Ocal_F} B).\]
    Define $\tilde{R}$ as the localization of $R \otimes_{\Ocal_F} B$ at the maximal ideal in the image of this morphism, so that $\Spec \, \tilde{R}$ is a $B$-scheme with a section of the structure morphism. By \Cref{lem:create tangent vector} below we find a surjection of $B$-algebras
    \[\tilde{B} \twoheadrightarrow B[T_2]/(T_2^2).\]
    Formal smoothness of $B \to \tilde{R}$ lifts this surjection to $\tilde{R} \to B[T_2]/(T_2^n)$ for any $n \ge 1$, which is again surjective. Indeed, arguing by induction, if $\tilde{R} \to B[T_2]/(T_2^n)$ is surjective and lifts to $\tilde{R} \to B[T_2]/(T_2^{n+1})$, the image of this lift is a $B$-algebra containing elements of the form $T_2 + x T_2^n$ and $T_2^{n-1} + y T_2^n$ with $x,y \in B$, hence it contains their product $T_2^n$, as well as the generator $T_2 = (T_2 + x T_2^n) - x T_2^n$ which proves surjectivity. Summarizing, we find a morphism of $B$-algebras
    \[\tilde{R} \to B[[T_2]] = W(\overline{k})[[T_1,T_2]] / (p-h)\]
    mapping surjectively onto the quotients $B[T_2]/(T_2^n)$. This map factors through the completion $\overline{R} \coloneq \widehat{\tilde{R}}$ of $\tilde{R}$ at its maximal ideal into a surjective map
    \[\overline{R} \twoheadrightarrow B[[T_2]] = W(\overline{k})[[T_1,T_2]] / (p-h)\]
    as desired.
\end{proof}

\begin{lem}\label{lem:create tangent vector}
    Let $R$ be a regular local ring of dimension at least 2, faithfully flat over some DVR $B$. Suppose $\Spec \, R \to \Spec \, B$ admits a section. Then there exists a surjection $R \twoheadrightarrow B[T]/(T^2)$ of $B$-algebras. 
\end{lem}
\begin{proof}
    As the structure morphism is separated, the section $\Spec \, B \to \Spec \, R$ is a closed embedding. Let $I$ denote the ideal corresponding to its image, so that $R / I \simeq B$ as $B$-algebras. Any morphism $R \to B[T]/(T^2)$ restricts to the identity on $B$, hence is determined by its restriction to $I$ mapping to the ideal $(T)$, which automatically factors through the quotient $I / I^2$. In other words:
    \[\Hom_{B \text{-alg}}(R, B[T]/(T^2)) = \Hom_{B \text{-mod}}(I / I^2, B).\]
    The only non-surjective $R \to B[T]/(T^2)$ is the one mapping to the zero morphism on the right-hand side. Thus it suffices to show that there are nonzero elements of $\Hom_{B \text{-mod}}(I / I^2, B)$. As $R$ is regular local, it is Noetherian, so $I/I^2$ is a finitely generated module over the DVR $B$: there is a nonzero map $I / I^2 \to B$ if and only if $(I / I^2) \otimes_{B} B_0 \neq 0$, where $B_0 = \mrm{Frac}(B)$. By flatness of $B \to B_0$ the latter equals $I_0 / I_0^2$, where $I_0 = IR_0$ is a maximal ideal of the localization $R_0 \coloneq R \otimes_{B} B_0$. Now, by flatness of the $B$-algebra $R$
    \[\dim_F (R_0) = \dim(R) -1 \ge 1\]
    and so 
    at the maximal prime $I_0$ we find $\dim_F(I_F / I_F^2) \ge 1$, hence $I_0 / I_0^2 \neq 0$.
\end{proof}

Consequently:

\begin{cor}
    Suppose $\Ocal_F$ has absolute ramification index $e \le p-1$. Let $\Sscr_{K_p}(G,\Omega)$ be an integral canonical model over $\Ocal_{(v)}$ as in \Cref{dfn:int can model}. Then the base-change $\Sscr_{K_p}(G,\Omega)_{\Ocal_F}$ has smooth extension.
\end{cor}

Let $\Sscr_{K_p}(G,\Omega)$ be an integral model of $\Sh_{K_p}(G,\Omega)$ over $\Ocal_(v)$.
Suppose either that
\begin{itemize}
    \item it is a smooth integral canonical model, or
    \item it is an integral canonical model, and $e \le p-1$.
\end{itemize}
Then we have established that $\Sscr_{K_p}(G,\Omega)_{\Ocal_F} / \Ocal_F$ has smooth extension, which has the following consequence. For any level $K = K^p K_p$, we let $\Sscr_{K}[G,\Omega] \coloneq [\Sscr_{K_p}(G,\Omega) / K^p]$ denote the quotient stack over $\Ocal_{(v)}$. Fix $K_0^p$ any compact open subgroup of $G(\A_f^p)$, then the projective system $(\Sscr_{K^p K_p}[G,\Omega])_{K^p \in \Ccal(K_0^p)}$ is fit for $K_0^p$, so from \Cref{dfn:can loc sys on stack} there is a canonical local system $\L^{\can}$ on $\Sscr_{K_0^p K_p}[G,\Omega]$. After base-change to $\Ocal_F$, we find:

\begin{prop}\label{prop:int model has PGR}
    The pair $(\Sscr_{K_0^p K_p}[G,\Omega]_{\Ocal_F}, \L^{\can})$ has L-extension, hence also PGR-extension, over $\Ocal_F$.
\end{prop}
\begin{proof}
    This follows from \Cref{prop:smooth extension implies L extension} and \Cref{prop:L-extension implies PGR extension}.
\end{proof}

Integral canonical models have been shown to exist for Shimura varieties of Siegel (in \cite{milne1992points} and \cite{moonen1998models}), Hodge (in \cite{kottwitzPointsShimuraVarieties1992} for PEL, \cite{kisin2010integral} for the general case) and abelian-type (in \cite{kisin2010integral}), as well as of certain orthogonal-type\footnote{The orthogonal-type Shimura varieties condisered by Madapusi are abelian-type, yet the construction of Madapusi does not follow from that of Kisin. Indeed, Kisin constructs integral models whenever $K_p$ is hyperspecial, while Madapusi chooses a specific $K_p$, called the \emph{discriminant kernel}, which may not be hyperspecial in general.} (in \cite{pera2016integral}), over a localization of the ring of integers of their reflex field. In the following paragraphs, we sketch their constructions. The important takeaway is that the abelian and orthogonal cases are bootstrapped from the Hodge case, itself bootstrapped from the Siegel case, itself boostrapped from moduli spaces of polarized abelian schemes.

\subsection{Smooth-extension for moduli spaces of polarized abelian schemes}\label{subsec:moduli of polar ab}

We begin with some recollections on moduli spaces of polarized abelian schemes following \cite[§1]{dejongModuliSpacesPolarized1993a}. Set $g,d>0$ integers; all polarized abelian schemes in this section are of relative dimension $g$ and degree $d$. Fix a sequence of integers $\delta = (\delta_1,\dots,\delta_g)$ such that $\delta_1 | \dots | \delta_g$ and $d = \prod_{i=1}^g \delta_i$. Consider the finite flat group scheme over $\Spec \, \Z$
\[K(\delta) \coloneq \prod_{i=1}^g \mu_{\delta_i, \Spec \, \Z} \times \u{\Z / \delta_i \Z}_{\Spec \, \Z}.\]
The following moduli stacks are defined over $\Z$; we restrict to $\Z_{(p)}$ for our applications.
\begin{dfn}\label{dfn:Ag}
    For any $\Z_{(p)}$-scheme $T$, let $\Ascr_{g,d}(T)$ denote the groupoid of polarized abelian schemes $(A,\lambda)$ over $T$ of relative dimension $g$ and degree $d$.
    Let $\Ascr_{g,\delta}(T)$ denote the subcategory of all objects $(A,\lambda)$ such that for some faithfully flat $T' \to T$, the base-change $(\ker \lambda)_{T'}$ is isomorphic to $K(\delta)_{T'}$.
\end{dfn}
Then (see Remarks 1.2 and Proposition 1.5 of \textit{loc. cit.}) $\Ascr_{g,d}$ and $\Ascr_{g,\delta}$ are separated Deligne-Mumford stacks of finite type over $\Spec \, \Z_{(p)}$, and the morphism
\[\Ascr_{g,\delta} \longrightarrow \Ascr_{g,d}\]
is represented by a locally-closed immersion. We let $\overline{\Ascr}_{g,\delta} \hookrightarrow \Ascr_{g,d}$ denote the substack obtained as the closure of $\Ascr_{g,\delta}$. The distinction between $\Ascr_{g,\delta}$ and its closure is only relevant when $p | d$:

\begin{fact}[\cite{dejongModuliSpacesPolarized1993a}, Remarks 1.2, Theorem 1.12]\label{fact:Agdelta and its closure}
    If $p \nmid d$ then $\Ascr_{g,d}$ and $\Ascr_{g,\delta}$ are smooth over $\Spec \, \Z_{(p)}$, and $\Ascr_{g,\delta} \longrightarrow \Ascr_{g,d}$ is a closed embedding, i.e. $\Ascr_{g,\delta} = \overline{\Ascr}_{g,\delta}$.
\end{fact}

\begin{rk}
    \Cref{fact:Agdelta and its closure} fails if $p |d$. Consider a principally polarized supersingular elliptic curve over $\overline{\F_p}$ and $E$ some lift of it to a local ring $R$ of characteristic zero (which exists by \cite[Corollary 3.2]{OortNormanModuli}). Consider the degree $p^2$ polarization $\lambda$ on $E$ obtained by pullback of the principal polarization via multiplication by $p$ on $E$. Let $(E_{\eta},\lambda_{\eta})$ and $(E_s,\lambda_s)$ be the degree $p^2$-polarized elliptic curves coming from the generic and closed fibers of $E \to \Spec \, R$. Then $(E_{\eta},\lambda_{\eta})$ is a point of $\Ascr_{g,\delta}$ for $\delta = (p,p)$, so that $(E_s,\lambda_s)$ is a point of the closure $\overline{\Ascr}_{g,\delta}$, yet it is not in $\Ascr_{g,\delta}$ since $\ker \lambda_s = E_s[p]$ does not equal $K(\delta)$ fppf-locally.
\end{rk}

If $(A,\lambda)$ is a polarized abelian variety over a field $k$ of characteristic 0 or $p$, then $\lambda$ induces a symplectic form on the $\widehat{\Z}^p$-Tate module of $A$, defined to be the product of all $\l$-adic Tate modules of $A$ for $p \neq \l >0$. There are positive integers $(d_1,\dots,d_g)$ (uniquely defined up to powers of $p$) satisfying $d_1 |\dots | d_g$ such that, for some appropriate choice of a basis, the symplectic form on the Tate-module is prescribed by the matrix
\[J_{\Delta} \coloneq 
\begin{pmatrix}
    0 & \Delta \\
    -\Delta & 0
\end{pmatrix}\]
where $\Delta = \mrm{diag}(d_1,\dots,d_g)$ is called a \emph{symplectic type}; we then say that $(A,\lambda)$ has \emph{symplectic type} $\Delta$. For convenience we sometimes write $\Delta$ as a $g$-tuple $(d_1,\dots,d_g)$.

The symplectic type is completely determined by the orders of the $N$-torsion subgroups $(\ker \lambda)[N]$ for all $p \nmid N$. In particular, if $\ker \lambda \simeq K(\delta)$ then $\Delta = \Delta_{\delta} \coloneq (\delta_1,\dots,\delta_g)$. For any base $\Z_{(p)}$-scheme $T$, a polarized abelian $T$-scheme has \emph{symplectic type} $\Delta$ if all its geometric fibers have symplectic type $\Delta$; thus any object of $\Ascr_{g,\delta}(T)$ has symplectic type $\Delta = \Delta_{\delta}$.

\begin{lem}
    Let $T$ be a $\Z_{(p)}$-scheme and $(A,\lambda)$ an object of $\overline{\Ascr}_{g,\delta}(T)$. Then $(A,\lambda)$ has symplectic type $\Delta = \Delta_{\delta}$.
\end{lem}
\begin{proof}
    This is a statement about the geometric points of $\overline{\Ascr}_{g,\delta}$, thus it suffices to show the result when $T \to \overline{\Ascr}_{g,\delta}$ is étale so that its images meets $\Ascr_{g,\delta}$. Then the result holds true over a dense subscheme $T_0$ of $T$. For $p \nmid N$ the group scheme $(\ker \lambda)[N]$ is finite flat with constant order over $T_0$, hence its order remains constant over $T$. Thus the symplectic type is constant over $T$, and equals that over $T_0$. 
\end{proof}

Let $V$ be a finite rank, free $\Z$-module with a symplectic form $\psi$ which is represented in an appropriate basis by the integral matrix
\[J_{\delta} \coloneq \begin{pmatrix}
    0 & \Delta_{\delta} \\
    -\Delta_{\delta} & 0.
\end{pmatrix}\]
Consider the $\Z$-group scheme $\Gcal_{\delta} \coloneq \GSp(V,\psi)$, set $G \coloneq \Gcal_{\delta,\Q}$ for its generic fiber. In an appropriate basis of $V$, for any ring $S$ we have
\[\Gcal_{\delta}(S) = \{M \in \GL_{2g}(S) \; | \; M J_{\delta} M^t = J_{\delta}\}.\]
The restriction $\Gcal_{\delta,\Z_{(p)}}$ is reductive (equiv. smooth) if and only if $p \nmid d$.

\medskip

We now add level structures using the formalism of \Cref{subsec:can vs univ}. Set $R \coloneq \widehat{\Z}^p = \prod_{\l \neq p} \Zl$ and $V_R \coloneq V \otimes R$, as well as $Q = \Gcal_{\delta}(R) \hookrightarrow \GL(V_R)$. Consider the following objects (see the first paragraph of \Cref{subsec:can vs univ} for the simpler case of principally polarized abelian schemes over $\Q$).
\begin{enumerate}
    \item[(i)] Set $S \coloneq \Spec \, \Z_{(p)}$ and let $\Cscr$ denote the big pro-étale site over $S$ (see \Cref{subsec:alg stacks on the big pro-étale site} for details), so that $\Ascr_{g,d}$ is a stack over $\Cscr$.
    \item[(ii)] For any $T \to \Ascr_{g,d}$ corresponding to a polarized abelian scheme $f:A \to T$, we set $\L(A)$ the $R$-local system defined as its (relative) $\widehat{\Z}^p$-Tate module\footnote{This is a nonstandard notation for Tate modules, which we use to suggest that they are candidates for L-extension.}. This functor defines a local system on $\Ascr_{g,d}$
    \[\L:\Ascr_{g,d} \to \Loc_{V_R}\]
    as in \Cref{dfn:loc sys and torsor on stack}.
    \item[(iii), (iv)] Similarly, the functor sending $f:A \to T$ to the $Q$-subsheaf of $\uIsom(\L(A), \Fcal_{V_R})$ consisting of those isomorphisms compatible with the symplectic structures, defines a $Q$-torsor on $\Ascr_{g,d}$
    \[\Lev:\Ascr_{g,d} \to \BB Q\]
    as in \Cref{dfn:loc sys and torsor on stack}.
\end{enumerate}

This data specifies a $Q$-moduli datum as in \Cref{dfn:Q moduli datum}. We may replace every occurence of $\Ascr_{g,d}$ with $\overline{\Ascr}_{g,\delta}$ so as to find a second $Q$-moduli datum. We will now use \Cref{dfn:Q moduli stack}. Set $K_0 = \Gcal_{\delta}(\widehat{\Z})$, so that $K_0^p = Q$ is a compact subgroup\footnote{We emphasize on this tautology to warn the reader that $K_0^p$ (resp. $K^p$) here is what plays the role of $K_0$ (resp. $K$) in \Cref{dfn:Q moduli stack}.} of $Q$, and $K_p \coloneq K_{0,p}$ is the stabilizer of $V_{\Zp}$ in $G(\Qp)$. For any compact open subgroup $K^p \subseteq K_0^p$, set $K = K_{p}K^p$. 

\begin{dfn}\label{dfn:AgK}
    We write $\Ascr_{g,d,K}$ (resp. $\overline{\Ascr}_{g,\delta,K}$) for the moduli stack over $\Z_{(p)}$ of level $K^p$ as in \Cref{dfn:Q moduli stack}. Its objects are pairs $(A,[\eta]_{K^p})$ where $A$ is an object of $\Ascr_{g,d,K}$ (resp. $\overline{\Ascr}_{g,\delta}$), and $[\eta]_{K^p}$ is a global section of the quotient sheaf $\Lev(A) / \Fcal_{K^p}$. Additionally, we set
    \[\Ascr_{g,d,K_p} \coloneq \varprojlim_{K^p \in \Ccal(K_0^p)} \Ascr_{g,d,K}\]
    (and similarly for $\overline{\Ascr}_{g,\delta,K_p}$).
    Its objects are pairs $(A,[\eta])$ where $A$ is an object of $\Ascr_{g,d,K}$ (resp.$\overline{\Ascr}_{g,\delta}$), and $[\eta]$ is a global section of $\Lev(A)$. In particular, the $\widehat{\Z}^p$-Tate module of $A$ has trivial monodromy.
\end{dfn}

The natural morphism $\overline{\Ascr}_{g,\delta,K} \hookrightarrow \Ascr_{g,d,K}$ is a closed embedding. For $K^p$ sufficiently small $\Ascr_{g,d,K}$ is represented by a quasi-projective $\Z_{(p)}$-scheme (see \cite[Theorem 7.9]{GeometricInvariantTheory}), hence so is $\overline{\Ascr}_{g,\delta,K}$. The stack $\Ascr_{g,d,K}$ coincides with Kisin's construction in \cite[§2.3.3]{kisin2010integral}.

\medskip

Sppose $\Ocal_F$ is a DVR that is a faithfully flat $\Z_{(p)}$-algebra.
Our goal for the end of this paragraph is to show the following.

\begin{fact}\label{fact:Ag has healthy extension}
    The stack $\overline{\Ascr}_{g,\delta,K_p,\Ocal_F} / \Ocal_F$ has healthy-extension as in \Cref{dfn:healthy extension}.
\end{fact}

Before proving \Cref{fact:Ag has healthy extension} we introduce one more extension property.

\begin{dfn}\label{dfn:DVR extension}
    Let $\Ocal$ be a DVR and let $\Sscr/\Ocal_F$ be a separated Deligne-Mumford stack. We say that it has \emph{DVR-extension}\footnote{This is Moonen's \emph{extended extension property}, see \cite[Remark 3.20]{moonen1998models}.} if every generic morphism $f:\Spec \, \Ocal' \rightsquigarrow \Sscr$ where $\Ocal \to \Ocal'$ is a faithfully flat extension of DVR, automatically and uniquely extends to an $\Ocal$-morphism $\Spec \, \Ocal' \to \Sscr$.
\end{dfn}

It is clear that if $\Sscr / \Ocal$ has DVR-extension, then so does $\Sscr_{\Ocal'}$ for any $\Ocal \to \Ocal'$ faithfully flat extension of DVR.

\begin{lem}\label{lem:Ag has DVR extension}
    The stack $\overline{\Ascr}_{g,\delta,K_p} / \Z_{(p)}$ has DVR-extension.
\end{lem}
\begin{proof}
    Let $\Z_{(p)} \to \Ocal'$ be a faithfully flat extension of DVR, and $\Spec \, \Ocal' \rightsquigarrow \Sscr$ a generic morphism. Let $A/\mrm{Frac}(\Ocal')$ the abelian variety corresponding to it. By the Serre-Tate criterion, it suffices to show that its $\widehat{\Z}^p$-Tate module is unramified as a Galois-module. In fact, the moduli interpretation of $\overline{\Ascr}_{g,\delta,K_p}$ (see \Cref{dfn:AgK}) ensures that the Galois action is trivial.
\end{proof}

\begin{proof}[Proof of \Cref{fact:Ag has healthy extension}]
    The proof is adapted from \cite[Theorem 2.10]{milne1992points}. \Cref{fact:polarization extends} ensures that we can ignore polarizations and focus solely on the abelian schemes.

    Let $\Xcal$ be a regular scheme over $\Ocal_F$, locally healthy, together with a generic morphism
    \[f: \Xcal \rightsquigarrow \overline{\Ascr}_{g,\delta,K_p,\Ocal_F}.\]
    Thus $f$ corresponds to a polarized abelian scheme $A_{\Xcal_{\eta}} \to \Xcal_{\eta}$. It suffices to show that it extends to an abelian scheme over some open subscheme $U$ containing $\Xcal_{\eta}$ and each generic point of the special fiber $\Xcal_s$. Let $\xi$ be such a generic point; since $\Xcal$ is regular we find that the local ring $\Ocal_{\Xcal,\xi}$ is a DVR. The result follows from \Cref{lem:Ag has DVR extension}.
\end{proof}

\begin{cor}\label{cor:Ag has smooth extension}
    Suppose $\Ocal_F$ has absolute ramification index $e \le p-1$. Then $\overline{\Ascr}_{g,\delta,K_p,\Ocal_F} / \Ocal_F$ has smooth-extension.
\end{cor}
\begin{proof}
    This follows from \Cref{fact:Ag has healthy extension} and \Cref{prop:smooth implies healthy}.
\end{proof}

\subsection{Smooth integral canonical models for Siegel and Hodge-type}\label{subsec:smooth ext for Siegel and Hodge}

We follow the construction from \cite[§2.3]{kisin2010integral}. Beginning with the Siegel case, let $\delta, (V, \psi), \Gcal_{\delta}, G$ and $K_0 = \Gcal_{\delta}(\widehat{\Z})$ as in \Cref{subsec:moduli of polar ab}. Let $\Omega$ be the space of complex structures on $V_{\R}$ for which $(V_{\R},\psi_{\R})$ is a polarized Hodge structure. Then $(G,\Omega)$ is a Siegel Shimura datum with reflex field $\Q$, and for any compact open subgroup $K^p \subseteq K_0^p$ we let $\Sh_K[G,\Omega]$ denote the associated Shimura stack of level $K$.

We define an integral model $\Sscr_{K_p}(G,\Omega)$ of $\Sh_{K_p}[G,\Omega]$ over $\Z_{(p)}$ with smooth-extension. It should be noted that $\Sscr_{K_p}(G,\Omega)$ implicitely\footnote{More accurately, the dependence lies in $K_p = \Gcal_{\delta}(\Zp)$.} depends on $\delta$ and will not be smooth in general (and so will not be a smooth integral canonical model as in \Cref{dfn:smooth int model}) unless $\delta$ is prime to $p$. 

\begin{fact}
    For $K^p$ sufficiently small so that $\overline{\Ascr}_{g,\delta,K}$ is represented by a quasi-projective scheme, we can identify
    \[\Sh_{K}(G,\Omega) = \overline{\Ascr}_{g,\delta,K,\Q}\]
    and similarly replacing $K$ with $K_p$.
\end{fact}
\begin{proof}[Sketch of proof]
    From the moduli interpretation of $\Sh_{K}(G,\Omega)$ provided in \cite[§4]{deligneTravaux} there is an embedding $\Sh_{K}(G,\Omega) \hookrightarrow \Ascr_{g,d,K,\Q}$ (see \cite[§2.3.3]{kisin2010integral}), which we claim induces the identification. It suffices to prove this over $\C$, where we find that $\Sh_{K}(G,\Omega)_{\C}$ and $\overline{\Ascr}_{g,\delta,K,\C} = \Ascr_{g,\delta,K,\C}$ represent the same functor.
\end{proof}

\begin{dfn}\label{dfn:Siegel integral model}
    For $K^p$ sufficiently small so that $\overline{\Ascr}_{g,\delta,K}$ is represented by a quasi-projective scheme, let $\Sscr_{K}(G,\Omega) / \Z_{(p)}$ (resp. $\Sscr_{K_p}(G,\Omega) / \Z_{(p)}$) denote the closure of $\Sh_{K}(G,\Omega)$ (resp. $\Sh_{K_p}(G,\Omega)$) in $\overline{\Ascr}_{g,\delta,K}$ (resp. $\overline{\Ascr}_{g,\delta,K_p}$).
\end{dfn}

We extend this definition to all $K^p \subseteq K_0^p$ as we did before, by setting $\Sscr_{K}[G,\Omega] \coloneq [\Sscr_{K_p}(G,\Omega) / K^p]$.

\begin{prop}\label{prop:Siegel has smooth-extension}
    Let $\Ocal_F$ denote a discrete valuation ring that is a faithfully flat $\Z_{(p)}$-algebra. Then $\Sscr_{K_p}(G,\Omega)_{\Ocal_F} / \Ocal_F$ has DVR-extension and healthy-extension.

    \noindent If $\Ocal_F$ has absolute ramification index $e \le p-1$, then $\Sscr_{K_p}(G,\Omega)_{\Ocal_F} / \Ocal_F$ has smooth-extension.
\end{prop}
\begin{proof}
    DVR-extension (resp. healthy-extension) follows from \Cref{lem:Ag has DVR extension} (resp. \Cref{fact:Ag has healthy extension}) and the DVR-extension counterpart (resp. healthy-extension counterpart) of \Cref{lem:normalized closed substack has smooth extension} (see \Cref{rk:closed substack has extension}). The second statement follows from \Cref{prop:smooth implies healthy}.
\end{proof}

\begin{cor}\label{cor:Siegel smooth int model}
    Let $\Ocal_F$ be as in \Cref{prop:Siegel has smooth-extension} with absolute ramification index $e \le p-1$. If $p \nmid d$ then $K_p$ is hyperspecial and $\Sscr_{K_p}(G,\Omega)_{\Ocal_F} = \overline{\Ascr}_{g,\delta,K_p,\Ocal_F}$ is a smooth integral canonical model of $\Sh_{K_p}(G,\Omega)$ over $\Ocal_F$.
\end{cor}

\begin{rk}
    It follows from \Cref{prop:can equals univ} that the canonical local system on $\Sscr_{K}[G,\Omega]$ coincides with the universal Tate module.
\end{rk}

We move on to Hodge-type Shimura varieties. Let $(G_1,\Omega_1)$ be a Shimura datum of Hodge-type, meaning that there is an embedding
\[(G_1,\Omega_1) \hookrightarrow (G,\Omega).\]
Let $E$ denote the reflex field of $(G_1,\Omega_1)$, fix $v$ a place of $E$ dividing $p$. Fix $K_{1,p} \subseteq G_1(\Qp)$ a hyperspecial level. We sketch the proof of the following.

\begin{thm}[\cite{kisin2010integral}, Theorem 2.3.8]\label{thm:hodge has smooth extension}
    Suppose\footnote{The statement holds true for $p=2$ under additional assumptions. We do not need this.} $p>2$.
    \begin{enumerate}
        \item[(1)] The scheme $\Sh_{K_{1,p}}(G_1,\Omega_1)$ admits a smooth integral canonical model $\Sscr_{K_{1,p}}(G_1,\Omega_1)$ over $\Ocal_{(v)}$.
        \item[(2)] Moreover, for any discrete valuation ring $\Ocal_F$ that is a faithfully flat $\Ocal_{(v)}$-algebra of absolute ramification index $e \le p-1$, $\Sscr_{K_{1,p}}(G_1,\Omega_1)_{\Ocal_F}$ is again a smooth integral canonical model of $\Sh_{K_{1,p}}(G_1,\Omega_1)$ over $\Ocal_F$.
    \end{enumerate}
\end{thm}

\begin{proof}[Sketch of proof]

Statement (1) is \cite[Theorem 2.3.8]{kisin2010integral}. Statement (2) follows from the proof of \emph{loc. cit.} and \Cref{prop:Siegel has smooth-extension}.

More precisely, for $K_{1}^p$ a small enough compact open subgroup of $G_1(\A_f^p)$, it is shown in \cite[§2.3.2]{kisin2010integral} that we can find $\delta$ and $(V,\psi)$ as above (note that $\delta$ may not be prime to $p$) and a subgroup 
\[K^p \subseteq K_0^p = \Gcal_{\delta}(\widehat{\Z}^p) \subseteq K_0 = K_p K_0^p = \Gcal_{\delta}(\widehat{\Z})\]
such that the embedding of Shimura data induces an inclusion
\[\iota:\Sh_{K_{1}}(G_1,\Omega_1) \hookrightarrow \Sh_{K}(G,\Omega)_E.\]

\noindent (The right-hand side depends on $\delta$ via $K = K_p K^p$.) Then one defines $\Sscr_{K_{1}}(G_1,\Omega_1)$ as the closure\footnote{Kisin additionally takes the normalization of this closure in order to obtain a smooth $\Ocal_F$-scheme. It was shown in \cite{xuNormalizationIntegralModels2026} that the normalization is an isomorphism, so that $\Sscr_{K_{1}}(G_1,\Omega_1)$ is already smooth.} of the image of $\iota$ in in the model $\Sscr_{K}(G,\Omega)$ from \Cref{dfn:Siegel integral model}. Kisin then considers
\[\Sscr_{K_{1,p}}(G_1,\Omega_1) \coloneq \varprojlim_{K_{1}^p} \Sscr_{K_{1}}(G_1,\Omega_1).\]

\noindent One argues that the pairs $(K_{1}^p,K^p)$ can be chosen coherently so that there are induced closed embedding
\[\Sscr_{K_{1,p}}(G_1,\Omega_1) \hookrightarrow \Sscr_{K_p}(G,\Omega)_{\Ocal_{(v)}}.\]

\noindent Base-changing to $\Ocal_F$, it follows from \Cref{lem:normalized closed substack has smooth extension} and \Cref{prop:Siegel has smooth-extension} that $\Sscr_{K_{1,p}}(G_1,\Omega_1)_{\Ocal_F}$ has smooth extension.
\end{proof}

From the construction of $\Sscr_{K_{1,p}}(G_1,\Omega_1)$ we find:
\begin{cor}\label{cor:Hodge has DVR}
    The scheme $\Sscr_{K_{1,p}}(G_1,\Omega_1) / \Ocal_{(v)}$ has DVR-extension.
\end{cor}
\begin{proof}
    Indeed, $\Sscr_{K_{1,p}}(G_1,\Omega_1)$ is a closed subscheme of $\Sscr_{K_p}(G,\Omega)_{\Ocal_{(v)}}$ which has DVR-extension by \Cref{prop:Siegel has smooth-extension}. Hence the result by \Cref{lem:normalized closed substack has smooth extension} and \Cref{rk:closed substack has extension}.
\end{proof}

\subsection{Smooth integral canonical models for abelian-type}\label{subsec:smooth ext for abelian type}

We sketch the construction from \cite[§3.4]{kisin2010integral}. Let $(G_2,\Omega_2)$ be a Shimura datum of abelian type, meaning that there exists $(G_1,\Omega_1)$ a Shimura datum of Hodge type, and a central isogeny
\[G_1^{\mrm{der}} \to G_2^{\mrm{der}}\]
inducing an isomorphism $(G_1^{\ad},\Omega_1^{\ad}) \simeq (G_2^{\ad},\Omega_2^{\ad})$ (notations from \cite[§3.3.6]{kisin2010integral}). Let $E_i$ (resp. $E^{\ad}$) denote the reflex field of $(G_i,\Omega_i)$ (resp. $(G_i^{\ad},X^{\ad}_i)$) so that $E_i \subseteq E^{\ad}$, fix $v^{\ad} | v_i | p$ places of $E^{\ad}$ and $E_i$. Consider $K_{2,p} \subseteq G_2(\Qp)$ hyperspecial. We wish to show:

\begin{thm}[\cite{kisin2010integral}, Theorem 3.4.14]\label{thm:abelian has smooth extension}
    Suppose\footnote{The statement holds true for $p=2$ under additional assumptions. We do not need this.} $p>2$.
    \begin{enumerate}
        \item[(1)] The scheme $\Sh_{K_{2,p}}(G_2,\Omega_2)$ admits a smooth integral canonical model $\Sscr_{K_{2,p}}(G_2,\Omega_2)$ over $\Ocal_{E_2,(v_2)}$.
        \item[(2)] Moreover, for any discrete valuation ring $\Ocal_F$ that is a faithfully flat $\Ocal_{E_2,(v_2)}$-algebra of absolute ramification index $e \le p-1$, $\Sscr_{K_{2,p}}(G_2,\Omega_2)_{\Ocal_F}$ is again a smooth integral canonical model of $\Sh_{K_{2,p}}(G_2,\Omega_2)$ over $\Ocal_F$.
    \end{enumerate}
\end{thm}

\begin{proof}[Sketch of proof]

Statement (1) is \cite[Theorem 3.4.14]{kisin2010integral}. Statement (2) follows from the proof of \emph{loc. cit.} and \Cref{thm:hodge has smooth extension}. More precisely, we sketch the proof in five steps.

\medskip

\textbf{Step 1.} Kisin shows (see the proof of Corollary 3.4.14 of \textit{loc. cit.}) that it is always possible to choose $(G_1,\Omega_1)$ as above having a hyperspecial level $K_{1,p} \subseteq G_1(\Qp)$; in particular $G_1$ and $G_2$ are the special fibers of reductive groups $\Gcal_1$ and $\Gcal_2$ over $\Z_{(p)}$. It follows that the reflex fields $E_i$ and $E^{\ad}$ are all unramified at $p$. If we let $\tilde{E}/E^{\ad}$ be the maximal extension unramified at $v^{\ad}$ and $w$ a place of it dividing $v^{\ad}$, then $E_2 \subseteq \tilde{E}$. In this setting\footnote{In §3.4.9 of \textit{loc. cit.} Kisin assumes that $E_2 \subseteq E_1$, and defines instead $\tilde{E}/E_1$ to be the maximal extension unramified at $v_1$. It is not clear to us why this assumption holds, so we deviate from his argument.}, Kisin constructs a smooth integral model of $\Sh_{K_{2,p}}(G_2,\Omega_2)$ over the DVR $\Ocal_{\tilde{E},(w)}$.

\medskip

\textbf{Step 2.} To construct this model, let $\Sscr_{K_{1,p}}(G_1,\Omega_1)/\Ocal_{E_1,(v_1)}$ be the smooth integral canonical model of $\Sh_{K_{1,p}}(G_1,\Omega_1)/E_1$ previously constructed. Then it follows from Proposition 2.2.4 of \textit{loc. cit.} that the connected components of the base-change $\Sscr_{K_{1,p}}(G_1,\Omega_1)_{\Ocal_{\tilde{E},(w)}}$ are geometrically connected. Let $\Sscr_{K_{1,p}}(G_1,\Omega_1)^+ / \Ocal_{\tilde{E},(w)}$ denote one such component. As $\Ocal_{\tilde{E},(w)}$ is unramified at $p$, $\Sscr_{K_{1,p}}(G_1,\Omega_1)_{\Ocal_{\tilde{E},(w)}}$ has smooth extension, and so does its closed subscheme $\Sscr_{K_{1,p}}(G_1,\Omega_1)^+$ by \Cref{lem:normalized closed substack has smooth extension}. 
Then Kisin defines a scheme as follows (see (3.4.11) of \textit{loc. cit.}):
\begin{equation}\label{eq:abelian type integral model}
    \Sscr_{K_{2,p}}(G_2,\Omega_2)_{\Ocal_{\tilde{E},(w)}} \coloneq [\Ascr_2 \times \Sscr_{K_{1,p}}(G_1,\Omega_1)^+] / \Ascr_1.
\end{equation}
Let us explain the meaning of \Cref{eq:abelian type integral model} (see also §3.3.5 of \textit{loc. cit.}):
\begin{itemize}
    \item $\Ascr_1$ and $\Ascr_2$ are locally compact, totally disconnected groups (defined in §3.3.2 of \textit{loc. cit.}, we will not use their description) with a proper group homomorphism $\phi:\Ascr_1 \to \Ascr_2$ between them having finite kernel (see §3.3.9 of \textit{loc. cit.}),
    \item The group $\Ascr_1$ acts continuously on the scheme $\Sscr_{K_{1,p}}(G_1,\Omega_1)^+$ (recall \Cref{rk:continuous action}); let us write
        \[\Sscr_{K_{1,p}}(G_1,\Omega_1)^+ = \varprojlim_{H_1} T_{H_1}\]
    where $H_1$ ranges over the compact open subgroups of $\Ascr_1$.
\end{itemize}
Then the right-hand side of \Cref{eq:abelian type integral model} is to be understood as the scheme with continuous $\Ascr_2$-action obtained as the projective limit $\varprojlim_{H} T'_{H}$ of the schemes

\[T'_{H} \coloneq \coprod_{[g] \in \phi(\Ascr_1) \backslash \Ascr_2 / H} T_{\phi^{-1}(g H g^{-1})}\]
where $H$ is a compact open subgroup of $\Ascr_2$. By Proposition 3.3.10 of \textit{loc. cit.}, $\Sscr_{K_{2,p}}(G_2,\Omega_2)_{\Ocal_{\tilde{E},(w)}}$ is a model of $\Sh_{K_{2,p}}(G_2,\Omega_2)_{\tilde{E}}$.

\medskip

\textbf{Step 3.} We argue that $\Sscr_{K_{2,p}}(G_2,\Omega_2)_{\Ocal_{\tilde{E},(w)}}$ has smooth-extension; the strategy is inspired from Moonen's argument\footnote{Kisin claims that Moonen's result can be adapted to prove smooth-extension with no further details. We detail the argument for the convenience of the reader.} in the proof of the following: 
\begin{fact}[\cite{moonen1998models}, Proposition 3.21.4]\label{fact:moonen argument}
    Let $\Sscr / \Ocal_{\tilde{E},(w)}$ be a scheme with DVR-extension, and $\Delta$ a finite group acting freely on $\Sscr$. Then the quotient $\Sscr / \Delta$ has smooth and DVR-extensions over $\Ocal_{\tilde{E},(w)}$.
\end{fact}
The takeaway for Step 5 is that the argument below still holds true after base-change to any faithfully flat DVR extension $\Ocal / \Ocal_{\tilde{E},(w)}$.

Let $\Xcal / \Ocal_{\tilde{E},(w)}$ be as in \Cref{dfn:smooth extension} and consider $\Xcal \rightsquigarrow \Sscr_{K_{2,p}}(G_2,\Omega_2)_{\Ocal_{\tilde{E},(w)}}$ a generic morphism. Fix some compact open subgroup $H_0 \subseteq \Ascr_2$ and some $[g_0] \in \phi(\Ascr_1) \backslash \Ascr_2 / H_0$; we may and will assume that the image of
\[\Xcal \rightsquigarrow \Sscr_{K_{2,p}}(G_2,\Omega_2)_{\Ocal_{\tilde{E},(w)}} \to T'_{H_0}\]
lies in $T_{\phi^{-1}(g_0 H_0 g_0^{-1})}$. Let $\Sscr_2$ denote the inverse image of $T_{\phi^{-1}(g_0 H_0 g_0^{-1})}$ under the right-most arrow, so that we wish to extend the generic morphism $\Xcal \rightsquigarrow \Sscr_2$.

Let us first assume that $\Xcal = \Spec \, A$ is the spectrum of a DVR faithfully flat over $\Ocal_{\tilde{E},(w)}$. Then for every compact open subgroup $H \subseteq H_0$, the generic morphism $\Xcal \rightsquigarrow T'_H$ factors through some $T_{\phi^{-1}(gHg^{-1})}$, compatibly as $H$ varies, so that we find a generic morphism
\[\Xcal \rightsquigarrow \Sscr_{K_{1,p}}(G_1,\Omega_1)^+ / \ker \phi.\]
Now $\Sscr_{K_{1,p}}(G_1,\Omega_1)^+$ has DVR-extension (recall \Cref{dfn:DVR extension}) as a closed subscheme of an integral model of a Hodge-type Shimura variety (this follows from \Cref{cor:Hodge has DVR} and \Cref{rk:closed substack has extension}), hence so does $\Sscr_{K_{1,p}}(G_1,\Omega_1)^+ / \ker \phi$ by \Cref{fact:moonen argument}. Thus we find the desired extension
\[\Xcal \longrightarrow \Sscr_{K_{1,p}}(G_1,\Omega_1)^+ / \ker \phi \longrightarrow \Sscr_2.\]

For the geneal case, we no longer assume $\Xcal = \Spec \, A$. Consider the profinite set
\[C \coloneq \phi(\Ascr_1) \backslash \phi(\Ascr_1) g_0 H_0\]
and write $\tilde{\Sscr}_{\infty} \coloneq C \times_{\Ocal_{\tilde{E},(w)}} \Sscr_{K_{1,p}}(G_1,\Omega_1)^+$. There is a proétale morphism $\tilde{\Sscr}_{\infty} \to \Sscr_2$ and we can find schemes $\tilde{\Sscr}_j$ such that this morphism decomposes as a cofiltered limit of finite étale maps
\[\tilde{\Sscr}_j \longrightarrow \Sscr_2\]
and $\tilde{\Sscr}_{\infty} = \varprojlim_j \tilde{\Sscr}_j$ with finite étale transition maps. We define
\begin{align*}
    \Xcal_{j,\tilde{E}} &\coloneq \Xcal_{\tilde{E}} \times_{\Sscr_{2,\tilde{E}}} \tilde{\Sscr}_{j,\tilde{E}} \\
    \Xcal_{\infty,\tilde{E}} &\coloneq \Xcal_{\tilde{E}} \times_{\Sscr_{2,\tilde{E}}} \tilde{\Sscr}_{\infty}
\end{align*}
and let $\Xcal_j \to \Xcal$ (resp. $\Xcal_{\infty} \to \Xcal$) denote the relative normalization of $\Xcal_{j,\tilde{E}} \to \Xcal$ (resp. $\Xcal_{\infty,\tilde{E}} \to \Xcal$); note that $\Xcal_{\infty} = \varprojlim_j \Xcal_j$. There is a commutative diagram
\[
\begin{tikzcd}
\Xcal_{\infty} \arrow[d] \arrow[r, squiggly] & \tilde{\Sscr}_{\infty} \arrow[d]                 \\
\Xcal_j \arrow[d, "\rho_j"'] \arrow[r, squiggly]        & \tilde{\Sscr}_j \arrow[d, "\text{finite étale}"] \\
\Xcal \arrow[r, squiggly]                    & \Sscr_2                                         
\end{tikzcd}
\]
From the case above, we see that $\Xcal \rightsquigarrow \Sscr_2$ extends at the local ring of codimension 1 points. It follows that $\rho_j$ is étale above those points, hence that $\rho_j$ is finite étale by Zariski-Nagata purity, and so all $\Xcal_j$ are regular schemes, formally smooth over $\Ocal_{\tilde{E},(w)}$. Thus $\Xcal_{\infty} / \Ocal_{\tilde{E},(w)}$ is also formally smooth, and it is regular by \Cref{fact:Milne regular ring}. Now as $\Sscr_{K_{1,p}}(G_1,\Omega_1)^+$ has smooth-extension, so does $\tilde{\Sscr}_{\infty}$ by \Cref{lem:proétale lifting of smooth extension}, and so we find an extension
\[\Xcal_{\infty} \longrightarrow \tilde{\Sscr}_{\infty}\]
which descends down to $\Xcal \to \Sscr_2$ by proétale descent, as desired.

\medskip

\textbf{Step 4.}
We descend the integral model down to the valuation ring of the reflex field $E_2$. Indeed, smooth extension can then be leveraged to extend descent data from $\Sh_{K_{2,p}}(G_2,\Omega_2)_{\tilde{E}}$ to its model $\Sscr_{K_{2,p}}(G_2,\Omega_2)_{\Ocal_{\tilde{E},(w)}}$, which is then defined over $\Ocal_{E_2,(v_2)}$ as desired.

\medskip

\textbf{Step 5.} It remains to show that the base-change $\Sscr_{K_{2,p}}(G_2,\Omega_2)_{\Ocal_F}$ also has smooth extension. If $\Ocal_{\tilde{E},(w)}$ were a subring of $\Ocal_F$, it would be enough to base-change to $\Ocal_F$ in \Cref{eq:abelian type integral model}, state that $\Sscr_{K_{1,p}}(G_1,\Omega_1)_{\Ocal_F}^+$ has smooth and DVR-extensions, and run the argument of Step 3 once again. As $\Ocal_F$ may not contain $\Ocal_{\tilde{E},(w)}$ in general, we take a quick detour.

Let $\Ocal_F^{\mrm{sh}}$ denote the strict henselization of $\Ocal_F$, so that $\Ocal_{\tilde{E},(w)} \subseteq \Ocal_F^{\mrm{sh}}$. Thus we can write
\[\Sscr_{K_{2,p}}(G_2,\Omega_2)_{\Ocal_F^{\mrm{sh}}} = [\Ascr_2 \times \Sscr_{K_{1,p}}(G_1,\Omega_1)_{\Ocal_F^{\mrm{sh}}}^+] / \Ascr_1.\]
Since $\Ocal_F^{\mrm{sh}}$ is a DVR and a faithfully flat $\Ocal_{\tilde{E},(w)}$-algebra, it follows from \Cref{thm:hodge has smooth extension} that $\Sscr_{K_{1,p}}(G_1,\Omega_1)_{\Ocal_F^{\mrm{sh}}}^+$ has smooth and DVR-extension. As before, Step 3 shows that $\Sscr_{K_{2,p}}(G_2,\Omega_2)_{\Ocal_F^{\mrm{sh}}}$ has smooth extension. As regularity is preserved under base-change along $\Ocal_F \to \Ocal_F^{\mrm{sh}}$, we may apply \Cref{lem:descent of smooth extension} to deduce that $\Sscr_{K_{2,p}}(G_2,\Omega_2)_{\Ocal_F}$ has smooth extension.

\end{proof}

\subsection{Integral canonical model for certain orthogonal-type}\label{subsec:smooth ext for orthogonal type}

We sketch the construction from \cite{pera2016integral}. Let $(L,Q)$ be a quadratic lattice over $\Z$ of signature $(n,2)$, $n \ge 1$. Let $\Omega$ denote the space of oriented negative definite planes in $L_{\R}$, set $\Gcal_0 \coloneq \SO(L)$ and $G_0 = (\Gcal_0)_{\Q}$. Then $(G_0,\Omega)$ defines a Shimura datum with reflex field $\Q$. Let $K_L$ be the \emph{discriminant kernel}, i.e. the largest subgroup of $\Gcal_0(\widehat{\Z})$ acting trivially on the discriminant $\mrm{disc}(L) \coloneq L^{\vee} / L$, where $L^{\vee} \subseteq L_{\Q}$ denotes the dual lattice; write $K_{0,p}$ for the $p$-part of $K_L$.

Integral canonical models for $\Sh_{K_{0,p}}(G_0,\Omega)$ have been constructed by Madapusi Pera under some conditions on $L$.

\begin{dfn}
    We say that $L_{\Z_{(p)}}$ is \emph{maximal} if there is no larger $\Z_{(p)}$-lattice in $L_{\Q}$ on which the quadratic form is $\Z_{(p)}$-valued. Otherwise it is \emph{non-maximal}.

    We say that $L$ has \emph{cyclic discriminant} if $\mrm{disc}(L)$ is cyclic.
\end{dfn}

\begin{thm}[\cite{pera2016integral}, Theorem 7.4 (1), Proposition 7.9]\label{thm:orthogonal has extension}
    Assume $L_{\Z_{(p)}}$ is maximal (resp. is non-maximal with cyclic determinant). 
    \begin{enumerate}
        \item[(1)] The scheme $\Sh_{K_{0,p}}(G_0,\Omega)$ admits an integral canonical model (resp. a smooth integral canonical model) over $\Z_{(p)}$, denoted $\Sscr_{K_{0,p}}(G_0,\Omega)$.
        \item[(2)] Moreover, let $\Ocal_F$ be a discrete valuation ring that is a faithfully flat $\Z_{(p)}$-algebra of absolute ramification index $e \le p-1$. Assume that $\widehat{\Ocal_F^{\sh}}$, the completion of the strict henselization of $\Ocal_F$, is finite over $W(\overline{\F_p})$. 
        
        \noindent Then $\Sscr_{K_{0,p}}(G_0,\Omega)_{\Ocal_F}$ again has smooth-extension (resp. is again a smooth integral canonical model of $\Sh_{K_{0,p}}(G_0,\Omega)$ over $\Ocal_F$).
    \end{enumerate}    
\end{thm}

\begin{proof}[Sketch of proof]

Statement (1) is \cite[Theorem 7.4 (1), Proposition 7.9]{pera2016integral}. Statement (2) follows from the proof of \emph{loc. cit.} and \Cref{thm:hodge has smooth extension}. More precisely, we introduce two additional Shimura data, staying as close as possible to the notations from \textit{loc. cit.}\footnote{There are a total of four Shimura data introduced in \cite{pera2016integral}. We only mention three of them here: the remaining datum $(\tilde{G}_0, \tilde{\Omega})$ is used for the construction of $\Sscr_{K_{0,p}}(G_0,\Omega)$, which we omit.}. We argue in three steps.

\medskip

\textbf{Step 1.} Let $\Gcal = \GSpin(L)$ and $G = \Gcal_{\Q}$, so that $(G,\Omega)$ is a Shimura datum with reflex field $\Q$. We claim that it suffices to show that \Cref{thm:orthogonal has extension} holds for $\Sh_{K_p}(G,\Omega)$ (here $K_p = \Gcal(\Zp)$ is hyperspecial). Indeed, assuming this, let $\Sscr_{K_p}(G,\Omega)$ be the corresponding model. Then, Madapusi Pera defines a candidate integral model $\Sscr_{K_{0,p}}(G_0,\Omega)$ of which $\Sscr_{K_p}(G,\Omega)$ is a pro-étale cover. In particular, we have a pro-finite pro-étale cover
\[\Sscr_{K_p}(G,\Omega)_{\Ocal_F} \longrightarrow \Sscr_{K_{0,p}}(G_0,\Omega)_{\Ocal_F}\]
where $\Sscr_{K_{p}}(G,\Omega)_{\Ocal_F}$ has smooth-extension by assumption. Descending smooth-extension to the scheme $\Sscr_{K_{0,p}}(G_0,\Omega)_{\Ocal_F}$ is done using the strategy of \cite[Proposition 3.21.4]{moonen1998models} similarly to Step 3 in the proof of \Cref{thm:abelian has smooth extension}.

\medskip

\textbf{Step 2.} From now-on, the goal is to show that \Cref{thm:orthogonal has extension} holds for $\Sh_{K_p}(G,\Omega)$. In this step, we sketch the construction of the integral model $\Sscr_{K_p}(G,\Omega)$. To this end, one embeds $(L,Q)$ into an appropriate self-dual $\Z_{(p)}$-lattice $(\tilde{L},\tilde{Q})$. We define $\tilde{\Omega}$, $\tilde{\Gcal} = \GSpin(\tilde{L})$, $\tilde{G} = \tilde{\Gcal}_{\Q}$ and $\tilde{K}_p = \tilde{\Gcal}(\Zp)$ as above; one shows that $(\tilde{G}, \tilde{\Omega})$ is a Shimura datum of Hodge-type and reflex field $\Q$. In particular:

\begin{itemize}
    \item For each $\tilde{K} = \tilde{K}^p \tilde{K}_p \subseteq \tilde{G}(\A_f)$, $\Sh_K[\tilde{G},\tilde{\Omega}]$ is a moduli space for certain polarized abelian schemes with level structure. We let $(\tilde{A}_{\tilde{K}}^{\mrm{KS}}, \tilde{\lambda}^{\mrm{KS}}, [\tilde{\eta}^{\mrm{KS}}]_{\tilde{K}})$ be the universal \emph{Kuga-Satake} object over $\Sh_K[\tilde{G},\tilde{\Omega}]$.
    \item There is a smooth integral canonical model $\Sscr_{\tilde{K}_p}(\tilde{G},\tilde{\Omega})$, such that the base change to $\Ocal_F$ is again a smooth integral canonical model. By construction $\Sscr_{\tilde{K}}[\tilde{G},\tilde{\Omega}]$ is also a moduli space, we write again $(\tilde{A}_{\tilde{K}}^{\mrm{KS}}, \tilde{\lambda}^{\mrm{KS}}, [\tilde{\eta}^{\mrm{KS}}]_{\tilde{K}})$ for the universal object.
\end{itemize}

Fix  $\tilde{K}^p \subseteq \tilde{G}(\A_f^p)$ small enough, and $K^p \subseteq G(\A_f^p)$ a subgroup of $\tilde{K}^p$. One then defines a finite unramified morphism $Z_{K^p} \to \Sscr_{\tilde{K}}(\tilde{G},\tilde{\Omega})$ via the functor of points of $Z_{K^p}$. To describe it, Madapusi Pera introduces the notion of \emph{special endomorphisms} corresponding to a subsheaf\footnote{The notation $\tilde{L}(T)$, introduced by Madapusi Pera, is not to be interpreted as "the $T$-points of $\tilde{L}$". Rather, the notation suggests that the defining property of $\tilde{L}(T)$ is closely related to the self dual lattice $\tilde{L}$.}
\[[T \to \Sscr_{\tilde{K}}(\tilde{G},\tilde{\Omega})] \mapsto \tilde{L}(T)\]
of the sheaf $\u{\mrm{End}}(\tilde{A}_{\tilde{K}}^{\mrm{KS}})$ on $\Sscr_{\tilde{K}}(\tilde{G},\tilde{\Omega})$-schemes. The following lemma is key to establishing smooth-extension for $Z_{K^p}$.

\begin{lem}\label{lem:extension of special endomorphisms}
    Let $T$ be a regular scheme faithfully flat over $\Z_{(p)}$, and consider a morphism $g:T \to \Sscr_{\tilde{K}}(\tilde{G},\tilde{\Omega})$. Then any section in $\tilde{L}(T_{\Q})$ uniquely extends to a section in $\tilde{L}(T)$.
\end{lem}
\begin{proof}
    Let $A_T \to T$ denote the pullback of the Kuga-Satake abelian scheme via $g$. Let $f \in \tilde{L}(T_{\Q}) \subseteq \mrm{End}(A_{T_{\Q}})$, then $f$ extends to an endomorphism of $A_T$ (see \cite[8.4, Corollary 6]{boschNeronModels1990a}). Our assumptions on $T$ ensure that any generic point of $T_{\F_p}$ is the specialization of a point in $T_{\Q}$, so by \cite[Lemma 5.13]{pera2016integral} the extension of $f$ is again special. 
\end{proof}

Let $\Lambda \coloneq L^{\perp} \hookrightarrow \tilde{L}$ (here $\tilde{L}$ is the $\Z_{(p)}$-lattice, not the sheaf of special endomorphisms). To any $\Sscr_{\tilde{K}}(\tilde{G},\tilde{\Omega})$-scheme $T$, let $Z_{K^p}(T)$ denote the set of all $(A_T, \lambda_T , [\eta_T]_{\tilde{K}}, \iota,[\eta_{\iota}]_{K^p})$, where:
\begin{itemize}
    \item The triple $(A_T, \lambda_T ,[\eta_T]_{\tilde{K}^p})$ is the object associated with $T \to \Sscr_{\tilde{K}}(\tilde{G},\tilde{\Omega})$,
    \item The map $\iota: \u{\Lambda}_T \to \tilde{L}(T)$ is an isometric embedding,
    \item The object $[\eta_{\iota}]_{K^p}$ is a section over $T$ of a certain quotient sheaf $I_{\iota}^p / K^p$ (see Definition 6.11 of \textit{loc. cit.} for details), and
    \item There are compatibility conditions between $\iota$, $[\eta_T]_{\tilde{K}^p}$ and $[\eta_{\iota}]_{K^p}$ which we omit for simplicity. In particular, $[\eta_{\iota}]_{K^p}$ determines $[\eta_T]_{\tilde{K}^p}$, so that changing $\tilde{K}^p$ while leaving $K^p$ fixed does not change the set $Z_{K^p}(T)$.
\end{itemize}
It is shown in 6.13 of \textit{loc. cit.} that $Z_{K^p}$ is represented by a scheme finite and unramified over $\Sscr_{\tilde{K}}(\tilde{G},\tilde{\Omega})$. Next, Madapusi Pera argues that there is a canonical morphism
\[\Sh_{K}(G,\Omega) \to Z_{K^p,\Q}\]
which is an isomorphism onto an open and closed subscheme of the target (see Lemma 7.1 of \textit{loc. cit.}). Before taking the Zariski closure, we first replace $Z_{K^p}$ with an open subscheme $Z_{K^p}^{\mrm{pr}}$ which contains the generic fiber $Z_{K^p,\Q}$ (see Lemma 6.16 of \textit{loc. cit.} for details). We define $\Sscr_{K}(G,\Omega)$ as the Zariski closure of the image of $\Sh_{K}(G,\Omega)$ in $Z_{K^p}^{\mrm{pr}}$. Finally, we set
\[\Sscr_{K_p}(G,\Omega) \coloneq \varprojlim_{K^p \subseteq \tilde{K}^p} \Sscr_{K^p K_p}(G,\Omega).\]

\textbf{Step 3.} It remains to show that $\Sscr_{K_p}(G,\Omega)_{\Ocal_F}$ has smooth extension. It follows from the functor of points description of $Z_{K^p}$, from \Cref{thm:hodge has smooth extension} and from \Cref{lem:extension of special endomorphisms} that the projective limit
\[Z_{\Ocal_F} \coloneq \varprojlim_{K^p} Z_{K^p,\Ocal_F}\]
has smooth-extension. Let $T$ be a regular scheme formally smooth over $\Ocal_F$, together with a generic morphism $T \rightsquigarrow \Sscr_{K_p}(G,\Omega)_{\Ocal_F}$. Then there is an extension
\[T \to Z_{\Ocal_F}\]
and our goal is to show that this extension factors through $Z_{\Ocal_F}^{\mrm{pr}} \coloneq \varprojlim_{K^p} Z_{K^p,\Ocal_F}^{\mrm{pr}}$.
We claim that $T$ can be expressed as a filtered limit
\[T = \varprojlim_i T_i\]
where each $T_i$ is smooth over $\Ocal_F$. Indeed, $T \to \Spec \, \Ocal_F$ is flat by \cite[Théorème 19.7.1]{EGAIV}, and each fiber is geometrically regular by \cite[\href{https://stacks.math.columbia.edu/tag/07EL}{Tag 07EL}]{stacks-project} and \cite[\href{https://stacks.math.columbia.edu/tag/07EC}{Tag 07EC}]{stacks-project}. The conclusion follows from Popescu's theorem, see \cite[\href{https://stacks.math.columbia.edu/tag/07GC}{Tag 07GC}]{stacks-project}. We are now reduced to the following statement, where we use the assumptions on $\widehat{\Ocal_F^{\sh}}$ in \Cref{thm:orthogonal has extension}.

\begin{lem}
    Let $T_i$ be a smooth $\Ocal_F$-scheme, consider a morphism 
    \[f:T_i \to Z_{K^p,\Ocal_F}\]
    such that the generic fiber $f_F$ factors through $\Sh_{K^pK_p}(G,\Omega)_{\Ocal_F}$. Then $f$ factors through $Z_{K^p}^{\mrm{pr}}$.
\end{lem}
\begin{proof}[Sketch of proof]
    The original statement in \cite[Lemma 6.16 (4)]{pera2016integral} deals with the case $\Ocal_F = \Z_{(p)}$. The only point in the argument which needs some care after base-change to general $\Ocal_F$ is the use of the comparison isomorphism between crystalline and de Rham realization functors with $\Ocal_F$-coefficients (in the proofs of Lemma 6.14 and Lemma 6.16 (3) of \textit{loc. cit.}). The isomorphism holds (see \cite[Corollary 6.7, Theorem 7.1]{BerthelotOgusNotes}) because the assumption $e \le p-1$ ensures that $\Ocal_F \to k_F$ is a divided power thickening.

    Additionally, the reference to \cite[Lemma 6.14]{pera2016integral} in the proof is where our assumption on $\widehat{\Ocal_F^{\sh}}$ originates from.
\end{proof}

Thus the extension $T \to Z_{\Ocal_F}$ factors through $Z_{\Ocal_F}^{\mrm{pr}}$, hence through $\Sscr_{K_p}(G,\Omega)_{\Ocal_F}$. The proof of \Cref{thm:orthogonal has extension} is now complete.

\end{proof}

\section{Applications}\label{sec:applications}

Assume $p\neq 2$. Let $F/\Q_p$ be a finite field extension of ramification index $e$. In this section we prove \Cref{thm:main thm} by showing that, in each case, the relevant moduli space $\Mscr$ has PGR-extension. For curves we recall a result of Stix stating directly that $\Mscr$ has PGR-extension. For abelian schemes, K3 surfaces and cubic fourfolds, the proof is as follows. We argue that period morphisms induce étale morphisms $\tilde{\Mscr} \to \Sscr_K$ from a finite torsor $\tilde{\Mscr}$ over to (smooth) integral canonical models of Shimura varieties of finite level $K$. Then recall that the models at the infinite level $\Sscr_{K_p} = \varprojlim_{K^p} \Sscr_K$ have smooth extension even after base-change to $\Ocal_F$, as we argued in \Cref{sec:extension for Shimura}, and so each $\Sscr_{K,\Ocal_F}$ has L and PGR-extensions from \Cref{prop:smooth extension implies L extension} and \Cref{prop:L-extension implies PGR extension}. It then follows from \Cref{thm:PGR extension pulls back for stacks} that $\tilde{\Mscr}_{\Ocal_F}$ has PGR-extension, hence so has $\Mscr_{\Ocal_F}$ by \Cref{prop:PGR descends along finite torsors}. Summarizing:

\[\begin{tikzpicture}[node distance=1cm and 1cm]

    \node (U) {$\tilde{\Mscr}_{\Ocal_F}$};
    \node[right=of U] (S) {$\Sscr_{K,\Ocal_F}$};
    \node[above=of S] (A) {$\Sscr_{K_p,\Ocal_F}$};
    \node[below=of U] (M) {$\Mscr_{\Ocal_F}$};

    \draw[->] (A) -- node[right] {projection} (S);
    \draw[->] (U) -- node[above] {étale} (S);
    \draw[->] (U) -- node[left] {finite torsor} (M);

    \node[draw, right=0.5cm of A] {smooth-extension};
    \node[draw, right=0.5cm of S] {L-extension};
    \node[draw, left=0.5cm of U] {PGR-extension};
    \node[draw, left=0.5cm of M] {PGR-extension};

\end{tikzpicture}\]

\subsection{Smooth curves}

Let $g \ge 2$, $n \ge 0$ be integers satisfying $2-2g<n$. Let $\Mscr_{g,n}$ denote the separated Deligne-Mumford stack of genus-$g$, $n$-pointed, geometrically connected, smooth, proper curves. Let $\ell_1 \neq \ell_2$ be prime numbers different from $p$, set $\Lscr = \{\ell_1,\ell_2\}$, and $\Z_{\Lscr} = \Z_{\l_1} \times \Z_{\l_2}$. In \cite[Remark 2.8 (a)]{stix2005monodromy} Stix defines a $\Z_{\Lscr}$-local system\footnote{We warn the reader that in \cite{stix2005monodromy} the notation $\L$ refers to the set denoted $\Lscr$ here.} $\L$ over $\Mscr_{g,n} \times_{\Z} \Z\left[ \frac{1}{\Lscr}\right]$. 

Let $\Sscr \coloneq \Mscr_{g,n} \times_{\Z} \Ocal_F$, the pullback of $\L$ to $\Sscr$ is again denoted $\L$.

\begin{prop}
    The pair $(\Sscr,\L)$ has L-extension, hence $\Sscr$ has PGR-extension.
\end{prop}
\begin{proof}
    This is a corollary of \cite[Theorem 2.10 (B)]{stix2005monodromy} and \Cref{prop:L-extension implies PGR extension}.
\end{proof}

Consequently, we recover \cite[§6.2, Example (1)]{cadoret2025pointwise}.

\begin{cor}\label{thm: PGR GGR curves}
    Let $\Xscr/\Ocal_F$ be a smooth scheme with generic fiber $X/F$. Let $f:Y \to X$ be a smooth proper family of genus-$g$, $n$-pointed, geometrically connected, smooth, proper curves.
    If it has pointwise good reduction, then it has global good reduction.
\end{cor}

\subsection{Abelian varieties}

We recall the setup from \Cref{subsec:moduli of polar ab}. Namely, let $g \ge 1$ be an integer, $\delta = (\delta_1,\dots,\delta_g)$ a tuple of positive integers satisfying $\delta_1 | \dots | \delta_g$, and fix a symplectic $\Z$-module $(V,\psi)$ of rank $2g$ whose symplectic form in an appropriate basis is prescribed by the matrix
\[J_{\delta} \coloneq \begin{pmatrix}
    0 & \Delta_{\delta} \\
    -\Delta_{\delta} & 0.
\end{pmatrix}\]
where $\Delta_{\delta} = \mrm{diag}(\delta_1\dots,\delta_g)$. Consider the $\Z$-group scheme $\Gcal_{\delta} \coloneq \GSp(V,\psi)$, set $G \coloneq \Gcal_{\Q}$ its generic fiber. Let $K_0 = K_p K_0^p = \Gcal_{\delta}(\widehat{\Z})$, and fix $K = K_p K^p \subseteq K_0$ a level structure; note that $K_p$ is fixed and $K^p$ varies. The stacks $\overline{\Ascr}_{g,\delta}$, $\overline{\Ascr}_{g,\delta,K}$ and $\overline{\Ascr}_{g,\delta, K_p}$ have been defined in (the paragraph following) \Cref{dfn:Ag} and \Cref{dfn:AgK}; in particular for $K = K_0$, $\overline{\Ascr}_{g,\delta} = \overline{\Ascr}_{g,\delta,K_0}$.

\begin{prop}\label{prop:Ag has PGR}
    Assume $e \le p-1$. Then $\overline{\Ascr}_{g,\delta,K,\Ocal_F}$ has L-extension, hence PGR-extension.
\end{prop}
\begin{proof}
    The stack $\overline{\Ascr}_{g,\delta, K_p, \Ocal_F}$ has smooth extension by \Cref{thm:abelian has smooth extension}. By definition it is the projective limit of the $K_0$-fit system $(\overline{\Ascr}_{g,\delta,K,\Ocal_F})_K$; it follows from \Cref{prop:smooth extension implies L extension} that each $\overline{\Ascr}_{g,\delta,K,\Ocal_F}$ has L-extension, hence PGR-extension by \Cref{prop:L-extension implies PGR extension}.
\end{proof}

\begin{rk}\label{rk:optimality of e}
    The bound on $e$ is optimal. Indeed, otherwise \Cref{cor:counterexample PGR not GGR} shows that for $\delta = (1,\dots,1)$ or $(1,\dots,1,p)$, there are some values of $g$ such that $\overline{\Ascr}_{g,\delta,K,\Ocal_F}$ does not have PGR-extension for any level structure $K = K^p K_p$. This implies in particular that for $e \ge p$ and $\delta = (1,\dots,1)$, the base-change $\Sscr_{K_p}(G,\Omega)_{\Ocal_F}$ of the Siegel modular variety (see \Cref{subsec:smooth ext for Siegel and Hodge}) does not have smooth-extension, hence is no longer a smooth integral canonical model. This answers the question of \cite[Remark 3.2.12.1]{vasiu1999integral} negatively.
\end{rk}

As a corollary, we recover \cite[§6.2, Example (3)]{cadoret2025pointwise}. Our proof up to this point has been a reformulation and generalization of theirs; no new idea was introduced. 

\begin{cor}\label{thm:PGR GGR ab schemes}
    Assume $e \le p-1$. Let $X/F$ be a smooth variety, and $\Xcal/\Ocal_F$ a smooth model.
    Let $A \to X$ be a polarized abelian scheme with level structure $K$. 
    If it has pointwise good reduction, then it has global good reduction.
\end{cor}
\begin{proof}
    We may assume that $X$ is irreducible, in which case $A \to X$ corresponds to a morphism $X \to \Ascr_{g,F}$ factoring through $\overline{\Ascr}_{g,\delta,K,F}$ for some $\delta$. The result follows from \Cref{prop:Ag has PGR}.
\end{proof}

\subsection{K3 surfaces}
Fix $d \ge 1$ an integer such that $p \nmid 2d$. The K3 lattice is the rank 22 quadratic lattice $\Lambda = E_8^{\oplus 2} \oplus U^{\oplus 3}$, where $E_8$ is the $E_8$ lattice of rank 8 and $U$ is the hyperbolic lattice of rank 2. Let $e,f \in U\subseteq \Lambda$ be isotropic elements satisfying $[e,f]=1$ in some copy of the hyperbolic lattice, where $[\cdot,\cdot]$ is the bilinear symmetric form on $U$. Set $L_d = \langle e-df \rangle^{\perp} \subseteq \Lambda$ a sublattice of discriminant $2d$ and rank 21 and write $W \coloneq L_{d,\Q}$. We let $L_{d}^{\vee} \subseteq W$ denote the dual lattice. To this data is associated a Shimura stack of orthogonal type, as in \Cref{subsec:smooth ext for orthogonal type} (see also \cite[§3.1]{madapusi2015tate}).

\begin{dfn}\label{dfn:orthogonal Shimura}
    Let $G = \SO(W)$ and $\Omega$ the space of oriented negative planes in $L_{d,\R}$. Let $K_0=K_0^p K_p \subseteq G(\A_f)$ denote the largest subgroup of $\SO(L_d)(\widehat{\Z})$ acting trivially on $L_{d}^{\vee} / L_{d}$. For any level $K = K^p K_p \subseteq K_0$, let $\Sh_K[G,\Omega] / \Z_{(p)}$ be the Shimura stack of level $K$ associated to the Shimura datum $(G,\Omega)$ with reflex field $\Q$.
\end{dfn}

\begin{prop}\label{fact:int model of orthogonal shimura}
    Assume $e \le p-1$. The Shimura variety $\Sh_{K_p}(G,\Omega)$ admits a integral model $\Sscr_{K_p}(G,\Omega)$ over $\Q$ such that
    \begin{enumerate}
        \item If $p^2 \nmid d$, then $\Sscr_{K_p}(G,\Omega)$ is an integral canonical model over $\Z_{(p)}$ as in \Cref{dfn:int can model}, and $\Sscr_{K_p}(G,\Omega)_{\Ocal_F}$ has smooth extension.
        \item If $p^2 \mid d$, then $\Sscr_{K_p}(G,\Omega)$ is a smooth integral canonical model over $\Z_{(p)}$ as in \Cref{dfn:smooth int model}, and so is its base-change $\Sscr_{K_p}(G,\Omega)_{\Ocal_F}$.
    \end{enumerate}
    In any case, $\Sscr_{K_p}(G,\Omega)_{\Ocal_F}$ has smooth extension
\end{prop}
\begin{proof}
    This is \Cref{thm:orthogonal has extension}. From \cite[§4.1]{madapusi2015tate} we have that $L_d$ has cyclic determinant, and is maximal at $p$ if and only if $p^2 \nmid d$.
\end{proof}

We assume from now-on that $e \le p-1$ and that $\Sscr_{K_p}(G,\Omega)$ is as in \Cref{fact:int model of orthogonal shimura}. From \Cref{prop:int model has PGR} we find that the quotient stack $\Sscr_{K}[G,\Omega]_{\Ocal_F}$ has PGR-extension.

Next we introduce the moduli space of primitively polarized K3 surfaces with level structure, following \cite[§2.1, §2.10]{madapusi2015tate}.

\begin{dfn}\label{dfn:K3 moduli space}
    Let $\Mscr_{2d}^{\circ}$ be the moduli problem that assigns to every $\Z_{(p)}$-scheme $T$ the groupoid of tuples $(f:Y \to T, \xi)$ where
    \begin{itemize}
        \item $f:Y \to T$ is a K3 surface\footnote{That is, $Y$ is a proper and smooth algebraic space over $T$, with geometric fibers that are K3 surfaces in the usual sense.}, and
        \item $\xi \in \underline{\mrm{Pic}}(X/T)(T)$ is a primitive polarization of degree $2d$.
    \end{itemize}
    It is a separated Deligne-Mumford stack of finite type over $\Z_{(p)}$ (\cite{rizov2006moduli} Theorem 4.3.3). 
\end{dfn}

If $(f:Y \to T,\xi) \in \Mscr_{2d}^{\circ}(T)$, let $\mbf{H}_{\widehat{\Z}^p}^2(f) = \prod_{\l \neq p} \mbf{H}_{\l}^2(f)$ be the $\widehat{\Z}^p$-local system on $T$ defined as the product of the relative $\l$-adic cohomology sheaves $R^2f_* \Zl$. The polarization defines a relative Chern class $\mrm{ch}_{\widehat{\Z}^p}(\xi) \in \mbf{H}_{\widehat{\Z}^p}^2(f)(1)$, the sub-local-system orthogonal to this class is the relative primitive cohomology sheaf $\mbf{P}_{\widehat{\Z}^p}^2(f)(1)$. 
Let $V \coloneq L_d \otimes_{\Z} \widehat{\Z}^p$. With notations from \Cref{sec:can loc sys on schemes and stacks} there is a local system
\begin{align*}
    \L:\Mscr_{2d}^{\circ} \longrightarrow & \Loc_{V} \\
    (f:Y \to T,\xi) \longmapsto & \mbf{P}_{\widehat{\Z}^p}^2(f)(1).
\end{align*}
The cup-product induces a quadratic form on (primitive) cohomology. We define
\begin{align*}
    \Lev: \Mscr_{2d}^{\circ} \longrightarrow & \BB \SO(L_d)(\widehat{\Z}^p) \\
    (f:Y \to T,\xi) \longmapsto & [(U \to T) \mapsto \{\text{Isometries } \mbf{P}_{\widehat{\Z}^p}^2(f)(1)_U \xrightarrow{\sim} \Fcal_V\}].
\end{align*}

Let $R=\widehat{\Z}^p$ and $Q = \SO(L_d)(R)$, now $(\Mscr_{2d}^{\circ}, \L, \Lev, \iota)$ is a $Q$-moduli datum as in \Cref{dfn:Q moduli datum}, where $\iota$ is the obvious sheaf inclusion of $\Lev(f:Y \to T,\xi)$ inside $\uIsom(\mbf{P}_{\widehat{\Z}^p}^2(f)(1),\Fcal_V)$. For any open subgroup $K = K^pK_p \subseteq K_0$ we may consider the associated moduli stack of level $K$ denoted $\Mscr_{2d,K}^{\circ}$. As $K_0$ is a finite index open normal subgroup of $Q$, we find that the projection $\Mscr_{2d,K_0}^{\circ} \to \Mscr_{2d}^{\circ}$ is a $Q/K_0$-torsor.

Ideally, $\Sh_K[G,\Omega]$ would serve as a period domain for the cohomology of K3 surfaces, so as to induce a morphism $\Mscr_{2d,K, \C}^{\circ} \to \Sh_{K}[G,\Omega]_{\C}$ which would descend over $\Q$, and finally extend to $\Mscr_{2d,K}^{\circ} \to \Sscr_{K}[G,\Omega]$ by invoking some extension property. To make this argument work however, $\Mscr_{2d}^{\circ}$ needs to be replaced with a double cover.

\begin{dfn}
    Let $\tilde{\Mscr}_{2d}^{\circ} \to \Mscr_{2d}^{\circ}$ be the double finite étale cover parametrizing isometric trivializations $\det (\mbf{P}_{\Z_2}^2) \xrightarrow{\sim} \Fcal_{\det(L_d) \otimes_{\Z} \Z_2}$ (notation from \Cref{dfn:sheaf from top space}). Its objects are called \emph{oriented, primitively polarized K3-surfaces of degree $2d$}.
\end{dfn}

From \cite[Proposition 4.2, Corollary 4.4 and Proposition 4.7]{madapusi2015tate} there is a natural period map $\tilde{\Mscr}_{2d, \C}^{\circ} \to \Sh_{K_0}(G,\Omega)_{\C}$ which extends to a $\Z_{(p)}$-morphism $\tilde{\Mscr}_{2d}^{\circ} \to \Sscr_{K_0}(G,\Omega)$.

\begin{fact}[\cite{madapusi2015tate}, Proposition 4.8]\label{fact:étale morphism from moduli to shimura}
    The map $\tilde{\Mscr}_{2d}^{\circ} \to \Sscr_{K_0}[G,\Omega]$ is étale.
\end{fact}

Let $\tilde{\Mscr}_{2d,K}^{\circ}$ denote the pullback of $\Mscr_{2d,K}^{\circ} \to \Mscr_{2d}^{\circ}$ along $\tilde{\Mscr}_{2d}^{\circ} \to \Mscr_{2d}^{\circ}$. The morphism in \Cref{fact:étale morphism from moduli to shimura} lifts to an étale morphism at level $K = K^p K_p \subseteq K_0$:
\begin{equation}\label{eq:étale integral period map}
    \tilde{\Mscr}_{2d,K}^{\circ} \to \Sscr_{K}[G,\Omega].
\end{equation}

\begin{rk}
    It follows from \Cref{prop:can equals univ} that the pullback of $\L_{\can}$ on $\Sscr_{K}[G,\Omega]$ to $\tilde{\Mscr}_{2d,K}^{\circ}$ coincides with the universal relative primitive cohomology $\mbf{P}_{\widehat{\Z}^p}^2(1)$.
\end{rk}

\begin{cor}\label{cor:moduli space has PGR}
    Assume $e \le p-1$. The stacks $\tilde{\Mscr}_{2d,\Ocal_F}^{\circ}$, $\Mscr_{2d,\Ocal_F}^{\circ}$, $\tilde{\Mscr}_{2d,K,\Ocal_F}^{\circ}$ and $\Mscr_{2d,K,\Ocal_F}^{\circ}$ all have PGR-extension.
\end{cor}
\begin{proof}
    For $\tilde{\Mscr}_{2d,K,\Ocal_F}^{\circ}$ and $\tilde{\Mscr}_{2d,\Ocal_F}^{\circ}$ this follows from $\Sscr_{K}[G,\Omega]_{\Ocal_F}$ having PGR-extension, from the base-change to $\Ocal_F$ of the étale map in \Cref{fact:étale morphism from moduli to shimura} (and \Cref{eq:étale integral period map}), and from \Cref{thm:PGR extension pulls back for stacks}. As $\tilde{\Mscr}_{2d}^{\circ} \to \Mscr_{2d}^{\circ}$ is a $\Z/2\Z$-torsor, PGR-extension for $\Mscr_{2d,\Ocal_F}^{\circ}$ and $\Mscr_{2d,K,\Ocal_F}^{\circ}$ is then deduced from \Cref{prop:PGR descends along finite torsors}.
\end{proof}

Thus, (oriented) primitively polarized K3-surfaces with $K$-level structure having pointwise good reduction also have global good reduction. In the particular case of $\Mscr_{2d,\Ocal_F}^{\circ}$, we find:

\begin{cor}\label{thm:PGR GGR for K3}
    Assume $e \le p-1$. Let $X/F$ be a smooth variety, and $\Xcal/\Ocal_F$ a smooth model. Consider $Y \to X$ a primitively polarized K3-surface of degree $2d$. If it has pointwise good reduction, then it has global good reduction.
\end{cor}

\begin{rk}
    All statements of this paragraph still hold true if $\Mscr_{2d}^{\circ}$ is replaced with $\Mscr_{2d}$ the moduli stack defined as in \Cref{dfn:K3 moduli space}, except $\xi$ is now a primitive quasi-polarization (as in \cite[§2.1]{madapusi2015tate}).
\end{rk}

\begin{rk}
    In trying to generalize the argument from K3 surfaces to general hyperkähler varieties, a couple of steps are missing. One can attach a Shimura datum $(G^{\mrm{HK}},\Omega^{\mrm{HK}})$ to the quadratic lattice defined by the second cohomology group of a fixed hyperkähler variety as in \cite[§4.5]{bindt2021hyperkahler}, as well as introduce the moduli stack $\Mscr_{2d}$ (resp. $\tilde{\Mscr}_{2d}$) of degree-2$d$ polarized, (resp. oriented) hyperkähler varieties over $\Z_{(p)}$, as in \cite[Definition 3.3.1, Definition 4.3.3]{bindt2021hyperkahler}\footnote{The stack introduced in \cite{bindt2021hyperkahler} is the disjoint union over $d$ of the $\tilde{\Mscr}_{2d,\Q}$.}. By \Cref{thm:orthogonal has extension} the variety $\Sh_{K_p}(G^{\mrm{HK}},\Omega^{\mrm{HK}})/\Q$ has an integral canonical model $\Sscr_{K_p}(G^{\mrm{HK}},\Omega^{\mrm{HK}})/\Z_{(p)}$, where $K_p \subseteq G^{\mrm{HK}}(\A_p)$ suitably chosen. By \cite[Theorem 4.5.2]{bindt2021hyperkahler} there is an étale period map $\tilde{\Mscr}_{2d,K,\Q} \to \Sh_{K}(G^{\mrm{HK}},\Omega^{\mrm{HK}})$ over $\Q$. If $p \nmid d$ then $\Sscr_{K_p}(G^{\mrm{HK}},\Omega^{\mrm{HK}})$ has smooth-extension and $\tilde{\Mscr}_{2d,K_p}$ is regular, formally smooth over $\Z_{(p)}$ hence an extension 
    \[\tilde{\Mscr}_{2d,K} \to \Sscr_{K}(G^{\mrm{HK}},\Omega^{\mrm{HK}}).\]
    If $p \nmid d$ then $\tilde{\Mscr}_{2d,K_p}$ is no longer smooth; if it can be shown to be locally healthy then the extension above still exists as $\Sscr_{K_p}(G^{\mrm{HK}},\Omega^{\mrm{HK}})$ has healthy-extension. In any case, if this extension were shown to be étale, it would follow as in \Cref{cor:moduli space has PGR} that $\tilde{\Mscr}_{2d,K}$, $\Mscr_{2d,K}$ and $\Mscr_{2d}$ have PGR-extension, and so \Cref{thm:PGR GGR for K3} would extend to hyperkähler varieties.
\end{rk}

\subsection{Cubic fourfolds}

The argument for K3 surfaces applies \textit{mutatis mutandis} to cubic fourfolds; we give references to the main ingredients following \cite[§5.13]{madapusi2015tate}.

Let $M_0$ be the even rank 2 quadratic lattice with bilinear form represented by
$\begin{pmatrix}
    2 & 1 \\
    1 & 2
\end{pmatrix}$
and set
\[M = E_8^{\oplus 2} \oplus U^{\oplus 2} \oplus M_0\]
which is has signature $(20,2)$ and is maximal at all primes.

\begin{dfn}\label{dfn:orthogonal Shimura for cubic fourfolds}
    Let $G = \SO(V)$ and $\Omega$ the space of oriented negative planes in $M_{\R}$. Let $K_0=K_0^p K_p \subseteq G(\A_f)$ denote the largest subgroup of $\SO(M)(\widehat{\Z})$ acting trivially on $M^{\vee} / M$. For any level $K = K^p K_p \subseteq K_0$, let $\Sh_K[G,\Omega]$ be the Shimura stack of level $K$ associated to the Shimura datum $(G,\Omega)$, over its reflex field $\Q$.
\end{dfn}

From \Cref{thm:orthogonal has extension} we find:

\begin{fact}
    The scheme $\Sh_{K_p}(G,\Omega)$ admits an integral canonical model $\Sscr_{K_p}(G,\Omega)$ over $\Z_{(p)}$. If $e \le p-1$ then $\Sscr_{K_p}(G,\Omega)_{\Ocal_F}$ has smooth-extension.
\end{fact}

Consider $\Mscr$ the moduli stack of cubic fourfolds over $\Z_{(p)}$, and $\tilde{\Mscr}$ its double cover trivializing the determinant of primitive cohomology. For $K$ as in \Cref{dfn:orthogonal Shimura for cubic fourfolds}, we let $\Mscr_K$ (resp. $\tilde{\Mscr}_K$) be the moduli stack with $K$-level structure on degree-4 primitive cohomology which is a finite étale cover of $\Mscr$ (resp. $\tilde{\Mscr}$). In particular, $\Mscr_{K_0} \to \Mscr$ is a finite torsor.

From the modular description of $\Sh_{K_0}(G,\Omega)_{\C}$ (see \cite[Proposition 3.3]{madapusi2015tate}) we find a period map:
\[\tilde{\Mscr}_{\C} \to \Sh_{K_0}(G,\Omega)_{\C}\]
which descends over $\Q$ as in \cite[Corollary 4.4]{madapusi2015tate}. By smooth-extension we find a period map $\tilde{\Mscr} \to \Sscr_{K_0}[G,\Omega]$ which is étale by \cite[Theorem 5.14]{madapusi2015tate}. Thus for $K$ as in \Cref{dfn:orthogonal Shimura for cubic fourfolds} we find étale maps $\tilde{\Mscr}_K \to \Sscr_{K}(G,\Omega)$. It follows that $\tilde{\Mscr}_{K,\Ocal_F}$, hence also $\Mscr_{K,\Ocal_F}$ and $\Mscr_{\Ocal_F}$, have PGR-extension. This proves the pointwise to global good reduction for (oriented) cubic fourfolds with $K$-level structure, and in the particular case of $\Mscr$:

\begin{thm}\label{thm:PGR GGR cubic fourfolds}
    Assume $e \le p-1$. Let $X/F$ be a smooth variety, and $\Xcal/\Ocal_F$ a smooth model. Consider $Y \to X$ a smooth, proper family of cubic fourfolds. If it has pointwise good reduction, then it has global good reduction.
\end{thm}

\nocite{}
\bibliographystyle{smfalpha}
\bibliography{ref.bib}

\end{document}